\documentclass{tran-l}

\usepackage{geometry}
\usepackage{mathtools}
\usepackage{amssymb}
\usepackage{stmaryrd}
\usepackage{subcaption}
\usepackage{hyperref}
\usepackage[capitalise]{cleveref}
\usepackage{tikz-cd}
\usepackage{tikz}
\usepackage{comment}
\usetikzlibrary{shapes.geometric, positioning}
\usetikzlibrary{arrows}
\usetikzlibrary{decorations.pathmorphing}
\usetikzlibrary{decorations.markings}

\newcommand{\Pp}{\mathbb{P}}
\newcommand{\R}{\mathbb{R}}
\newcommand{\Z}{\mathbb{Z}}

\newcommand{\Cfrak}{\mathfrak{C}}
\newcommand{\Pfrak}{\mathfrak{P}}

\newcommand{\boldn}{\boldsymbol{n}}

\newcommand{\bx}{\boldsymbol{x}}

\newcommand{\Bcal}{\mathcal{B}}
\newcommand{\Ccal}{\mathcal{C}}
\newcommand{\Dcal}{\mathcal{D}}
\newcommand{\Ecal}{\mathcal{E}}
\newcommand{\Fcal}{\mathcal{F}}
\newcommand{\Hcal}{\mathcal{H}}
\newcommand{\Ical}{\mathcal{I}}

\newcommand{\Pcal}{\mathcal{P}}
\newcommand{\Scal}{\mathcal{S}}
\newcommand{\Tcal}{\mathcal{T}}
\newcommand{\Wcal}{\mathcal{W}}

\DeclareMathOperator{\add}{add}
\DeclareMathOperator{\spann}{span}
\DeclareMathOperator{\Hom}{Hom}
\DeclareMathOperator{\cone}{cone}
\DeclareMathOperator{\rank}{rk}
\DeclareMathOperator{\link}{lk}
\DeclareMathOperator{\codim}{codim}
\DeclareMathOperator{\Hasse}{Hasse}
\DeclareMathOperator{\mods}{mod}
\DeclareMathOperator{\proj}{proj}
\DeclareMathOperator{\Filt}{Filt}
\DeclareMathOperator{\Ext}{Ext}
\DeclareMathOperator{\Fac}{Fac}
\DeclareMathOperator{\Faq}{Faq}
\DeclareMathOperator{\Sub}{Sub}
\DeclareMathOperator{\End}{End}
\DeclareMathOperator{\brick}{brick}
\DeclareMathOperator{\coker}{coker}
\DeclareMathOperator{\tors}{tors}
\DeclareMathOperator{\torf}{torf}
\DeclareMathOperator{\TF}{TF}
\DeclareMathOperator{\starr}{star}
\DeclareMathOperator{\flatt}{flat}
\DeclareMathOperator{\shard}{shard}
\DeclareMathOperator{\Part}{Part}
\DeclareMathOperator{\APart}{APart}
\DeclareMathOperator{\strigid}{s\tau-rigid}
\DeclareMathOperator{\sttilt}{s\tau-tilt}
\DeclareMathOperator{\diag}{diag}
\DeclareMathOperator{\bdim}{\boldsymbol{\dim}}

\copyrightinfo{}{}

\newtheorem{thm}{Theorem}[section]
\newtheorem{lem}[thm]{Lemma}
\newtheorem{cor}[thm]{Corollary}
\newtheorem{prop}[thm]{Proposition}

\theoremstyle{definition}
\newtheorem{defn}[thm]{Definition}
\newtheorem{exmp}[thm]{Example}

\theoremstyle{remark}
\newtheorem{rmk}[thm]{Remark}

\numberwithin{equation}{section}

\begin{document}

\title{The category of a partitioned fan}


%    author one information
\author{Maximilian Kaipel}
\address{Abteilung Mathematik, Department Mathematik/Informatik der Universität
zu Köln, Weyertal 86-90, 50931 Cologne, Germany}
\curraddr{}
\email{mkaipel@uni-koeln.de}
\thanks{}

%    \subjclass is required.

\date{}

\dedicatory{}

%    Abstract is required.
\begin{abstract}
    In this paper, we introduce the notion of an \textit{admissible partition} of a simplicial polyhedral fan and define the \textit{category of a partitioned fan} as a generalisation of the $\tau$-cluster morphism category of a finite-dimensional algebra. This establishes a complete lattice of categories around the $\tau$-cluster morphism category, which is closely tied to the fan structure. We prove that the classifying spaces of these categories are cube complexes, which reduces the process of determining if they are $K(\pi,1)$ spaces to three sufficient conditions. We characterise when these conditions are satisfied for fans in $\R^2$ and prove that the first one, the existence of a certain faithful functor, is satisfied for hyperplane arrangements whose normal vectors lie in the positive orthant. As a consequence we obtain a new infinite class of algebras for which the $\tau$-cluster morphism category admits a faithful functor and for which the cube complexes are $K(\pi,1)$ spaces. In the final section we also offer a new algebraic proof of the relationship between an algebra and its $g$-vector fan.
\end{abstract}

\maketitle

\section{Introduction}
Polyhedral fans arise naturally in many areas of mathematics. In \textit{toric geometry}, fans serve as fundamental tools for defining toric varieties \cite{Fulton1993}. In commutative algebra, the \textit{Gröbner fan} is an invariant associated to an ideal in a commutative polynomial ring \cite{MoraRobbiano1988}. Polytopes give rise to \textit{normal fans} and \textit{face fans} \cite{ZieglerLectures} which play an important role in the theory of optimisation \cite{Schrijver}. In matroid theory, the \textit{Bergman fan} is a subfan of the normal fan of the matroid polytope \cite{Bergman1971}. Recently, it was shown that abelian categories define \textit{heart fans} \cite{BPPW2023}. However, our motivation comes from the \textit{$g$-fan} of a finite-dimensional algebra \cite{DIJ2019} which was recently used to construct the \textit{$\tau$-cluster morphism category} \cite{STTW2023}. \\

The $g$-fan embeds into the \textit{wall-and-chamber structure} which arises from stability conditions \cite{Asai2019WS, BST2019, King1994}. The wall-and-chamber structure is the support of the stability scattering diagram \cite{Bridgeland2017}. For hereditary algebras, the wall-and-chamber structure is equivalent to the \textit{semi-invariant picture} defined in \cite{IOTW2015}. This picture defines the \textit{picture group}, see \cite{IgusaTodorovWeyman2016}, which has close connections to cluster theory and in particular to maximal green sequences \cite{IgusaTodorovMGS2021}. To compute the cohomology ring of the picture group, the authors of \cite{IgusaTodorovWeyman2016} associate a finite CW-complex, called the \textit{picture space}, to every hereditary algebra. Later, the \textit{cluster morphism category} is introduced in \cite{IgusaTodorov2017} as a categorical analogue of the picture space in the following sense: The \textit{classifying space} of this category is homeomorphic to the picture space and for representation-finite hereditary algebras this is a $K(\pi,1)$ space with $\pi$ the picture group. \\

The definition of the cluster morphism category was extended to $\tau$-tilting finite algebras in \cite{BuanMarsh2018} and then as the \textit{$\tau$-cluster morphism category} to all finite-dimensional algebras in \cite{BuanHanson2017}. The difficulty in the original constructions of the $\tau$-cluster morphism category lies in showing that composition of morphisms is associative \cite{BuanMarsh2018}. Different proofs of associativity are given in \cite{BuanHanson2017} and \cite{Borve2022}. Together with the work of \cite{MST2023}, this motivated the authors of \cite{STTW2023} to take a geometric viewpoint and construct the $\tau$-cluster morphism category from the $g$-fan of a finite-dimensional algebra. In this construction, the cones of the $g$-fan correspond to support $\tau$-rigid pairs, which have a wide subcategory associated to them \cite{BST2019,DIRRT2017,Jasso2015}. To obtain the definition of the category in \cite{STTW2023}, cones are identified whenever their corresponding wide subcategories coincide. \\

In this paper, we generalise the $\tau$-cluster morphism category to the \textit{category of a partitioned fan} by allowing more general identifications of cones. To guarantee composition of morphisms in these categories is well-defined we introduce \textit{admissible partitions} of a fan, which reflect the geometric properties induced by the identification of cones by shared wide subcategories (cf. \cite[Cor. 3.7, Lem. 3.8]{STTW2023}). Hence, the category of a partitioned fan has as its objects the equivalence classes of cones of the fan under the identifications of the admissible partition. In \cref{prop:lattice} we show that this means the $\tau$-cluster morphism category sits in a lattice of categories associated to the fan, which is ordered by refinement of partitions. In a similar way, the lattice of categories of a hyperplane arrangement admits two distinguished partitions of its underlying fan which define the \textit{category of the flat-partition} and the \textit{category of the shard-partition}, see \cref{subsec:HA}. \\

Each category defines a classifying space, which is the geometric realisation of its chains of composable morphisms. We generalise results of \cite{HansonIgusa2021} and \cite{IgusaTodorov2017} and prove that the classifying spaces of the categories of a partitioned fan are cube complexes. Therefore, we can use the sufficient conditions developed in \cite{Gromov1987} to determine classifying spaces which are $K(\pi,1)$ spaces, a subclass of Eilenberg-Maclane spaces whose only non-trivial homotopy group is in degree 1. We follow the approach of \cite{Igusa2022}, where the author categorifies the ideas of \cite{Gromov1987} and defines a \textit{cubical category}, roughly speaking, as a category whose classifying space is a cube complex and translates the criteria of \cite{Gromov1987} into this setting.

\begin{thm} \textnormal{(\cref{thm:cubicalcat})}
    Let $\Sigma$ be a simplicial fan and $\Pfrak$ an admissible partition. Then the category $\Cfrak(\Sigma, \Pfrak)$ of the partitioned fan is a cubical category and its classifying space $\Bcal \Cfrak(\Sigma, \Pfrak)$ thus a cube complex.
\end{thm}

The following categorical conditions are sufficient for the classifying space of the category $\Cfrak(\Sigma, \Pfrak)$ to be a $K(\pi,1)$ space (see \cite[Prop. 3.4, Prop. 3.7]{Igusa2022} and \cref{prop:igusacriteria}):
\begin{enumerate}
    \item There exists a faithful functor from $\Cfrak(\Sigma, \Pfrak)$ to some group $G$, viewed as a groupoid with one object. 
    \item + (3) The second and third conditions are referred to as the pairwise compatibility property of first and last factors respectively. These are more technical conditions about the compatibility of certain morphisms in the category, see \cref{prop:igusacriteria}. 
\end{enumerate}

The pairwise compatibility property of first factors is satisfied by the $\tau$-cluster morphism category of any finite-dimensional algebra. The other two conditions are satisfied for the $\tau$-cluster morphism categories of special classes of algebras, like representation-finite hereditary algebras \cite{IgusaTodorov2017}, Nakayama(-like) algebras \cite{HansonIgusaPW2SMC, HansonIgusa2021} and certain gentle algebras \cite{HansonIgusaPW2SMC}. Moreover, the $\tau$-cluster morphism category of $\tau$-tilting finite algebras of rank at most 3 satisfies the pairwise compatibility of last factors \cite{BarnardHanson2022}, whereas some examples are known which do not satisfy this property (cf. \cite{BarnardHanson2022}, \cite{HansonIgusaPW2SMC}). We generalise the picture group of a finite-dimensional algebra \cite{IgusaTodorovWeyman2016} to a fan with a well-behaved partial order, called a \textit{weak fan poset}, on its maximal cones (see \cref{defn:picgroup}) and show the following:
 
\begin{thm} (\cref{thm:rank2kpi1}, \cref{cor:rank2KPG1}, simplified)
    Let $\Sigma$ be a simplicial fan in $\R^2$ with a weak fan poset $\Pcal$ and an admissible partition $\Pfrak$. Then there exists a faithful functor from $\Cfrak(\Sigma, \Pfrak)$ to its picture group.
    \begin{enumerate}
        \item If no set of three pairwise-compatible rank 1 morphisms exist, then $\Bcal \Cfrak(\Sigma, \Pfrak)$ is a $K(\pi,1)$ space.
        \item If additionally, all maximal cones are identified, then the picture group is isomorphic to the fundamental group.
    \end{enumerate}
\end{thm}

For example, whenever the $g$-fan of a finite-dimensional algebra is finite, it admits a canonical fan poset (see \cref{thm:picgroupgeneralisation}) and as a consequence the $\tau$-cluster morphism category is a $K(\pi,1)$ for its picture group in rank 2. We establish the following relationships in the lattice of categories of a fan ordered by refinement of partitions.

\begin{thm} \textnormal{(\cref{thm:latticeproperties}, simplified)}
    Let $(\Sigma, \Pcal)$ be a weak fan poset and $\Pfrak_1, \Pfrak_2$ be two admissible partitions of $\Sigma$ such that $\Pfrak_1$ is finer than $\Pfrak_2$. Then the following hold:
    \begin{enumerate}
        \item There exists a faithful surjective-on-objects functor $F: \Cfrak(\Sigma, \Pfrak_1) \to \Cfrak(\Sigma, \Pfrak_2)$.
        \item The classifying space $\Bcal \Cfrak(\Sigma, \Pfrak_2)$ is a quotient space of $\Bcal \Cfrak(\Sigma, \Pfrak_1)$.
        \item The picture group of $\Cfrak(\Sigma, \Pfrak_2)$ is a quotient of the picture group of $\Cfrak(\Sigma, \Pfrak_1)$.
    \end{enumerate}
\end{thm}

As a consequence if $\Cfrak(\Sigma, \Pfrak_2)$ admits a faithful group functor $G$ to its picture group, then so does the category of any finer partition $\Cfrak(\Sigma, \Pfrak_1)$, see \cref{cor:faithfulfunct}. We use this fact to show the following.

\begin{thm} (\cref{thm:mainthm}, \cref{cor:HAallfaithful})
    Let $\Hcal \subseteq \R^n$ be a simplicial hyperplane arrangement such that the normal vector to every hyperplane can be taken to lie in the positive orthant $\R_{\geq 0}^n$. Then $\Cfrak(\Hcal, \Pfrak)$ admits a faithful functor to some group for all admissible partitions $\Pfrak$.
\end{thm}

Hence it is sufficient to study the pairwise compatibility properties of first and last factors for showing that $\Bcal \Cfrak(\Sigma, \Pfrak)$ is a $K(\pi,1)$ space. In particular, the above implies the following for the $\tau$-cluster morphism category.

\begin{thm} (\cref{thm:algmainthm})
    Let $A$ be a finite-dimensional algebra such that the $g$-vector fan is a finite hyperplane arrangement. Then the $\tau$-cluster morphism category $\Cfrak(A)$ admits a faithful functor to its picture group viewed as a groupoid with one object.
\end{thm}

This is a new family of algebras for which the existence of such a faithful functor is established. Previously, the existence of a faithful functor was known only for the family of $K$-stone algebras, see \cite[Sec. 5.2]{HansonIgusaPW2SMC} and the finite or tame hereditary algebras, see \cite{IgusaTodorov2017, IgusaTodorov22}. In particular, these three families are all infinite and distinct but not disjoint. One large family of new examples are the generalised preprojective algebras coming from Cartan matrices of finite (Dynkin) type as introduced in \cite{GLS2017}, which by \cite{Murakami2022} define finite hyperplane arrangements. Another large family of algebras are the contraction algebras, as introduced by \cite{DonovanWemyss2016} which define finite hyperplane arrangements as shown in \cite{August2020a}. %\\
%Using the bijection \cite[Thm. 4.1, Cor, 4.8]{AIR2014} between maximal rigid objects in certain 2-Calabi-Yau categories and the $\tau$-rigid pairs of the endomorphism algebra of one of its cluster-tilting objects, we can apply the results of \cite{JorgensenYakimov} to relate the picture group(oid) with the Deligne groupoid \cite{Deligne71} and the Picard group \cite{MiyachiYekutieli, RouquierZimmerman}, see \cref{subsec:tauclustermorph} for more details. 
For rank 3, \cite[Thm. 1]{BarnardHanson2022} implies that we obtain a new family of $K(\pi,1)$ spaces.
\begin{cor}
    Let $A$ be such that the $g$-vector fan is a finite hyperplane arrangement in $\R^3$. Then the classifying space $\Bcal \Cfrak(A)$ of the $\tau$-cluster morphism category is a $K(\pi,1)$ space for the picture group. 
\end{cor}

The structure of this paper is as follows: In Section 2 we give a short background on polyhedral fans and simplicial complexes, in Section 3 we define admissible partitions of a fan and the category of a partitioned fan. In Section 4 we study the classifying space of the category in relation to the picture groups of the fan and show the existence of the faithful functor. In Section 5 the lattice structure of admissible partitions and thus of categories is described and in the final Section 6 we study the $\tau$-cluster morphism category and the $g$-fan algebraically.

\section{Background on fans and simplicial complexes} \label{sec:background}

A \textit{convex polyhedral cone $\sigma$} in $\R^n$ is a set of the form $\sigma = \{ \sum_{i=i}^s \lambda_i v_i \in \R^n : \lambda_i \geq 0\}$, where $v_1, \dots, v_s \in \R^n$. We also denote these non-negative linear combinations by $\sigma = \cone\{v_1,\dots, v_s\}$. For two cones $\sigma = \cone\{v_1, \dots, v_s\}$ and $\kappa = \cone\{w_1, \dots, w_t\} \subseteq \R^n$, we sometimes write $\cone\{\sigma, \kappa\} = \cone \{ v_1, \dots, v_s, w_1, \dots, w_t\}$. Note that $\{ 0\}$ is also regarded as a convex polyhedral cone. Unless otherwise specified, in this paper, a \textit{cone} $\sigma$ is a polyhedral cone with the following two properties:
\begin{itemize}
    \item We say $\sigma$ is \textit{strongly convex} if $\sigma \cap (-\sigma) = \{0\}$ holds.
    \item A cone $\sigma$ is \textit{simplicial} if the generating set is linearly independent (up to duplicate generators).
\end{itemize}

The \textit{dimension} $\dim(\sigma)$ of a cone is the dimension of the linear subspace $\spann \{ \sigma\}$ in $\R^n$. Denote by $\langle -,- \rangle$ the standard inner product in $\R^n$. A \textit{face} of a cone $\sigma$ is the intersection of $\sigma$ with a hyperplane $\{v \in \R^n: \langle u,v \rangle = 0\}$ for some $u \in \R^n$ satisfying $\langle u,w \rangle \geq 0$ for all $w \in \sigma$. If $\sigma = \cone\{v_1, \dots, v_s\} \subseteq \R^n$ is a simplicial cone, then a face of $\sigma$ is simply a cone generated by a proper subset of $\{ v_1, \dots, v_s\}$. 

\begin{defn} \label{defn:fan}
A \textit{fan} $\Sigma$ in $\R^n$ is a collection of cones in $\R^n$ satisfying the following:
\begin{enumerate}
    \item Each face of a cone in $\Sigma$ is also a cone contained in $\Sigma$. 
    \item The intersection of two cones in $\Sigma$ is a face of each of the two cones.
\end{enumerate}
\end{defn}

Denote by $\Sigma^i \subseteq \Sigma$ the subset of cones of dimension $i$. We call a fan $\Sigma$ in $\R^n$ \textit{finite} if it consists of a finite number of cones and \textit{complete} if $\bigcup_{\sigma \in \Sigma} \sigma = \R^n$. We may view fans in two different ways. Firstly, fans naturally have the structure of a poset $(\Sigma, \subseteq)$ ordered by inclusion, which we view as a category whose objects are cones and which has a unique morphism $\sigma \to \tau$ whenever $\sigma \subseteq \tau$. The \textit{rank} $\rank(\Sigma)$ of finite and complete fan in $\R^n$ is the maximal dimension of any maximal cone with respect to inclusion, hence $n$. We use the poset category of the fan as a basis for the construction of the category of the partitioned fan in \cref{defn:thecategory}. For a thorough introduction to fans in the context of toric geometry we refer to \cite{Fulton1993}. \\

On the other hand we can view fans as simplicial complexes, which are finite sets $\Delta^0$ together with a collection $\Delta$ of subsets of $\Delta^0$ such that if $X \in \Delta$ and $Y \subseteq X$, then $Y \in \Delta$. Elements $v \in \Delta^0$ such that $\{ v \} \in \Delta$ are called \textit{vertices} and subsets consisting of vertices are called \textit{faces} or more specifically \textit{k-simplices} if they consist of exactly $k+1$ vertices. A simplicial fan $\Sigma$ in $\R^n$ defines a simplicial complex $\Delta(\Sigma)$ whose vertices are the dimension 1 cones $\Sigma^1$ and whose simplices are sets of vertices which span a cone of the fan. When a fan $\Sigma$ in $\R^n$ is simplicial, finite and complete, the \textit{geometric realisation} of $\Delta(\Sigma)$ is a \textit{simplicial sphere}, a triangulation of the $(n-1)$-sphere, obtained by intersecting the unit $(n-1)$-sphere with the cones of $\Sigma$. However, we remark that this geometric realisation is not necessarily a polytope \cite[Ex. 7.5]{ZieglerLectures}. \\

An arbitrary category $\Ccal$ defines a topological space, called the \textit{classifying space} $\Bcal \Ccal$. This space is the geometric realisation of a simplicial set known as the \textit{simplicial nerve} of the category. The 0-simplices correspond to the objects of $\Ccal$ and the $k$-simplices correspond to the chains of composable non-identity morphisms $(X_0 \xrightarrow[]{f_1} X_1 \xrightarrow[]{f_2} \dots \xrightarrow[]{f_k} X_k)$ in $\Ccal$. In \cref{sec:CWconstr} we view the fan as a simplicial complex and build the classifying space of the category of a partitioned fan from it. For this purpose we need the following two constructions for simplicial complexes:
\begin{enumerate}
    \item The \textit{link} $\link_\Delta(\sigma)$  of a simplex $\sigma \in \Delta$ is the simplicial subcomplex of $\Delta$ given by
    \[ \link_\Delta(\sigma) \coloneqq \{ \tau \in \Delta: \sigma \cap \tau = \emptyset \text{ and } \sigma \cup \tau \in \Delta\}.\]
    \item The \textit{join} $\Delta_1 * \Delta_2$ of two simplicial complexes $\Delta_1, \Delta_2$ has vertex set $\Delta_1^0 \cup \Delta_2^0$ and simplices given by
    \[\Delta_1 * \Delta_2 \coloneqq \{\sigma \in \Delta_1^0 \cup \Delta_2^0 : \sigma \cap \Delta_1^0 \in \Delta_1 \text{ and } \sigma \cap \Delta_2^0 \in \Delta_2\}. \]
   The join with a simplicial complex consisting of a single vertex is called the \textit{(topological) cone} over a simplicial complex.
\end{enumerate}

\section{The category of a partitioned fan} \label{sec:definition}

We now generalise the geometric construction of the $\tau$-cluster morphism category \cite{STTW2023} to a simplicial fan. Cones which are identified in the construction of \cite{STTW2023} ``have the same relative fan structure around them'' in the following sense: The collection of cones containing a cone $\sigma$ is denoted by $\starr(\sigma) \coloneqq \{ \tau \in \Sigma: \sigma \subseteq \tau\}$. Let $\pi_{\sigma}: \R^n \to \spann\{ \sigma\}^\perp$ be the projection onto its orthogonal complement. For each cone $\sigma \in \R^n$, this defines another fan $\pi_{\sigma}( \starr (\sigma))$. Then two cones $\sigma_1, \sigma_2 \in \R^n$ which get identified in the construction of \cite{STTW2023} satisfy $\spann\{ \sigma_1\}^\perp = \spann\{ \sigma_2\}^\perp$ and $\pi_{\sigma_1}(\starr(\sigma_1)) = \pi_{\sigma_2}(\starr (\sigma_2))$. Importantly, not all pairs of cones sharing these properties are identified. However, when generalising from the $g$-fan of a finite-dimensional algebra to an arbitrary simplicial fan, the information of which cones to identify is lost. Therefore, let $\Sigma$ be a fan in $\R^n$ and consider for each cone $\sigma_1 \in \Sigma$ the collection of \textit{potential identifications}
\[\Ecal_{\sigma_1} \coloneqq \{ \sigma_2 \in \Sigma : \spann\{ \sigma_1 \}^\perp = \spann\{ \sigma_2 \}^\perp \text{ and } \pi_{\sigma_1}(\starr (\sigma_1)) = \pi_{\sigma_2} (\starr (\sigma_2)) \}. \]

It is clear that this is an equivalence relation and therefore $\Ecal_{\sigma_1} = \Ecal_{\sigma_2}$ for any two cones sharing these properties. Now the set of potential identifications may be partitioned into sets of \textit{actual identifications}. Recall, that a \textit{partition} of a set $X$ is a set $P$ of non-empty  pairwise-disjoint subsets, called \textit{blocks}, of $X$ whose union is $X$. 
In other words, we split each $\Ecal_\sigma$ into blocks $\Ecal_{\sigma}^1, \dots, \Ecal_{\sigma}^{m_\sigma}$ for some $1 \leq m_\sigma \leq |\Ecal_\sigma|$ such that these coincide for all representatives of $\sigma \in \Ecal_{\sigma}$. This induces a partition $\Pfrak$ of the fan $\Sigma$ and we write $\sigma_1 \sim \sigma_2$ whenever $\sigma_1, \sigma_2 \in \Ecal_{\sigma}^k$ for some $1 \leq k \leq m_\sigma$ and $\sigma \in \Sigma$. 

\begin{defn} \label{defn:admissiblepartition}
    A partition $\Pfrak$ of $\Sigma$ as described above is called \textit{admissible} if whenever $\sigma_1 \sim \sigma_2$ are such that $\pi_{\sigma_1}(\tau_1) = \pi_{\sigma_2}(\tau_2)$ for some $\tau_1 \in \starr(\sigma_1)$ and $\tau_2 \in \starr(\sigma_2)$, then $\tau_1 \sim \tau_2$. A \textit{partitioned fan} is a pair $(\Sigma, \Pfrak)$ of a simplicial fan $\Sigma$ and an admissible partition $\Pfrak$ of $\Sigma$.
\end{defn}

This says that if two cones $\sigma_1, \sigma_2$ are in the same equivalence class, then any two cones which are ``in the same relative position'' should be identified. This restriction is necessary to make the composition of morphisms in the category of a partitioned fan well-defined. It is not obvious that non-trivial admissible partitions exist, because the cones $\tau_1, \tau_2$ in the definition may not satisfy $\tau_1, \tau_2 \in \Ecal_\tau$ for some $\tau \in \Sigma$, in other words it may not be possible to identify $\tau_1$ and $\tau_2$ with the rules defined above. Before proving that non-trivial admissible partitions always exist, we state the following elementary result from linear algebra for the sake of completeness in our context.

\begin{lem} \label{lem:linearalglem}
    Let $\sigma \subseteq \tau \in \Sigma^n$, then $\pi_{\tau} \circ \pi_{\sigma}= \pi_{\tau}$.
\end{lem}
\begin{proof}
    Every vector $v \in \R^n$ has a unique orthogonal decomposition $v = \pi_{\sigma}(v) + p_{\sigma}(v)$, where $p_{\sigma}: \R^n \to \spann\{ \sigma\}$ is the orthogonal projection onto the subspace of $\R^n$ spanned by $\sigma$. Then
    \[ \pi_{\tau}(v) = \pi_{\tau}(\tau_\sigma(v) + p_{\sigma}(v)) = \pi_{\tau} \circ \pi_{\sigma}(v) + \pi_{\tau}\circ p_{\sigma}(v).\]
    But since $p_{\sigma}(v) \subseteq \spann\{ \tau\}$, it follows that $\pi_{\tau} \circ p_{\sigma}= 0$. 
\end{proof}

\begin{lem} \label{lem:possident}
    Let $\sigma_1 \sim \sigma_2$ in $\Cfrak(\Sigma, \Pfrak)$. If there exist $\tau_1 \in \starr(\sigma_1)$ and $\tau_2 \in \starr(\sigma_2)$ such that $\pi_{\sigma_1}(\tau_1) = \pi_{\sigma_2}(\tau_2)$, then $\tau_1, \tau_2 \in \Ecal_{\tau}$ for some $\tau \in \Sigma$.
\end{lem}
\begin{proof}
    By definition, we need to show that in this case $\spann\{ \tau_1\}^\perp = \spann\{ \tau_2\}^\perp$ and $\pi_{\tau_1}(\starr (\tau_1)) = \pi_{\tau_2}(\starr (\tau_2))$, which would imply $\tau_1, \tau_2 \in \Ecal_{\tau_1}= \Ecal_{\tau_2}$. Because $\pi_{\sigma_1}(\tau_1) \cap \sigma_1 = \{ 0 \}$, we may take any basis $B_1$ of $\pi_{\sigma_1}(\tau_1) = \pi_{\sigma_2}(\tau_2)$ and any basis $B_2$ of $\spann(\sigma_1) = \spann(\sigma_2)$, then $B_1 \cup B_2$ is a basis of $\spann\{ \tau_1\}$ and $\spann\{ \tau_2\}$, which immediately implies $\spann\{ \tau_1 \}^\perp = \spann \{ \tau_2\}^\perp$, as required. To show that $\pi_{\tau_1}(\starr (\tau_1)) = \pi_{\tau_2}(\starr (\tau_2))$, we use \cref{lem:linearalglem}. Since $\sigma_i \subseteq \tau_i$ for $i=1,2$, we have
    \[ \pi_{\tau_1}(\starr(\tau_1)) = \pi_{\tau_1}(\pi_{\sigma_1}(\starr(\tau_1))) = \pi_{\tau_2}(\pi_{\sigma_2}(\starr(\tau_2))) = \pi_{\tau_2}(\starr(\tau_2)), \]
    because $\pi_{\tau_1} = \pi_{\tau_2}$ and since $\sigma_1 \sim \sigma_2$ and $\starr_{\pi_{\sigma_i}(\starr(\sigma_i))}(\pi_{\sigma_i}(\tau_i)) = \pi_{\sigma_i}(\starr_\Sigma (\tau_i))$ for $i=1,2$ imply $\pi_{\sigma_1}(\starr(\tau_1)) = \pi_{\sigma_2}(\starr(\tau_2))$. Therefore, $\Ecal_{\tau_1} = \Ecal_{\tau_2}$.
\end{proof}

Therefore, admissible partitions exist and throughout this paper, let $\Pfrak$ denote an admissible partition unless stated otherwise. We are ready to define our central object.

\begin{defn} \label{defn:thecategory}
    Given a partitioned fan ($\Sigma, \Pfrak)$, define the \textit{category of the partitioned fan} $\Cfrak(\Sigma, \Pfrak)$ as follows: 
    \begin{enumerate}
        \item The objects of $\Cfrak(\Sigma, \Pfrak)$ are equivalence classes $[\sigma]$ of the partition $\Pfrak$ of $\Sigma$.
        \item The set of morphisms $\Hom_{\Cfrak(\Sigma, \Pfrak)}([\sigma], [\tau])$ consists of equivalence classes of objects in 
        \[ \bigcup_{\sigma_i \in [\sigma], \tau_j \in [\tau]} \Hom_\Sigma(\sigma_i, \tau_j)\]
        under the equivalence relation where $f_{\sigma_1 \tau_1} \sim f_{\sigma_2 \tau_2}$ if and only if $\pi_{\sigma_1}(\tau_1) = \pi_{\sigma_2}(\tau_2)$.
        \item Given $[f_{\sigma \kappa}] \in \Hom_{\Cfrak(\Sigma, \Pfrak)}([\sigma], [\kappa])$ and $[f_{\kappa \tau}] \in \Hom_{\Cfrak(\Sigma, \Pfrak)}([\kappa], [\tau])$, their composition is defined as $[f_{\kappa \tau}] \circ [f_{\sigma \kappa}]= [f_{\sigma \tau}]$.
    \end{enumerate}
\end{defn}

Because there exists a unique morphism $f_{\sigma \tau}$ in the poset category $\Sigma$ whenever $\sigma \subseteq \tau$, we also need to identify any two compositions of morphisms $[f_{\kappa_1 \tau}] \circ [f_{\sigma \kappa_1}]$ and $[f_{\kappa_2 \tau}] \circ [f_{\sigma \kappa_2}]$ which map to the same representative of an equivalence class. Similar to \cite[Rem. 3.5]{STTW2023} it is not clear that the composition of morphisms in $\Cfrak(\Sigma, \Pfrak)$ is well-defined for two reasons:
    \begin{enumerate}
        \item In order to define the composition of two non-zero morphisms $[f_{\sigma_1 \kappa_1}]$ and $[f_{\kappa_2 \tau_2}]$ where $\kappa_1 \sim \kappa_2$, we need to find a morphism $f_{\kappa_1 \tau_1} \sim f_{\kappa_2 \tau_2}$ so that $[f_{\kappa_2 \tau_2}] \circ [f_{\sigma_1 \kappa_1}] = [f_{\sigma_1 \tau_1}]$.
        \item Given morphisms $f_{\sigma_1 \kappa_1} \sim f_{\sigma_2 \kappa_2}$ and $f_{\kappa_1 \tau_1} \sim f_{\kappa_2 \tau_2}$, we need to show $ [f_{\kappa_1 \tau_1}] \circ [f_{\sigma_1 \kappa_1}] = [f_{\kappa_2 \tau_2}] \circ [f_{\sigma_2 \kappa_2}]$.
    \end{enumerate}

Now we resolve the other two problems in an analogous way to \cite[Lem. 3.9, Lem. 3.10]{STTW2023}.

\begin{lem}\label{lem:morphcomp1}
    For any two morphisms $[f_{\sigma_1 \kappa_1}]$ and $[f_{\kappa_2 \tau_2}]$ in $\Cfrak(\Sigma, \Pfrak)$ with $\kappa_1 \sim \kappa_2$ there exists a morphism $f_{\kappa_1 \tau_1} \sim f_{\kappa_2 \tau_2}$ with $\tau_1 \sim \tau_2$.
\end{lem}
\begin{proof}
    Since $\kappa_1 \sim \kappa_2$ it follows by definition that $\pi_{\kappa_1}(\starr (\kappa_1)) = \pi_{\kappa_2}(\starr (\kappa_2))$. Thus, for each $\tau_2 \in \starr(\kappa_2)$ there exists $\tau_1 \in \starr(\kappa_1)$ such that $\pi_{\kappa_1}(\tau_1)= \pi_{\kappa_2}(\tau_2)$. Since the partition is admissible, it follows that therefore $\tau_1 \sim \tau_2$. Thus there exists a morphism $f_{\kappa_1 \tau_1}$ which satisfies $f_{\kappa_1 \tau_1} \sim f_{\kappa_2 \tau_2}$.
\end{proof}

The following is a special case of \cref{lem:morphcomp1} where $\sigma_1 = \kappa_1$, which will be used throughout this paper.
\begin{cor}\label{cor:morphstartingpoint}
    Let $\Pfrak$ be an admissible partition and let $f_{\sigma_2 \tau_2}$ be a morphism of the poset category $\Sigma$. If $\sigma_1 \sim \sigma_2$ in $\Pfrak$, then there exists a morphism $f_{\sigma_1 \tau_1}$ in the poset category such that $f_{\sigma_1 \tau_1} \sim f_{\sigma_2 \tau_2}$ in $\Cfrak(\Sigma, \Pfrak)$.
\end{cor}

\begin{lem}
    Let $f_{\sigma_1 \kappa_1} \sim f_{\sigma_2 \kappa_2}$ and $f_{\kappa_1 \tau_1} \sim f_{\kappa_2 \tau_2}$ be two pairs of identified morphisms in $\Cfrak(\Sigma, \Pfrak)$. Then $[f_{\sigma_1 \kappa_1}] \circ [f_{\kappa_1 \tau_1}] = [f_{\kappa_2 \tau_2}] \circ [f_{\sigma_2 \kappa_2}]$.
\end{lem}
\begin{proof}
    We know by assumption that $\pi_{\sigma_1}(\kappa_1) = \pi_{\sigma_2}(\kappa_2)$ and $\pi_{\kappa_1}(\tau_1) = \pi_{\kappa_2}(\tau_2)$. Now take $w_1 \in \kappa_1$ and $w_2 \in \kappa_2$ such that $\pi_{\sigma_1}(w_1) = \pi_{\sigma_2}(w_2)$ and take $v \in \pi_{\kappa_1}(\tau_1)= \pi_{\kappa_2}(\tau_2)$. Then, similar to the proof of \cref{lem:possident}, a basis of $\spann \{ \tau_1\}$ consists of the union of a basis of $\kappa_1$ and a basis of $\pi_{\kappa_1}(\tau_1)$. Hence there exists a scalar $\epsilon_1 > 0$ such that $w_1 + \epsilon_1 v \in \tau_1$. Similarly there exists $\epsilon_2 > 0$ such that $w_2 + \epsilon_2 v \in  \tau_2$. Set $\delta = \min \{ \epsilon_1, \epsilon_2\}$, then $w_i + \delta v \in \tau_i$ for $i=1,2$. Then
    \begin{align*}
        \pi_{\sigma_1}(w_1 + \delta v) = \pi_{\sigma_1}(w_1) + \delta \pi_{\sigma_1}(v) = \pi_{\sigma_2}(w_2) + \delta \pi_{\sigma_2}(v) = \pi_{\sigma_2}(w_2 + \delta v).
    \end{align*}
    Hence $\pi_{\sigma_1}(\tau_1) \cap \pi_{\sigma_2}(\tau_2) \neq \{0\}$. Because of the equivalence of fans $\pi_{\sigma_1}(\starr \sigma_1) = \pi_{\sigma_2}(\starr \sigma_2)$, the intersection is either zero or the two projections coincide. This implies $\pi_{\sigma_1}(\tau_1) = \pi_{\sigma_2}(\tau_2)$ and thus $f_{\sigma_1 \tau_1} \sim f_{\sigma_2 \tau_2}$.
\end{proof}

As a consequence we have the following.

\begin{prop}
    The category of \cref{defn:thecategory} is well-defined.
\end{prop}

\begin{exmp}
    Consider the complete fan in \cref{fig:hirzebruch} giving rise to a common toric variety called the \textit{Hirzebruch surface} $\mathbb{F}_{a}$ where $\sigma_3 = \cone\{ (-1,a)\}$ and $a$ is a positive integer. The possible identifications for this fan are
    \[ \Ecal_{0} = \{0\}, \quad \Ecal_{\sigma_1}= \{ \sigma_1\}, \quad \Ecal_{\sigma_2} = \{ \sigma_2, \sigma_4\}, \quad \Ecal_{\sigma_3}= \{ \sigma_3\}, \quad \Ecal_{\tau_1} = \{ \tau_1, \tau_2, \tau_3, \tau_4\}.\]
    
    The trivial case of not making any identifications gives the standard poset category $\Sigma = \Cfrak(\Sigma, \Pfrak_{\text{poset}})$ whose classifying space is the disk, or more specifically a square. There are only two cones of dimension 1 we may identify, namely $\sigma_2$ and $\sigma_4$ since their linear spans coincide. Moreover, they are such that $\pi_{\sigma_2}(\tau_2) = \pi_{\sigma_4}(\tau_1)$ and $\pi_{\sigma_2}(\tau_3) = \pi_{\sigma_4}(\tau_4)$ and therefore to make the partition admissible, we must identify the cones in the following way:
    \[\Pfrak_{1} = \{\{0\}, \{ \sigma_1\}, \{ \sigma_2, \sigma_4\}, \{ \sigma_3\}, \{ \tau_1, \tau_2\}, \{\tau_3,  \tau_4\}\}.\]
    
    The category is displayed in \cref{fig:ex1identifications}, where the identified morphisms are given the same label and colour. The classifying space is a cylinder and is obtained from the square by identifying the opposite ``sides'' $\tau_4 \xleftarrow{} \sigma_4 \xrightarrow{} \tau_1$ and $\tau_3 \xleftarrow{} \sigma_2 \xrightarrow{} \tau_2$, see \cref{subsec:CW} for more details. It is also possible to additionally identify all rank 2 cones whose classifying space would join the two ends of the cylinder in one point. If the rank 1 cones are not identified, arbitrary identifications may be made among the rank 2 cones, giving rise to the topological spaces coming from a square with any combination of vertices identified. Recall in \cref{fig:ex1identifications} that any two compositions having the same target are identified.

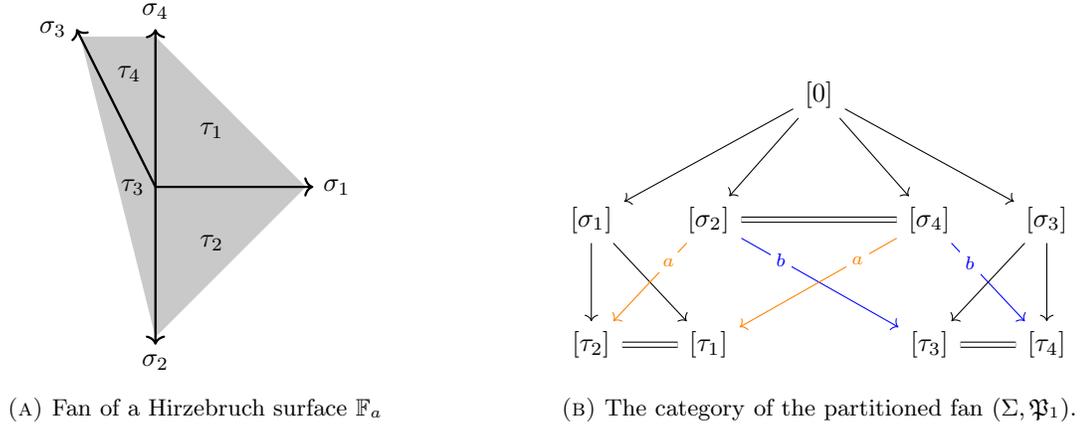
\begin{figure}[ht!]
\centering
\begin{subfigure}{.5\textwidth}
     \centering
        \begin{tikzpicture}
          \draw[fill=gray!42] (0,2) -- (0,0) -- (2,0) (0.75,0.75) node{$\tau_1$};
          \draw[fill=gray!42] (0,-2) -- (0,0) -- (2,0) (0.75,-0.75)  node{$\tau_2$};
          \draw[fill=gray!42] (-1,2) -- (0,0) -- (0,2) (-0.35,1.5) node{$\tau_4$};
          \draw[fill=gray!42] (-1,2) -- (0,0) -- (0,-2) (-0.3, 0) node{$\tau_3$};
          \draw[line width=0.3mm,->] (0, 0) -- (0, 2.1) node[above] {$\sigma_4$};
          \draw[line width=0.3mm,->] (0, 0) -- (0, -2.1) node[below] {$\sigma_2$};
          \draw[line width=0.3mm,->] (0,0 ) -- (2.1,0) node[right] {$\sigma_1$};
          \draw[line width=0.3mm,->] (0,0 ) -- (-1.05,2.1) node[left] {$\sigma_3$};
        \end{tikzpicture}
      \caption{Fan of a Hirzebruch surface $\mathbb{F}_a$}
      \label{fig:hirzebruch}
\end{subfigure}%
\begin{subfigure}{.5\textwidth}
  \centering
  \[ \begin{tikzcd}[ampersand replacement=\&,column sep = 2em, row sep=3em]
\& \& {[}0{]} \arrow[ld] \arrow[lld] \arrow[rd] \arrow[rrd] \\
%Second row
{[}\sigma_1{]}  \arrow[d] \arrow[rd] \& {[}\sigma_2{]} \arrow[rr, equal] \arrow[ld, orange, "a" {pos=0.25, description}] \arrow[rrd, blue, "b" {pos=0.25, description}] \&\& {[}\sigma_4{]} \arrow[lld, orange, "a" {pos=0.25, description}] \arrow[rd, blue, "b" {pos=0.25, description}]\& {[}\sigma_3{]} \arrow[ld] \arrow[d] \\
%Bottom row
{[}\tau_{2}{]} \arrow[r, equal] \& {[}\tau_{1}{]} \& \& {[}\tau_{3}{]} \arrow[r, equal] \& {[}\tau_{4}{]} 
\end{tikzcd}\]
\caption{The category of the partitioned fan $(\Sigma, \Pfrak_{1})$.}
\label{fig:ex1identifications}
\end{subfigure}%
\caption{A fan and an associated category}
\end{figure}

\end{exmp}

\subsection{Cubical categories}

Gromov \cite{Gromov1987} observed that for cube complexes, being \textit{locally} CAT(0) and thus a $K(\pi,1)$ space, is equivalent to a local combinatorial condition. Informally speaking, CAT(0) spaces are topological spaces whose geodesic triangle are ``no fatter than Euclidean triangles'' and locally CAT(0) spaces are those which admit a CAT(0) universal cover. Such spaces are important examples of $K(\pi,1)$ spaces, a class of Eilenberg-MacLane spaces. A connected topological space $X$ is a $K(\pi,1)$ space if it satisfies the following equivalent conditions:
\begin{enumerate}
    \item The homotopy groups of $X$ above degree 1 are all trivial;
    \item The universal cover of $X$ is contractible;
    \item The cohomology of $X$ with arbitrary coefficients is isomorphic to the cohomology of its fundamental group.
\end{enumerate}

Igusa \cite{Igusa2022} categorified this approach and introduced \textit{cubical categories}, whose classifying spaces are cube complexes, where a cube complex is a topological space built from $k$-cubes, similar to a simplicial complex being built from $k$-simplices. In this setting, the conditions of Gromov give a set of sufficient conditions for the classifying space of the category to be a $K(\pi,1)$ (cf. \cite[Prop. 3.4]{Igusa2022}). For representation-finite hereditary algebras, Igusa and Todorov \cite{IgusaTodorov2017} show that the cluster morphism category is a cubical category and make use of the sufficient conditions to show it is a $K(\pi,1)$ for $\pi$ the picture group. Moreover, the $\tau$-cluster morphism category is a cubical category for all algebras whose $g$-fan is finite \cite{HansonIgusa2021}. We remark that their proof actually works for all finite-dimensional algebras. Results about $K(\pi,1)$ spaces have been obtained by Hanson and Igusa for Nakayama algebras \cite{HansonIgusa2021} and in a follow-up paper for Nakayama-like and certain gentle algebras \cite{HansonIgusaPW2SMC}. In this section we show that all categories associated to a partitioned fan are cubical categories, extending results of \cite{HansonIgusa2021}, \cite{Igusa2022} and \cite{IgusaTodorov2017}. For the definition of a cubical category we need the following two categories:
\begin{itemize}
    \item The \textit{standard $k$-cube category} $\Ical^k$ is the name-giving example of a cubical category. It is the poset category on subsets of $\{1, \dots,k\}$ where morphisms are given by inclusion.
    \item For any category $\Ccal$ and any morphism $(A \xrightarrow[]{f} B) \in \Ccal$, the \textit{factorisation category} $\Faq(f)$ is the category whose objects are factorisations $A \xrightarrow[]{g} C \xrightarrow[]{h} B$ such that $h \circ g = f$ and whose morphisms
\[
\begin{tikzcd}[row sep=1.5]
    & C_1 \arrow[rd, "h_1"] \arrow[dd, "\phi", dashed] \\
    A \arrow[ru, "g_1"] \arrow[rd, "g_2", swap] & & B \\
    & C_2 \arrow[ru, "h_2", swap]
\end{tikzcd}\]
are morphisms $\phi: C_1 \to C_2$ such that $\phi \circ g_1 = g_2$ and $h_1 = h_2 \circ \phi$. 
\end{itemize}

Given an object $(A \xrightarrow[]{g} C \xrightarrow[]{h} B)$ in $\Faq(f)$, we call $g$ a \textit{first factor of $f$} if $g$ is irreducible in $\Ccal$ and $h$ a \textit{last factor of $f$} if $h$ is irreducible in $\Ccal$. We are now ready to define cubical categories.

\begin{defn} \label{def:cubical}
    A \textit{cubical category} is a small category $\Ccal$ with the following properties:
    \begin{enumerate}
        \item Every morphism $f: A \to B$ in $\Ccal$ has a \textit{rank} $\rank(f)$ which is a non-negative integer such that $\rank(g \circ f) = \rank(f) + \rank(g)$ for all composable morphisms $f,g \in \Ccal$. 
        \item If $\rank(f)=k$ then there is an isomorphism $\Faq(f) \cong \Ical^k$. 
        \item The forgetful functor $\Faq(f) \to \Ccal$ sending $(A \xrightarrow[]{} C \xrightarrow[]{} B) \mapsto C$ is an embedding. Thus, every morphism of rank $k$ has $k$ distinct first factors and $k$ distinct last factors.
        \item Every morphism of rank $k$ is determined by its $k$ first factors.
        \item Every morphism of rank $k$ is determined by its $k$ last factors.
    \end{enumerate}
\end{defn}

Condition 2 says that in a cubical category, the classifying space of any morphism $f$ is a solid cube, i.e. $\Bcal\Faq(f) = [0,1]^{\rank(f)}$. For the category $\Cfrak(\Sigma, \Pfrak)$ of a partitioned fan, the rank of a morphism $[f_{\sigma \tau}]$ in $\Cfrak(\Sigma, \Pfrak)$ is given by the difference of the dimensions of the cones, in other words,
\[ \rank( [f_{\sigma \tau}]) = \dim \tau - \dim \sigma.\]

This is well-defined since two identified cones $\sigma_1 \sim \sigma_2$ in $\Cfrak(\Sigma, \Pfrak)$ have the same linear span in $\R^n$ by definition and therefore the same dimension. We begin by investigating the finest partition $\Pfrak_{\text{poset}}$ of the fan, in which case the the category $\Cfrak(\Sigma, \Pfrak_{\text{poset}})$ is just the poset category naturally associated with the fan.

\begin{lem} \label{lem:posetcon2}
    Let $\Sigma$ be a simplicial fan. The category $\Cfrak(\Sigma, \Pfrak_{\textnormal{poset}})$ satisfies condition 2. of \cref{def:cubical}.
\end{lem}
\begin{proof}
    Let $\sigma \subseteq \tau \in \Sigma$ be two cones and consider the morphism $f_{\sigma \tau} \in \Cfrak(\Sigma, \Pfrak_{\text{poset}})$. Assume that $\dim \sigma = k$ and $\dim \tau = \ell$. Then we may express $\tau$ in the following way:
    \[ \tau = \cone \{ v_1, \dots, v_{\ell-k}, \sigma\} \in \Sigma\]
    where $v_1, \dots, v_{\ell-k}\in \Sigma$ are the linearly independent dimension 1 cones of $\Sigma$ contained in $\tau$ which are not contained in $\sigma$. This way of writing $\tau$ is unique. The following bijection between objects of $\Ical^{\ell-k}$ and $\Faq(f_{\sigma \tau})$ induces an isomorphism of categories
    \begin{align*} \Ical^{\ell-k} &\to \Faq (f_{\sigma \tau}) \\
    \{1,\dots, \ell-k\} \supseteq S &\mapsto (\sigma \to \cone\left\{ \{ v_i\}_{i \in S}, \sigma\right\} \to \tau)
    \end{align*}
    where morphisms are induced by the individual poset structures, in other words $S \subseteq P \in \Ical^{\ell-k}$ if and only if $\cone\left\{ \{ v_i\}_{i \in S}, \sigma\right\} \subseteq \cone\left\{ \{ v_i\}_{i \in P}, \sigma\right\}$. It is clear how to define the inverse and that this is an isomorphism.
\end{proof}

The categories of a partitioned fan are constructed from the poset category $\Cfrak(\Sigma, \Pfrak_{\textnormal{poset}})$ via identifications. The following key lemma is essential in understanding how the category changes when cones are identified. In particular, it shows that two morphisms which are identified will have coinciding factorisation cubes. 

\begin{lem} \label{lem:uniquefaq}
    Let $\sigma_1 \neq \sigma_2 \in \Sigma$ and $\tau_1 \neq \tau_2 \in \Sigma$ be distinct cones such that $f_{\sigma_1 \tau_1} \sim f_{\sigma_2 \tau_2}$ in $ \Cfrak(\Sigma, \Pfrak)$. Then for every factorisation $\sigma_1 \xrightarrow[]{f_{\sigma_1 \kappa_1}} \kappa_1 \xrightarrow[]{f_{\kappa_1 \tau_1}} \tau_1$ of $f_{\sigma_1 \tau_1}$ there exists a unique factorisation $\sigma_2 \xrightarrow[]{f_{\sigma_2 \kappa_2}} \kappa_2 \xrightarrow[]{f_{\kappa_2 \tau_2}} \tau_2$
    of $f_{\sigma_2 \tau_2}$ such that $\kappa_1 \sim \kappa_2$. In this case we have $f_{\sigma_1 \kappa_1}\sim f_{\sigma_2, \kappa_2}$ and $f_{\kappa_1 \tau_1} \sim f_{\kappa_2 \tau_2}$ in $\Cfrak(\Sigma, \Pfrak)$.
\end{lem}
\begin{proof}
    If $\sigma_1$ or $\sigma_2$ is a maximal cone, then such $\tau_1$ and $\tau_2$ do not exist and the result is trivial. Therefore assume that $\sigma_1$ and $\sigma_2$ are not maximal. By assumption $\sigma_1 \sim \sigma_2$ and $\tau_1 \sim \tau_2$ and thus $\pi_{\sigma_1}(\starr (\sigma_1)) = \pi_{\sigma_2}(\starr(\sigma_2))$. So for every $\kappa_1 \in \starr(\sigma_1)$ satisfying $\kappa_1 \subseteq \tau_1$ there exists a unique $\kappa_2 \in \starr(\sigma_2)$ such that $\pi_{\sigma_1}(\kappa_1) = \pi_{\sigma_2}(\kappa_2)$. Because the partition $\Pfrak$ is admissible and $\sigma_1 \sim \sigma_2$, this implies that $\kappa_1 \sim \kappa_2$. This implies $f_{\sigma_1 \kappa_1} \sim f_{\sigma_2 \kappa_2}$ by definition. We must have that $\kappa_2 \subseteq \tau_2$, since
    \[ \pi_{\sigma_2}(\kappa_2) = \pi_{\sigma_1}(\kappa_1) \subseteq \pi_{\sigma_1}(\tau_1) = \pi_{\sigma_2}(\tau_2)\]
    implies that $\kappa_2$ is a face of $\tau_2$. To show that $f_{\kappa_1 \tau_1} \sim f_{\kappa_2 \tau_2}$, we use \cref{lem:linearalglem} since $\sigma_i \subseteq \kappa_i$ implies $\kappa_i^\perp \subseteq \sigma_i^\perp$ for $i=1,2$. Thus,
    \[ \pi_{\kappa_1}(\tau_1) = \pi_{\kappa_1} \circ \pi_{\sigma_1}(\tau_1) = \pi_{\kappa_2} \circ \pi_{\sigma_2}(\tau_2) = \pi_{\kappa_2}(\tau_2),\]
    since $\pi_{\sigma_1}(\tau_1) = \pi_{\sigma_2}(\tau_2)$ follows from $f_{\sigma_1 \tau_1} \sim f_{\sigma_2 \tau_2}$ by definition. 
\end{proof}

Therefore, the category $\Cfrak(\Sigma, \Pfrak)$ of a partitioned fan satisfies condition 2. of \cref{def:cubical} for any admissible partition. The following example illustrates the identification of factorisation cubes described in the previous lemma.

\begin{exmp}
    Consider the standard coordinate fan $\Sigma$ in $\R^3$ whose cones are spanned by the linearly independent (positive) combinations of the vectors $\pm e_i \in \R^3$. Choose a partition $\Pfrak$ such that $e_1 \sim - e_1 \in \Cfrak(\Sigma, \Pfrak)$. Then the factorisation cubes of $[f_{e_1, \cone\{e_1,e_2,e_3\}}]$ and $[f_{-e_1, \cone\{-e_1,e_2,e_3\}}]$ are identified as follows:
    
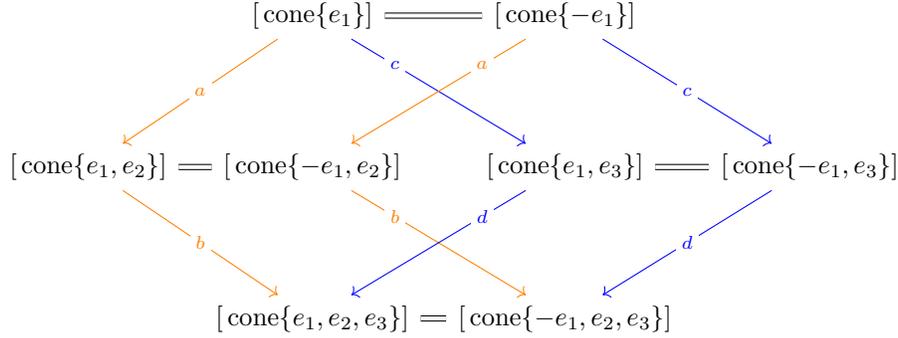
\[ \begin{tikzcd}[ampersand replacement=\&,column sep = 1em, row sep=4em]
    \& {[}\cone\{e_1\}{]} \arrow[r, equal] \arrow[ld, orange, "a" { description}, swap]\arrow[rd, blue, "c" {pos=0.25, description}] \&  {[}\cone\{-e_1\}{]}  \arrow[rd, blue, "c" {description}] \arrow[ld, orange, "a" {pos=0.25, description}, swap] \\
    %Second row
    {[}\cone\{ e_1, e_2\}{]} \arrow[r, equal] \arrow[rd, orange, "b" {description}, swap]\& {[}\cone\{ -e_1, e_2\}{]} \arrow[rd, orange, "b" {pos=0.25, description}, swap]  \& {[}\cone\{ e_1, e_3\}{]} \arrow[r, equal] \arrow[ld, blue, "d" {pos=0.25, description}] \& {[}\cone\{ -e_1, e_3\}{]} \arrow[ld, blue, "d" {description}]\\
    \& {[}\cone\{ e_1, e_2, e_3\}{]} \arrow[r, equal] \&  {[}\cone\{ -e_1, e_2, e_3\}{]}
    \end{tikzcd}\]
\end{exmp}

Consider the following analogue of \cite[Lem. 2.5b]{HansonIgusa2021} a tool to obtain the embedding of \cref{def:cubical}.
\begin{lem} \label{lem:embedsetup}
    Let $[f_{\sigma \tau}] \in \Cfrak(\Sigma, \Pfrak)$. There exists exactly one morphism
    \[( [\sigma] \xrightarrow[]{[g_1]} [\kappa] \xrightarrow[]{[h_1]} [\tau]) \to ([\sigma] \xrightarrow[]{[g_2]} [\lambda] \xrightarrow[]{[h_2]} [\tau])\]
    in $\Faq([f_{\sigma \tau}])$ whenever $\kappa_1 \subseteq \lambda_1$ for some $\kappa_1 \in [\kappa]$ and $\lambda_1 \in [\lambda]$ and none otherwise.
\end{lem}
\begin{proof}
    The assumption that two such maps $[f_{\kappa_1 \lambda_1}] \neq [f_{\kappa_2 \lambda_2}]$ exist implies the existence of distinct representatives $\kappa_1 \sim \kappa_2$ and $\lambda_1 \sim \lambda_2$ in the equivalence classes. Assume that there exists only one representative $\sigma \in [\sigma]$, then $\kappa_1, \kappa_2 \in \starr(\sigma)$ implies $\pi_{\sigma}(\kappa_1) \neq \pi_{\sigma}(\kappa_2)$. Hence $[f_{\sigma \kappa_1}] \neq [f_{\sigma \kappa_2}]$ which is a contradiction since they both equal $[g_1]$. It follows that there need to be distinct representatives $\sigma_1, \sigma_2 \in [\sigma]$ and thus we have two sets of inclusions $\sigma_i \subseteq \kappa_i \subseteq \lambda_i$ satisfying $[f_{\sigma_i \lambda_i}] = [g_2]$ for $i=1,2$. Then \cref{lem:uniquefaq} implies that $[f_{\kappa_1 \lambda_1}] = [f_{\kappa_2 \lambda_2}]$, a contradiction. The (non-)existence is obvious.
\end{proof}

We remark that there is not necessarily a unique morphism $[\kappa] \to [\lambda]$ in $\Cfrak(\Sigma, \Pfrak)$ in general. However, if one $\lambda_1 \supseteq \kappa_1$ exists, then there exist $\lambda_i \supseteq \kappa_i$ for every $\kappa_i \in [\kappa]$, such that $\lambda_i \in [\lambda]$ and $f_{\kappa_i \lambda_i} \in [f_{\kappa \lambda}]$ by \cref{lem:uniquefaq}. We are now able to prove that the third condition of the definition of a cubical category holds for the category of a partitioned fan.

\begin{lem}\label{lem:embedding}
    Let $[f_{\sigma \tau}]$ be a morphism in $\Cfrak(\Sigma, \Pfrak)$. The forgetful functor $\Faq([f_{\sigma \tau}]) \to \Cfrak(\Sigma, \Pfrak)$ given by $([\sigma] \to [\kappa] \to [\tau]) \mapsto [\kappa]$ is an embedding i.e. faithful and injective-on-objects.
\end{lem} 
\begin{proof}
    Assume two factorisations of a morphism $[f_{\sigma \tau}] \in \Cfrak(\Sigma, \Pfrak)$
    \[ [\sigma_1] \xrightarrow[]{[f_{\sigma_1 \kappa_1}]} [\kappa_1] \xrightarrow[]{[f_{\kappa_1 \tau_1}]} [\tau_1] \quad \text{ and }\quad [\sigma_2] \xrightarrow[]{[f_{\sigma_2 \kappa_2}]} [\kappa_2] \xrightarrow[]{[f_{\kappa_2 \tau_2}]} [\tau_2] \]
    satisfy $[\kappa_1] = [\kappa_2]$ but $\kappa_1 \neq \kappa_2$. We want to use \cref{lem:uniquefaq} to show the pairwise identification of the morphisms and thus an equivalence of factorisations. We need to show that $\sigma_1 \neq \sigma_2$ and $\tau_1 \neq \tau_2$. Since $\kappa_1$ and $\kappa_2$ have the same linear span but are distinct, it follows that there is no generator of one which is linearly independent with respect to the generators of the other and thus they cannot both be contained in the same simplicial cone $\tau$. Hence there must exist distinct $\tau_i \supseteq \kappa_i$ for $i=1,2$. However, this implies the existence of two distinct representatives $\sigma_1$ and $\sigma_2$ in $[\sigma]$ similar to the proof of \cref{lem:embedsetup} as otherwise $\tau_1 = \tau_2$ follows from the identification of the maps $[f_{\sigma \tau_1}]= [f_{\sigma \tau_2}]$. Therefore, \cref{lem:uniquefaq} implies the pairwise-identification of the morphisms and hence the two factorisations are identical and the functor is injective-on-objects. The functor is faithful since there exists at most one morphism between any two objects in $\Faq([f_{\sigma \tau}])$ by \cref{lem:embedsetup}.
\end{proof}

In $\Cfrak(\Sigma, \Pfrak)$, just like in the poset category of the fan $\Sigma$, the irreducible morphisms are exactly the morphisms of rank 1. Hence, given a morphism $[f_{\sigma\tau}] \in \Cfrak(\Sigma, \Pfrak)$ where $\dim(\tau)= \ell$ and $\dim(\sigma) = k$, we may write $\tau = \cone\{ \sigma, v_1, \dots, v_{\ell-k}\}$ like before. Then it is easily seen that the $\ell-k$ first factors are the rank 1 morphisms $[f_{\sigma  \kappa_i}]$, where $\kappa_i \coloneqq \cone \{ \sigma, v_i\} \subseteq \tau$ and the last factors are the rank 1 morphisms $[f_{\lambda_i \tau}]$, where 
\[\lambda_i \coloneqq \cone\{ \sigma, v_1, \dots, v_{i-1}, v_{i+1}, \dots, v_{\ell-k}\}. \]

Putting the above results together we arrive at the following.

\begin{thm} \label{thm:cubicalcat}
    The category $\Cfrak(\Sigma, \Pfrak)$ of a partitioned fan is cubical.
\end{thm}
\begin{proof}
    Like stated before, the rank function is given by $\rank([f_{\sigma \tau}]) = \dim \tau - \dim \sigma$. Condition 2. of \cref{def:cubical} follows from combining \cref{lem:posetcon2} and \cref{lem:uniquefaq}, whereas condition 3. was shown explicitly in \cref{lem:embedding}. Let $[f_{\sigma \tau}] \in \Cfrak(\Sigma, \Pfrak)$ be a morphism. Given a choice of representative $f_{\sigma' \tau'} \in [f_{\sigma \tau}]$, the first factors are given by $\{[f_{\sigma' \kappa_i'}]\}_{i=1}^{\ell-k}$, where $\kappa_i' \subseteq \tau'$ are constructed from $\sigma'$ and $\tau'$ as above. It is then clear that they uniquely determine the morphism $[f_{\sigma' \tau'}]$ by writing $\tau' = \cone\{ \sigma', \kappa_1', \dots, \kappa_{\ell-k}'\}$. On the other hand, the last factors $[f_{\lambda_i \tau}]$ are easily seen to determine the morphism $[f_{\sigma \tau}]$, where $\sigma = \bigcap_{i=1}^{\ell-k} \lambda_i$.
\end{proof}

\section{The classifying space and picture group} \label{sec:CWconstr}

We have established that the categories of a partitioned fan are cubical categories, hence it is natural to study their topological properties next. Throughout this section, let $\Sigma$ be a finite and complete fan in $\R^n$ and $\Pfrak$ an admissible partition of $\Sigma$. In this case, we describe the classifying space of the category $\Cfrak(\Sigma, \Pfrak)$ as a finite CW-complex, similar to \cite{HansonIgusa2021}, \cite{Igusa2022} and \cite{IgusaTodorov2017}. Moreover, we associate the \textit{picture group} to a fan, by generalising the definitions of \cite{HansonIgusa2021} and \cite{IgusaTodorovWeyman2016}, see also \cref{thm:picgroupgeneralisation}. We investigate the sufficient conditions for the categories $\Cfrak(\Sigma, \Pfrak)$ to be a $K(\pi,1)$ by studying the picture group. One of the conditions is the existence of a faithful functor from the category $\Cfrak(\Sigma, \Pfrak)$ to a group. We show that a faithful functor to the picture group exists for fans of rank 2 and a large class of hyperplane arrangements.

\subsection{Description of the CW-complex}
\label{subsec:CW}

A \textit{CW-complex} $X$ is a topological space of particular importance in algebraic topology. It is constructed from a discrete set $X^0$, called \textit{$0$-cells}. The \textit{$k$-skeleton} $X^k$ is formed from $X^{k-1}$ by attaching $k$-cells $e_i^k$ via maps $\varphi_i: S^{k-1} \to X^{k-1}$ for some index set $I$. Hence $X^k$ is the quotient space of the disjoint union $X^{k-1} \sqcup \bigsqcup_{i \in I} D_{i}^k$ of $X^{k-1}$ with a collection of  $k$-disks $\{D_{i}^k\}_{i \in I}$ under the identification $x \sim \varphi_{i}(x)$ for $x \in \partial D_{i}^k$. As a set, $X^k$ is the disjoint union of $X^{k-1}$ with open $k$-disks. The name CW-complex comes from two properties of such complexes: closure-finiteness and weak topology. Furthermore, the fundamental group of CW-complexes are completely determined by their 1-cells and 2-cells. For more details we refer the reader to \cite{Hatcher2002}. In the following construction of the CW-complex $\Bcal \Cfrak(\Sigma, \Pfrak)$, each cell is the (topological) cone of the simplicial sphere described in the following definition, hence a disk.

\begin{defn} \label{def:simplicialcomplex}
    Let $(\Sigma, \Pfrak)$ be a partitioned fan and $\sigma \in \Sigma$ be a cone of dimension $k \neq n$. Define $\Scal(\sigma)$ to be the simplicial complex whose vertices are the cones $\tau \in \starr(\sigma)^{k+1}$ and for which $\{ \tau_i\}_{i=1}^{\ell-k}$ spans a simplex if and only if $\cone\{ \pi_{\sigma}(\tau_1), \dots, \pi_{\sigma}(\tau_{\ell-k})\} \in \pi_{\sigma}(\starr(\sigma))$.
\end{defn}

\begin{lem}\label{lem:simplicialcomplex}
    $\Scal(\sigma)$ is homeomorphic to a $(n-k-1)$-sphere.
\end{lem}
\begin{proof}
    By intersecting a finite and complete simplicial fan in $\R^n$ with a $(n-1)$-sphere centered at the origin, we obtain the geometric realisation of $\Sigma$ to be a $(n-1)$-sphere. Similarly for $\Scal(\sigma)$, the projection $\pi_{\sigma}(\starr \sigma)$ is a finite and complete simplicial fan and thus the geometric realisation of it is a $(n - \dim(\sigma) -1)$-sphere. 
\end{proof}

We remark that all representatives $\sigma_i \in [\sigma]$ define isomorphic simplicial complexes because $\sigma_i \sim \sigma_j$ implies that the fans $\pi_{\sigma_1}(\starr \sigma_1) = \pi_{\sigma_2}(\starr \sigma_2)$ coincide. Therefore, given an equivalence class $[\sigma] \in \Pfrak$, we denote by $[\Scal(\sigma)]$ the isomorphism class of simplicial complexes $S(\sigma_i)$ for $\sigma_i \in [\sigma]$.

\begin{exmp} \label{exmp:CWexample}
Consider the fan $\Sigma$ given in \cref{fig:simplicialsub1} which is the complete fan underlying the toric variety $\Pp^1 \times \Pp^1$ and at the same time the $g$-fan of the semisimple rank 2 algebra. In \cref{fig:simplicialsub2} (the geometric realisation of) the simplicial complex $\Scal(0)$ is homeomorphic to a 1-sphere, where we have labelled the vertices it is defined on. Similarly, in \cref{fig:simplicialsub3} the simplicial complex $\Scal(\sigma_1)$ consists only of the vertices and is the 0-sphere.
\end{exmp}
\begin{figure}[ht!]
\centering
\begin{subfigure}{.33\textwidth}
    \begin{center}
    \begin{tikzpicture}
      \draw[fill=gray!42] (0,1.5) -- (0,0) -- (1.5,0) (0.5,0.5) node{$\tau_1$};
      \draw[fill=gray!42] (0,-1.5) -- (0,0) -- (1.5,0) (0.5,-0.5) node{$\tau_2$};
      \draw[fill=gray!42] (0,-1.5) -- (0,0) -- (-1.5,0) (-0.5,-0.5) node{$\tau_3$};
      \draw[fill=gray!42] (0,1.5) -- (0,0) -- (-1.5,0) (-0.5,0.5) node{$\tau_4$};
      \draw[->, line width = 0.3mm] (0, 0) -- (0, 1.6) node[above] {$\sigma_4$};
      \draw[->, line width = 0.3mm] (0, 0) -- (0, -1.6) node[below] {$\sigma_2$};
      \draw[->, line width = 0.3mm] (0,0 ) -- (1.6,0) node[right] {$\sigma_1$};
      \draw[->, line width = 0.3mm] (0,0 ) -- (-1.6,0 ) node[left] {$\sigma_3$};
    \end{tikzpicture}
\end{center}
  \caption{The complete fan $\Sigma$}
  \label{fig:simplicialsub1}
\end{subfigure}%
\begin{subfigure}{.33\textwidth}
  \begin{center}
    \begin{tikzpicture}
      \draw[-] (-1.5,0) -- (0, 1.5) node[above] {$\sigma_4$};
      \draw[-] (1.5,0) -- (0, -1.5) node[below] {$\sigma_2$};
      \draw[-] (0, 1.5) -- (1.5,0) node[right] {$\sigma_1$};
      \draw[-] (0,-1.5) -- (-1.5,0) node[left] {$\sigma_3$};
      \filldraw[black] (0, 1.5) circle (2pt) node[left] {};
      \filldraw[black] (0, -1.5) circle (2pt) node[left] {};
      \filldraw[black] (1.5,0) circle (2pt) node[left] {};
      \filldraw[black] (-1.5,0) circle (2pt) node[left] {};
    \end{tikzpicture}
\end{center}
  \caption{The complex $\Scal(0)$}
  \label{fig:simplicialsub2}
\end{subfigure}%
\begin{subfigure}{.33\textwidth}
  \begin{center}
    \begin{tikzpicture}
      \filldraw[black] (0, 1.5) circle (2pt) node[above] {$\tau_1$};
      \filldraw[black] (0, -1.5) circle (2pt) node[below] {$\tau_2$};
    \end{tikzpicture}
\end{center}
  \caption{The complex $\Scal(\sigma_1)$}
  \label{fig:simplicialsub3}
\end{subfigure}
\caption{An example of the simplicial complexes of a fan.}
\label{fig:simplicialexample}
\end{figure}
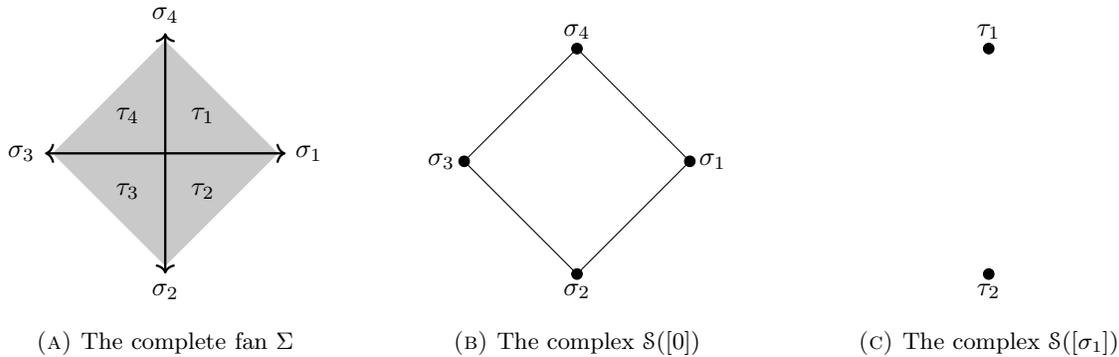

The remainder of this subsection is devoted to proving the following result. We closely follow the proof of \cite[Prop. 4.7]{HansonIgusa2021} and identify each morphism $[f_{\sigma \tau}] \in \Cfrak(\Sigma, \Pfrak)$ with its factorisation cube in $\Bcal \Cfrak(\Sigma, \Pfrak)$. 

\begin{thm} \label{thm:CWcomplex}
    Let $\Sigma$ be a finite and complete fan in $\R^n$ and $\Pfrak$ an admissible partition of $\Sigma$. The classifying space $\Bcal \Cfrak(\Sigma, \Pfrak)$ is a $n$-dimensional CW-complex having one cell $e([\sigma])$ of dimension $k= n- \dim(\sigma)$ for each equivalence class $[\sigma] \in \Pfrak$. The $k$-cell $e([\sigma])$ is the union of the factorisation cubes of the morphisms $[f_{\sigma \tau}]$, where $\tau \in \starr(\sigma)^{n}$.
\end{thm}

Let $e([\sigma])$ be the union of factorisation cubes of all morphisms $[f_{\sigma \tau}]$, where $\tau \in \starr(\sigma)^n$. To obtain the CW-structure of $\Bcal \Cfrak(\Sigma, \Pfrak)$, we must show that each $e([\sigma])$ is a disk of dimension $k= n- \dim(\sigma)$ attached to lower dimensional cells along its boundary. We show this in three steps.

\subsubsection*{Step 1: Start with the disjoint union of factorisation cubes.}

Let $\sigma_1, \sigma_2 \in [\sigma]$ and $\tau_1 \in \starr(\sigma_1)$ and $\tau_2 \in \starr(\sigma_2)$ be maximal cones such that $f_{\sigma_1 \tau_1} \sim f_{\sigma_2 \tau_2}$. From \cref{lem:uniquefaq} we know that factorisation cubes of the two morphisms $[f_{\sigma_1 \tau_1}]$ and $[f_{\sigma_2 \tau_2}]$ are identified, hence it suffices to consider only one representative $\sigma \in [\sigma]$ by \cref{cor:morphstartingpoint}. We begin by defining the disjoint union of factorisation cubes $X([\sigma]) = \bigsqcup_{\tau \in \starr(\sigma)^n} [f_{\sigma \tau}]$. It is clear that a face of the factorisation cube of $[f_{\sigma \tau}]$ corresponds to the factorisation cube of some morphism $[f_{\kappa \lambda}]$ satisfying $\sigma \subseteq \kappa \subseteq \lambda \subseteq \tau$. Consider the equivalence relation $\sim_1$ on $X([\sigma])$ which identifies faces corresponding to the same morphism. By definition we obtain $e([\sigma]) \equiv X([\sigma]) / \sim_1$. We split $\sim_1$ into two types of identifications. First, we only identify factorisation cubes of morphisms of the form $[f_{\sigma \lambda}]$. Denote this equivalence on $X([\sigma])$ by $\sim_2$. After showing that the resulting space is a disk, we let $\sim_3$ identify the factorisation cubes of morphisms $[f_{\kappa \tau}]$ not starting at $\sigma$ and show that the identifications of $\sim_3$ occur only on the boundary. Thus since $X([\sigma])/\sim_1 = (X([\sigma])/\sim_2 )/\sim_3$ we obtain the desired result.

\subsubsection*{Step 2: Show that $X([\sigma])/\sim_2$ is a disk.} We compare this quotient space with the classifying space of the following category.

\begin{defn}
    Given an object $[\sigma] \in \Cfrak(\Sigma, \Pfrak)$, the \textit{under category} (or \textit{coslice category}) $[\sigma] \backslash \Cfrak(\Sigma, \Pfrak)$ is the category whose objects are morphisms $[\sigma] \to [\tau]$ in $\Cfrak(\Sigma, \Pfrak)$ and whose morphisms $([\sigma]\to[\tau_1]) \to ( [\sigma] \to [\tau_2])$ are morphisms $[f_{\tau_1 \tau_2}]$ making the following triangle commute:
        \[\begin{tikzcd} &{[} \sigma{]} \arrow[rd, "{[}f_{\sigma \tau_2}{]}"]  \arrow[ld,"{[}f_{\sigma \tau_1}{]}", swap] \\ 
    {[} \tau_1{]} \arrow[rr,"{[}f_{\tau_1 \tau_2}{]}"]& & {[} \tau_2 {]} \end{tikzcd}\]
\end{defn}

The 0-simplices of the classifying space $\Bcal[\sigma] \backslash \Cfrak(\Sigma, \Pfrak)$ are in bijection with (identity morphisms of) elements of $[\sigma] \backslash \Cfrak(\Sigma, \Pfrak)$. There exists one 0-simplex for every cone in $\starr(\sigma)$. It follows from \cref{lem:embedsetup} that this category is a poset category, and one can easily observe that $\Bcal [\sigma] \backslash \Cfrak(\Sigma, \Pfrak) \equiv X([\sigma])/\sim_2$. \\

Moreover, there is a bijection between 0-simplices corresponding to objects in $\starr(\sigma)^{\dim(\sigma) + 1}$ and vertices of the simplicial complex $\Scal(\sigma)$ of \cref{def:simplicialcomplex}. The other 0-simplices of $[\sigma] \backslash \Cfrak(\Sigma, \Pfrak)$ correspond to cones $\tau \in \starr(\sigma)^{\dim(\sigma)+\ell}$ for $\ell \in \{2, \dots, k\}$ and are in bijection with the $(\ell-1)$-simplices of the simplicial complex $\Scal(\sigma)$, recalling the that $\sigma_1 \sim \sigma_2$ implies $[\Scal(\sigma_1)]=[\Scal(\sigma_2)]$. Therefore $\Bcal [\sigma] \backslash \Cfrak(\Sigma, \Pfrak)$ may be viewed as the cone with cone point $[f_{\sigma \sigma}]$ over $\Scal(\sigma)$. This implies that $\Bcal [\sigma] \backslash \Cfrak(\Sigma, \Pfrak)$ and thus $X([\sigma])/\sim_2$ is a $(n-k-1)$-disk.

\subsubsection*{Step 3: Show that identifications happen on the boundary.}
Define $\sim_3$ to be the equivalence relation on $\Bcal [\sigma] \backslash \Cfrak(\Sigma, \Pfrak)$ identifying all faces corresponding to the same morphism $[f_{\kappa \tau}]$ for $[\kappa] \neq [\sigma]$. By definition $e([\sigma]) \equiv \Bcal [\sigma] \backslash \Cfrak(\Sigma, \Pfrak) / \sim_3$. From the construction of $\Bcal [\sigma] \backslash \Cfrak(\Sigma, \Pfrak)$ as the cone over $\Scal(\sigma)$, it follows from \cref{lem:simplicialcomplex} that the link of $[f_{\sigma \sigma}]$ is a $(n-k-1)$-sphere. Hence $[f_{\sigma \sigma}]$ is in the interior of the disk. Consider now a different 0-simplex $[f_{\sigma \tau}]$ of $\Bcal [\sigma] \backslash \Cfrak(\Sigma, \Pfrak)$. Then its link $\link([f_{\sigma \tau}])$ is given by the simplicial join of:
\begin{itemize}
    \item the link of $[f_{\tau \tau}]$ in $\Bcal [\tau] \backslash \Cfrak(\Sigma, \Pfrak)$, which is the part in the boundary of $e([\sigma])$, and
    \item the link of $[\sigma] \xrightarrow[]{[f_{\sigma \tau}]} [\tau] \xrightarrow[]{{[}f_{\tau \tau}{]}} [\tau]$ in the factorisation cube $[f_{\sigma \tau}]$, which is the part in the interior of $e([\sigma])$.
\end{itemize}
This is the join of a $(n- \dim(\tau)-1)$-sphere with a $(\dim(\tau)-\dim(\sigma)-1)$-disk, which is a $(k-1)$-disk. Thus $[f_{\sigma \tau}]$ is a boundary vertex of $\Bcal [\sigma] \backslash \Cfrak(\Sigma, \Pfrak)$. Any simplex which contains the cone point $[f_{\sigma \sigma}]$ has to be the only representative in its equivalence class of $\sim_3$, by definition. It follows that $\sim_3$ identifies simplices containing exclusively simplices on the boundary of $\Bcal [\sigma] \backslash \Cfrak(\Sigma, \Pfrak)$. Moreover, these identified simplices are factorisation cubes of morphisms of rank strictly less than $k$. Hence $e([\sigma])$ is attached to lower-dimensional cells. This concludes the proof of \cref{thm:CWcomplex}.

\begin{exmp}
    Consider the complete fan of \cref{fig:simplicialsub1} with the partition $\Pfrak$ identifying $\sigma_1 \sim \sigma_3$ and $\sigma_2 \sim \sigma_4$. This partition implies the identification of all maximal cones. In \cref{fig:3step}, the three steps of constructing the 2-cell $e([0])$ are illustrated. First we begin with the disjoint union of factorisation cubes in \cref{fig:3stepa} and then $\sim_2$ identifies corresponding factorisation cubes of morphisms starting at 0, which are labelled with the same number. We showed that the resulting space is a disk with only one 0-simplex in the interior. Then $\sim_3$ identifies all factorisation cubes of morphisms on the boundary which are identified in $\Cfrak(\Sigma, \Pfrak)$. For example, in this partition $[f_{\sigma_4 \tau_4}] = [f_{\sigma_2 \tau_3}]$, which are labelled ``5'' in \cref{fig:3stepb}. A detailed picture of $X([0])/ \sim_2 = \Bcal [0] \backslash \Cfrak(\Sigma, \Pfrak)$ is given in \cref{fig:CWcell}. These identifications give the 2-cell $e([0])$ of the classifying space $\Bcal \Cfrak(\Sigma, \Pfrak)$ to be a torus.

\begin{figure}[ht!]
\centering
\begin{subfigure}{.3\textwidth}
    \begin{center}
    \begin{tikzpicture}
    %Top right square
      \draw[-] (0.2, 0.2) -- (0.2, 1.7) node[midway,fill=white] {\small4};
      \draw[-] (0.2, 0.2) -- (1.7, 0.2) node[midway,fill=white] {\small1};
      \draw[-] (0.2, 0.2) -- (1.7, 1.7) node[above] {};
      \draw[-] (1.7, 0.2) -- (1.7, 1.7) node[above] {};
      \draw[-] (0.2, 1.7) -- (1.7, 1.7) node[above] {};
      \filldraw[black] (0.2,0.2) circle (1pt) node[left] {};
      \filldraw[black] (1.7,0.2) circle (1pt) node[left] {};
      \filldraw[black] (0.2,1.7) circle (1pt) node[left] {};
      \filldraw[black] (1.7,1.7) circle (1pt) node[left] {};
      %Top left square
      \draw[-] (-0.2, 0.2) -- (-0.2, 1.7) node[midway,fill=white] {\small4};
      \draw[-] (-0.2, 0.2) -- (-1.7, 0.2) node[midway,fill=white] {\small3};
      \draw[-] (-0.2, 0.2) -- (-1.7, 1.7) node[above] {};
      \draw[-] (-1.7, 0.2) -- (-1.7, 1.7) node[above] {};
      \draw[-] (-0.2, 1.7) -- (-1.7, 1.7) node[above] {};
      \filldraw[black] (-0.2,0.2) circle (1pt) node[left] {};
      \filldraw[black] (-1.7,0.2) circle (1pt) node[left] {};
      \filldraw[black] (-0.2,1.7) circle (1pt) node[left] {};
      \filldraw[black] (-1.7,1.7) circle (1pt) node[left] {};
      %Bottom right square
      \draw[-] (0.2, -0.2) -- (0.2, -1.7) node[midway,fill=white] {\small2};
      \draw[-] (0.2, -0.2) -- (1.7, -0.2) node[midway,fill=white] {\small1};
      \draw[-] (0.2, -0.2) -- (1.7, -1.7) node[above] {};
      \draw[-] (1.7, -0.2) -- (1.7, -1.7) node[above] {};
      \draw[-] (0.2, -1.7) -- (1.7, -1.7) node[above] {};
      \filldraw[black] (0.2,-0.2) circle (1pt) node[left] {};
      \filldraw[black] (1.7,-0.2) circle (1pt) node[left] {};
      \filldraw[black] (0.2,-1.7) circle (1pt) node[left] {};
      \filldraw[black] (1.7,-1.7) circle (1pt) node[left] {};
      %Bottom left square
      \draw[-] (-0.2, -0.2) -- (-0.2, -1.7) node[midway,fill=white] {\small2};
      \draw[-] (-0.2, -0.2) -- (-1.7, -0.2) node[midway,fill=white] {\small3};
      \draw[-] (-0.2, -0.2) -- (-1.7, -1.7) node[above] {};
      \draw[-] (-1.7, -0.2) -- (-1.7, -1.7) node[above] {};
      \draw[-] (-0.2, -1.7) -- (-1.7, -1.7) node[above] {};
      \filldraw[black] (-0.2,-0.2) circle (1pt) node[left] {};
      \filldraw[black] (-1.7,-0.2) circle (1pt) node[left] {};
      \filldraw[black] (-0.2,-1.7) circle (1pt) node[left] {};
      \filldraw[black] (-1.7,-1.7) circle (1pt) node[left] {};
    \end{tikzpicture}
\end{center}
  \caption{$X([0])$}
\label{fig:3stepa}
\end{subfigure}%
\begin{subfigure}{.3\textwidth}
  \begin{center}
    \begin{tikzpicture}
    \filldraw[black] (0,0) circle (1pt) node[left] {};
    %Top right square
      \draw[-] (0, 0) -- (0, 1.5) node[above] {};
      \draw[-] (0,0 ) -- (1.5,0) node[above] {};
      \draw[-] (0,0) -- (1.5, 1.5) node[above] {};
      \draw[-] (1.5,0) -- (1.5,1.5) node[midway,fill=white] {\small7};
      \draw[-] (0,1.5) -- (1.5,1.5) node[midway,fill=white] {\small 6};
      \filldraw[black] (1.5,1.5) circle (1pt) node[left] {};
      \filldraw[black] (0,1.5) circle (1pt) node[left] {};
      \filldraw[black] (1.5,0) circle (1pt) node[left] {};
      %Bottom right square
      \draw[-] (0,0 ) -- (0,-1.5) node[above] {};
      \draw[-] (0,0) -- (1.5, -1.5) node[above] {};
      \draw[-] (1.5,0) -- (1.5, -1.5) node[midway,fill=white] {\small8};
      \draw[-] (0,-1.5) -- (1.5, -1.5) node[midway,fill=white] {\small6};
      \filldraw[black] (1.5,-1.5) circle (1pt) node[left] {};
      \filldraw[black] (0,-1.5) circle (1pt) node[left] {};
      %Bottom left square
      \draw[-] (0,0) -- (-1.5,-1.5) node[above] {};
      \draw[-] (0,0) -- (-1.5, 0) node[above] {};
      \draw[-] (0,-1.5) -- (-1.5, -1.5) node[midway,fill=white] {\small5};
      \draw[-] (-1.5,0) -- (-1.5, -1.5) node[midway,fill=white] {\small8};
      \filldraw[black] (-1.5,-1.5) circle (1pt) node[left] {};
      \filldraw[black] (-1.5,0) circle (1pt) node[left] {};
      %Top left square
      \draw[-] (0,0) -- (-1.5,1.5) node[above] {};
      \draw[-] (-1.5,0) -- (-1.5, 1.5) node[midway,fill=white] {\small7};
      \draw[-] (0,1.5) -- (-1.5, 1.5) node[midway,fill=white] {\small5};
      \filldraw[black] (-1.5,1.5) circle (1pt) node[left] {};
    \end{tikzpicture}
\end{center}
  \caption{$X([0])/\sim_2$}
\label{fig:3stepb}
\end{subfigure}%
\begin{subfigure}{.4\textwidth}
  \begin{center}
    \begin{tikzpicture}
    \useasboundingbox (-3,-1.5) rectangle (3,1.5);
    \draw (0,0) ellipse (3 and 1.5);
      \begin{scope}
        \clip (0,-1.8) ellipse (2.8 and 2.3);
        \draw (0,2.2) ellipse (3 and 2.5);
      \end{scope}
      \begin{scope}
        \clip (0,2.2) ellipse (3 and 2.5);
        \draw (0,-2.2) ellipse (3 and 2.5);
      \end{scope}
      \begin{scope}
        \clip (0,-0.3) rectangle (-1,-1.5);
        \draw (0,-0.9) ellipse (0.2 and 0.6);
      \end{scope}
      \begin{scope}
        \clip (0,-0.3) rectangle (1,-1.5);
        \draw[dashed] (0,-0.9) ellipse (0.2 and 0.6);
      \end{scope}
      \begin{scope}
         \clip (-3,0) rectangle (3,-1);
          \draw[] (0,0) ellipse (3 and 1);
      \end{scope}
      \begin{scope}
         \clip (-3,0) rectangle (3,1);
          \draw[dashed] (0,0) ellipse (3 and 1);
      \end{scope}
      \filldraw[black] (-0.2,-1) circle (1pt) node[left] {};
      \filldraw[] (-0.1,-0.6) circle (0pt) node[left] {\tiny 7};
      \filldraw[] (-0.05,-1.3) circle (0pt) node[left] {\tiny 8};
      \filldraw[] (-0.4,-0.85) circle (0pt) node[left] {\tiny 5};
      \filldraw[] (0.7,-0.85) circle (0pt) node[left] {\tiny 6};
  \end{tikzpicture}
\end{center}
  \caption{$(X([0])/\sim_2)/\sim_3$}
\label{fig:3stepc}
\end{subfigure}
\caption{The construction of $e([0])$ from factorisation cubes.}
\label{fig:3step}
\end{figure}
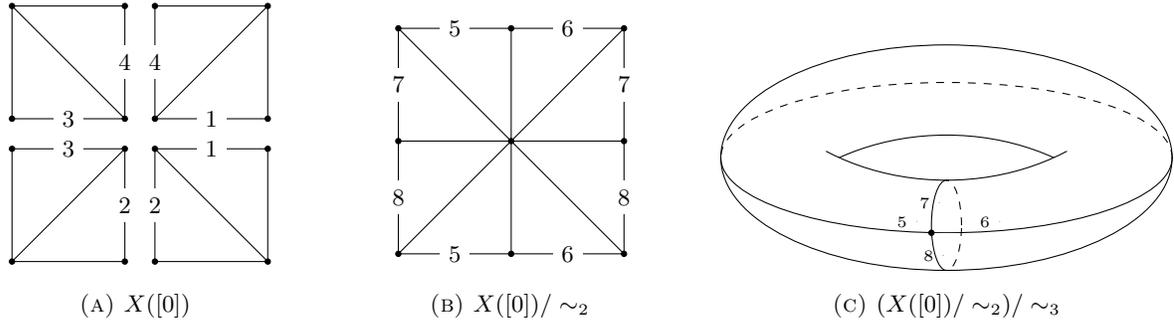

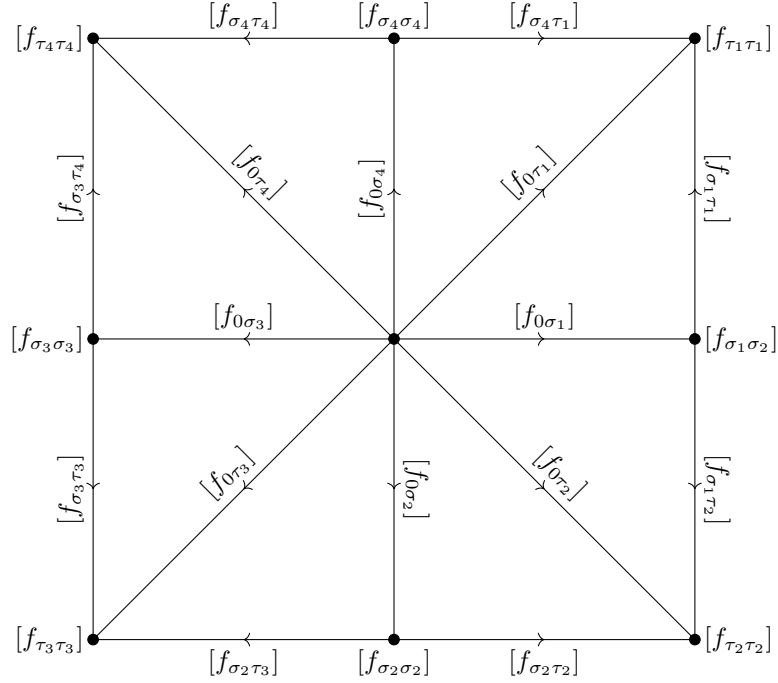
\begin{figure}[ht!]
    \centering
    \begin{tikzpicture}[middlearrow/.style={ decoration={markings, mark= at position 0.5 with {\arrow{#1}} ,}, postaction={decorate}}]
        \draw[middlearrow={>}] (0,0) -- (4,4) node[midway,sloped,above] {$[f_{0 \tau_1}]$};
        \draw[middlearrow={>}] (0,0) -- (4,0) node[midway,sloped,above] {$[f_{0 \sigma_1}]$};
        \draw[middlearrow={>}] (0,0) -- (4,-4) node[midway,sloped,above] {$[f_{0 \tau_2}]$};
        \draw[middlearrow={>}] (0,0) -- (0,-4) node[midway,sloped,above] {$[f_{0 \sigma_2}]$};
        \draw[middlearrow={>}] (0,0) -- (-4,-4) node[midway,sloped,above] {$[f_{0 \tau_3}]$};
        \draw[middlearrow={>}] (0,0) -- (-4,0) node[midway,sloped,above] {$[f_{0 \sigma_3}]$};
        \draw[middlearrow={>}] (0,0) -- (-4,4) node[midway,sloped,above] {$[f_{0 \tau_4}]$};
        \draw[middlearrow={>}] (0,0) -- (0,4) node[midway,sloped,above] {$[f_{0 \sigma_4}]$};
        \draw[middlearrow={>}] (0,4) -- (4,4) node[midway,sloped,above] {$[f_{\sigma_4 \tau_1}]$};
        \draw[middlearrow={>}] (0,4) -- (-4,4) node[midway,sloped,above] {$[f_{\sigma_4 \tau_4}]$};
        \draw[middlearrow={>}] (4,0) -- (4,4) node[midway,sloped,above,rotate=180] {$[f_{\sigma_1 \tau_1}]$};
        \draw[middlearrow={>}] (4,0) -- (4,-4) node[midway,sloped,above] {$[f_{\sigma_1 \tau_2}]$};
        \draw[middlearrow={>}] (0,-4) -- (4,-4) node[midway,sloped,below] {$[f_{\sigma_2 \tau_2}]$};
        \draw[middlearrow={>}] (0,-4) -- (-4,-4) node[midway,sloped,below] {$[f_{\sigma_2 \tau_3}]$};
        \draw[middlearrow={>}] (-4,0) -- (-4,-4) node[midway,sloped,above,rotate=180] {$[f_{\sigma_3 \tau_3}]$};
        \draw[middlearrow={>}] (-4,0) -- (-4,4) node[midway,sloped,above] {$[f_{\sigma_3 \tau_4}]$};
        \filldraw[black] (0, 4) circle (2pt) node[above] {$[f_{\sigma_4 \sigma_4}]$};
        \filldraw[black] (0, 0) circle (2pt) node[above] {};
      \filldraw[black] (0, -4) circle (2pt) node[below] {$[f_{\sigma_2 \sigma_2}]$};
      \filldraw[black] (-4,0) circle (2pt) node[left] {$[f_{\sigma_3 \sigma_3}]$};
      \filldraw[black] (4,0) circle (2pt) node[right] {$[f_{\sigma_1 \sigma_2}]$};
      \filldraw[black] (4,4) circle (2pt) node[right] {$[f_{\tau_1 \tau_1}]$};
      \filldraw[black] (4,-4) circle (2pt) node[right] {$[f_{\tau_2 \tau_2}]$};
      \filldraw[black] (-4,-4) circle (2pt) node[left] {$[f_{\tau_3 \tau_3}]$};
      \filldraw[black] (-4,4) circle (2pt) node[left] {$[f_{\tau_4 \tau_4}]$};
    \end{tikzpicture}
    \caption{The classifying space $\Bcal [0] \setminus \Cfrak(\Sigma, \Pfrak)$.}
    \label{fig:CWcell}
\end{figure}

\end{exmp}

\subsection{The picture group of a fan}
\label{sec:picturegroup}

Igusa, Todorov and Weyman \cite{IgusaTodorovWeyman2016} associate the picture group to a representation-finite hereditary algebra and study its cohomology. Later, Hanson and Igusa \cite{HansonIgusa2021} extended their definition to $\tau$-tilting finite algebras. In these cases it was shown by the respective authors, that the classifying space of the ($\tau$)-cluster morphism category has fundamental group isomorphic to the picture group and thus if the classifying space is a $K(\pi,1)$ the cohomologies of the two coincide. An algebra induces a partial ordering on the maximal cones of its $g$-fan which defines the picture group, see \cref{sec:specialcase}. To generalise the picture group to finite and complete simplicial fans we assume the existence of a \textit{weak fan poset}, generalising the notion of a fan poset introduced in \cite[Sec. 3]{Reading2005}.

\begin{defn} \label{defn:fanposet}
    A \textit{fan poset} is a pair $(\Sigma, \Pcal)$ where $\Sigma$ is a finite complete fan in $\R^n$ and $\Pcal$ is a finite poset whose elements are the maximal cones of $\Sigma$, subject to the following conditions:
    \begin{enumerate}
        \item For every cone $\sigma \in \Sigma$, the set of maximal cones $\starr(\sigma)^n$ containing $\sigma$ is an interval in $\Pcal$, which we denote by $[\sigma^-, \sigma^+]$ and call a \textit{facial interval}. 
        \item For every interval $I$ of $\Pcal$, the union of all maximal cones in $I$ is a convex polyhedral cone, which may not be strongly-convex.
    \end{enumerate}
\end{defn}

For example, the poset of regions of a central simplicial hyperplane arrangement with an arbitrary choice of base region as defined by Edelman \cite{Edelman} defines a fan poset by \cite[Thm. 4.2]{Reading2005}. Moreover, $g$-fans of $\tau$-tilting finite algebras are equipped with a natural fan poset induced by the poset of torsion classes as we show in \cref{prop:fanposet}. To achieve greater generality we weaken the definition above in the following way.

\begin{defn} \label{defn:weakfanposet}
    A \textit{weak fan poset} is a pair $(\Sigma, \Pcal)$ where $\Sigma$ is a finite complete fan in $\R^n$ and $\Pcal$ is a finite poset whose elements are the maximal cones of $\Sigma$, subject to the following conditions:
    \begin{enumerate}
        \item For every cone $\sigma \in \Sigma$, the set of maximal cones $\starr(\sigma)^n$ containing $\sigma$ is an interval in $\Pcal$, which we denote by $[\sigma^-, \sigma^+]$ and call a \textit{facial interval}. 
        \item Every cover relation $\tau_1 \lessdot \tau_2$ in $\Pcal$ can be seen as a facial interval $[\sigma^-, \sigma^+]$ for some $(n-1)$-dimensional cone $\sigma = \tau_1 \cap \tau_2$.
    \end{enumerate}
\end{defn}

We call a triple $(\Sigma, \Pfrak, \Pcal)$ a \textit{partitioned fan poset} if $(\Sigma, \Pfrak)$ is a partitioned fan and $(\Sigma, \Pcal)$ is a weak fan poset. The result \cite[Prop. 3.3]{Reading2005} implies that if $(\Sigma, \Pcal)$ is a fan poset then it is also a weak fan poset.

\begin{exmp}
    Consider the fan of the Hirzebruch surface in \cref{fig:hirzebruch}, then the poset $\Pcal$ whose two maximal chains are given by $\tau_3 \lessdot \tau_2 \lessdot \tau_1$ and $\tau_3 \lessdot \tau_4 \lessdot \tau_1$ with no other cover relations forms a weak fan poset but not a fan poset, since the union $\tau_2 \cup \tau_3$ is not a convex polyhedral cone.
\end{exmp}

For a poset $\Pcal$ we denote by $\Hasse(\Pcal)$ the oriented graph whose vertices are elements of $\Pcal$ and which has an arrow $x \xleftarrow{} y$ whenever $x \lessdot y$ is a cover relation in $\Pcal$. If $(\Sigma, \Pcal)$ is a weak fan poset, we label the arrows of $\Hasse(\Pcal)$ with the cone of codimension 1 giving rise to the cover relation. We are now ready to define the picture group of a partitioned fan with a choice of fan poset.

\begin{defn} \label{defn:picgroup}
    Let $(\Sigma, \Pfrak, \Pcal)$ be a partitioned fan poset. Define the \textit{picture group} $G(\Sigma, \Pfrak, \Pcal)$ to have generators $\{ X_{[\sigma]}: \sigma \in \Sigma^{n-1}\}$ and the following sets of relations:
    \begin{enumerate} 
        \item $ X_{[\sigma_1]} \dots X_{[\sigma_k]} = X_{[\sigma_1']} \dots X_{[\sigma_\ell']}$ whenever $(\sigma_1, \dots, \sigma_k)$ and $(\sigma_1', \dots, \sigma_\ell')$ are two distinct ordered sequences of cones of codimension 1 labelling the arrows of some maximal chain of an interval $[\tau_1, \tau_2]$ in $\Pcal$. We abbreviate this group element by $X_{[\tau_1, \tau_2]}$.
        \item $X_{[\sigma_1^-, \kappa_1^-]} = X_{[\sigma_2^-, \kappa_2^-]}$, whenever $[f_{\sigma_1 \kappa_1}] = [f_{\sigma_2 \kappa_2}]$ in $\Cfrak(\Sigma, \Pfrak)$.
    \end{enumerate}
\end{defn}

Notice that the picture group $G(\Sigma, \Pfrak, \Pcal)$ satisfies $X_{[\sigma^-, \kappa^-]} X_{[\kappa^-, \tau^-]} = X_{[\sigma^-, \tau^-]}$ for appropriate $\sigma, \kappa, \tau \in \Sigma$. We prove in \cref{subsec:tauclustermorph} that the picture group defined for $\tau$-tilting finite algebras in \cite{HansonIgusa2021} satisfies these relations. Given a fan $\Sigma$, different choices of a weak fan poset $(\Sigma, \Pcal)$ may define non-isomorphic picture groups. Moreover, it is possible that some generators become trivial due to the relation $X_{[\sigma]}=X_e$ arising from relations of the second type in \cref{defn:picgroup}. We call a fan poset which does not annihilate any generators in this way \textit{non-degenerate}. For a partitioned fan poset $(\Sigma, \Pfrak, \Pcal)$ we say that $\Pcal$ is \textit{well-defined on identified stars} whenever the induced fan posets $\Pcal|_{\pi_{\sigma_1}}$ and $\Pcal|_{\pi_{\sigma_2}}$ on $\pi_{\sigma_1}(\starr \sigma_1)$ and $\pi_{\sigma_2}(\starr \sigma_2)$ coincide for distinct cones $\sigma_1, \sigma_2 \in \Sigma$ such that $[\sigma_1] = [\sigma_2]$ in $\Cfrak(\Sigma, \Pfrak)$. These two notions are equivalent. 

\begin{lem} \label{lem:nondegimplystars}
    A fan poset $(\Sigma, \Pcal)$ is non-degenerate if and only if $\Pcal$ is well-defined on identified stars. In this case the second set of relations of \cref{defn:picgroup} is trivially satisfied.
\end{lem}
\begin{proof}
    The fan poset is determined by its covering relations of the form $\tau_1 \lessdot \tau_2$ whenever $\codim \tau_1 \cap \tau_2 = 1$. Hence it is clear that if $\Pcal$ is well-defined on identified stars, then for $\sigma_1 \sim \sigma_2$ of codimension 1 we have $\pi_{\sigma_1}(\sigma_1^-) = \pi_{\sigma_1}(\sigma_1)^- = \pi_{\sigma_2}(\sigma_2)^- = \pi_{\sigma_2}(\sigma_2^-)$. Thus the fan poset is non-degenerate. Conversely, if $(\Sigma,\Pcal)$ is non-degenerate, assume for a contradiction that there is some pair of identified stars such that the induced poset on their projections does not coincide. Then at least one covering relation of the poset must be different in the projection and that wall would be crossed in two different directions in the fan poset and thus the poset would no longer be non-degenerate. Thus $\Pcal$ is well-defined on identified stars. In this case, the relations of type 2 in \cref{defn:picgroup} are satisfied because by definition $[f_{\sigma_1 \tau_1}]= [f_{\sigma_2 \tau_2}]$ implies $\pi_{\sigma_1}(\tau_1) = \pi_{\sigma_2}(\tau_2)$ and since the poset is well-defined on identified stars we have $\pi_{\sigma_1}(\sigma_1^-) = \pi_{\sigma_1}(\sigma_1)^- = \pi_{\sigma_2}(\sigma_2)^- = \pi_{\sigma_2}(\sigma_2^-)$ and similarly for $\tau_1^-$ and $\tau_2^-$. Since the paths from $\pi_{\sigma_i}(\tau_i^-)$ to $\pi_{\sigma_i}(\sigma_i^-)$ coincide in the projection for $i=1,2$ and the partition is admissible, the labels of the paths from $\tau_i^-$ to $\sigma_i^-$ must be pairwise identified. 
\end{proof}

The assumption on a (weak) fan poset $\Pcal$ of the following lemma is referred to as $\Pcal$ being \textit{induced by} a linear functional $b \in (\R^n)^{\lor}$ in \cite[Sec. 3]{Reading2005}.

\begin{lem}\label{lem:linfunctposet}
    Let $(\Sigma,\Pcal)$ be a weak fan poset. Let $b$ be a linear functional on $\R^n$. For any $\tau_1 \lessdot \tau_2$ in $\Pcal$, let $\nu$ be the unit normal vector to the hyperplane $\spann\{ \tau_1 \cap \tau_2\}\subseteq \R^n$ separating $\tau_1$ from $\tau_2$, oriented to point from $\tau_1$ to $\tau_2$. If for every $\tau_1 \lessdot \tau_2$ in $\Pcal$ with corresponding $\nu$ we have $b(\nu) > 0$, then the poset is non-degenerate, equivalently, well-defined on identified stars.
\end{lem}
\begin{proof}
    The poset is determined by its covering relation. Let $[\sigma_1]=[\sigma_2] \in \Cfrak(\Sigma, \Pfrak)$ be an equivalence class of codimension 1 cones. We need to show that the poset induced on $\pi_{\sigma_i}(\starr \sigma_i)$ coincides. Assume for a contradiction that $\tau_1 \lessdot \tau_2$ is the covering relation given by the star of $\sigma_1$ and $\tau_3 \lessdot \tau_4$ the covering relation given by the star of $\sigma_2$ such that $\pi_{\sigma_1}(\tau_1) = \pi_{\sigma_2}(\tau_4)$ and $\pi_{\sigma_1}(\tau_2) = \pi_{\sigma_2}(\tau_3)$. Then the normal vector $\nu_1$ pointing from $\tau_1$ to $\tau_2$ is orthogonal to the hyperplane $\spann \{ \sigma_1\}$ and satisfies $b(\nu_1)> 0$, but at the same time the normal vector $\nu_2$ pointing from $\tau_3$ to $\tau_4$ is orthogonal to the hyperplane $\spann\{ \sigma_2\}$ and satisfies $b(\nu_2) > 0$. Since $[\sigma_1]=[\sigma_2]$ we have $\spann \{ \sigma_1\} = \spann\{\sigma_2\}$, it follows that $\nu_1 = - \nu_2$, but then either $b(\nu_1) < 0$ or $b(\nu_2) < 0$, a contradiction.
\end{proof}

In a similar way to \cite[Prop. 4.4d]{HansonIgusa2021} we may give an alternate description of the picture group as follows.
\begin{lem}\label{defn:altdefn}
    If $\Pcal$ is non-degenerate, the picture group $G(\Sigma, \Pfrak, \Pcal)$ may be presented with the set of generators $\{ X_{[\sigma]}: \sigma \in \Sigma^{n-1}\} \cup \{g_{\tau}: \tau \in \Sigma^n\}$ and a relation
    \[ g_{\tau_2} = X_{[\sigma]} g_{\tau_1}\]
    if there is an arrow $\tau_1 \xleftarrow{\sigma} \tau_2$ in $\Hasse(\Pcal)$ labelled by by $\sigma$ and the relation $g_{0^-} = e$.
\end{lem}
\begin{proof}
    Let $H$ be a group with presentation given as above and let $\tau_2 \xrightarrow{\sigma_1} \dots \xrightarrow{\sigma_k} \tau_1$ and $\tau_2 \xrightarrow{\sigma_1'} \dots \xrightarrow{\sigma_\ell'} \tau_1$ be two distinct sequences of codimension 1 cones labelling the arrows of some maximal chain in the interval $[\tau_1, \tau_2] \subseteq \Pcal$. These sequences of codimension 1 cones give rise to the relation
    \[ X_{[\sigma_1]} \dots X_{[\sigma_k]} g_{\tau_1} = g_{\tau_2} = X_{[\sigma_1']} \dots X_{[\sigma_\ell']} g_{\tau_1}\]
    in $H$ which implies that $H$ satisfies the relation $X_{[\sigma_1]} \dots X_{[\sigma_k]} = X_{[\sigma_1']} \dots X_{[\sigma_\ell']}$ as required. Moreover, since $\Pcal$ has a minimal element $0^-$, there exists a sequence $\tau \xrightarrow{\sigma_1''} \dots \xrightarrow{\sigma_s''} 0^-$ labelling the arrows of a maximal chain in the interval $[0^-, \tau]$ in $\Hasse(\Pcal)$. We then obtain
    \[ g_{\tau} = X_{[\sigma_1'']} \dots X_{[\sigma_s'']} \]
    since $g_{0^-} = e$. Hence every generator of $H$ can be written in terms of generators of $G(\Sigma, \Pfrak, \Pcal)$. As a consequence of these two observations we can replace the generators $\{g_{\tau}: \tau \in \Pcal\}$ by expressions using only generators $X_{[\sigma_i]}$ and obtain the presentation of the picture group in \cref{defn:picgroup}.
\end{proof}

The finite CW-structure obtained for $\Bcal \Cfrak(\Sigma, \Pfrak)$ in \cref{thm:CWcomplex}, allows us to describe the fundamental group in the following way: By definition, the 1-skeleton of the CW-complex $\Bcal \Cfrak(\Sigma, \Pfrak)$ is a graph, which contains a loop whenever two adjacent maximal cones get identified. Since this graph is connected, it contains a maximal tree $T$. Now every edge $e$ which is not part of the tree determines a loop $f_e$ in the graph and thus a generator of the fundamental group. For more details we refer to \cite[Sec. 1.A]{Hatcher2002}.

\begin{lem} \label{lem:fundamentalgroup}
    The fundamental group $\pi_1 (\Bcal \Cfrak(\Sigma, \Pfrak))$ is the free group with one generator $[f_e]$ for each edge $e \in \Bcal \Cfrak(\Sigma, \Pfrak)^1 - T$ modulo the relations given by the attaching maps of the 2-cells.
\end{lem}

\begin{rmk}
    Taking any fan with its trivial poset partition $\Pfrak_{\text{poset}}$, since the classifying space is a ball, its fundamental group is trivial but the picture group is not. So, in contrast to the setting of finite-dimensional algebras, the picture group $G(\Sigma, \Pfrak, \Pcal)$ is not necessarily isomorphic to the fundamental group of $\Bcal \Cfrak(\Sigma, \Pcal)$.
\end{rmk}

Nonetheless, for a special class of fan posets, we obtain an isomorphism between picture group and fundmantal group. Recall that a \textit{lattice} $L$ is a partially ordered set for which any two elements $x,y$ have a unique common maximal lower bound, $x \land y$, and a unique minimal upper bound, $x  \lor y$. A \textit{polygon} in a finite lattice $L$ is an interval $[x,y]$ such that $(x,y)$ consists of two disjoint non-empty chains. The lattice is \textit{polygonal} if for any two arrows $y_1 \to x$ and $y_2 \to x$ in $\Hasse(L)$ the interval $[x, y_1 \lor y_2]$ is a polygon and for arrows $y \to x_1$ and $y \to x_2$ in $\Hasse(L)$ the interval $[x_1 \land x_2, y]$ is a polygon. We say two maximal chains in an interval $[x,y] \subseteq L$ are related by a \textit{polygon move} if the two chains differ only in that one chain covers one side of a polygon while the other covers the other side. For further background on polygonal lattices we refer to \cite[Sec. 9.6]{Reading2016}.

\begin{prop} \label{prop:polygonalandmaximal}
    Let $(\Sigma, \Pfrak, \Pcal)$ be a non-degenerate partitioned fan poset. If $\Pcal$ is a polygonal lattice, then it suffices to consider facial intervals coming from cones of codimension 2 to obtain all relations of $G(\Sigma, \Pfrak, \Pcal)$. If additionally $\Pfrak$ identifies all maximal cones of $\Sigma$, then $G(\Sigma, \Pfrak, \Pcal)$ is isomorphic to $\pi_1(\Bcal \Cfrak(\Sigma, \Pfrak))$.
\end{prop}
\begin{proof}
    The type 2 relations of \cref{defn:picgroup} are satisfied due to non-degeneracy of the poset by \cref{lem:nondegimplystars}. Let $[x,y]$ be an interval in $\Pcal$. Since $\Pcal$ is a finite polygonal lattice, any two maximal chains in $[x,y]$ are related by a sequence of polygon moves by \cite[Lem. 9-6.3]{Reading2016}. Trivially, the labels of two maximal chains which are related by a polygon move differ only in the labels of the two sides of the polygon. Thus it is sufficient to consider the group relations coming from polygons of $\Pcal$ to give a presentation of $G(\Sigma, \Pfrak, \Pcal)$.\\
    
    We now show that every cone of codimension 2 of $\Sigma$ gives rise to a polygon of $\Pcal$ and that in fact every polygon of $\Pcal$ arises this way, as claimed. Thus, let $\sigma \in \Sigma$ be a cone of codimension 2 and consider the induced weak fan poset $\Pcal|_{\pi_{\sigma}}$ on $\pi_{\sigma}(\starr \sigma)$, which has a maximal and minimal element, and two disjoint chains similar to \cref{fig:polygonrel}. Hence the interval $[\sigma^-, \sigma^+]$ is a polygon of $\Pcal$. Now, take a polygon $[\tau_1 \land \tau_2, \tau_3] \subseteq \Pcal$ for some $\tau_1 \lessdot \tau_3$ and $\tau_2 \lessdot \tau_3$, then $\kappa = \tau_1 \cap \tau_3 \cap \tau_2$ is a cone of codimension 2 since $\tau_1 \cap \tau_3$ and $\tau_2 \cap \tau_3$ are both generated by distinct subsets of $(n-1)$ vectors generating $\tau_3$. By the previous, $[\kappa^-, \kappa^+]$ is a polygon which must contain $\tau_1, \tau_2$ and $\tau_3$. More precisely we must have $\kappa^+ = \tau_3$ and by the uniqueness of the meet in a lattice we have $\kappa^- = \tau_1 \land \tau_2$, hence every polygon arises as a facial interval of a cone of codimension 2.\\

    Additionally, if all maximal cones are identified, then there exists a unique 0-cell in $\Bcal \Cfrak(\Sigma, \Pfrak))$ and \cref{lem:fundamentalgroup} implies that the generators of $G(\Sigma, \Pfrak, \Pcal)$ and $\pi_1(\Bcal \Cfrak(\Sigma, \Pfrak))$ coincide. The relations of the fundamental group of a CW-complex are given exactly by the 2-cells which correspond with the cones of codimension 2. 
\end{proof}

The previous result implies that different choices of non-degenerate fan posets define isomorphic picture groups when all maximal cones are identified by $\Pfrak$ and the fan poset is a polygonal lattice. A similar result holds for any fan in $\R^2$. 

\begin{lem} \label{lem:isopicgroups}
    Let $\Sigma$ be a finite and complete fan in $\R^2$ and $(\Sigma, \Pcal_1)$ and $(\Sigma, \Pcal_2)$ be non-degenerate fan posets. Then $G(\Sigma, \Pfrak, \Pcal_1) \cong G(\Sigma, \Pfrak, \Pcal_2)$.
\end{lem}
\begin{proof}
    \cref{lem:nondegimplystars} and non-degeneracy imply that the type 2 relations are satisfied by both $G(\Sigma, \Pfrak, \Pcal_1)$ and $G(\Sigma, \Pfrak, \Pcal_2)$. Denote the generators of $G(\Sigma, \Pfrak, \Pcal_1)$ by $X_{[\sigma_i]}$ and the generators of $G(\Sigma, \Pfrak, \Pcal_1)$ by $Y_{[\sigma_i]}$ for $\sigma_i \in \Sigma^{1}$. Given an arbitrary cover relation $\tau_i \lessdot \tau_j$ in $\Pcal_1$ with $\sigma_k \coloneqq \tau_i \cap \tau_j$, define
    \[ \delta^{\sigma_k} \coloneqq \begin{cases} 1 & \text{ if } \tau_i \lessdot \tau_j \text{ in } \Pcal_2; \\ -1 & \text{ if } \tau_j \lessdot \tau_i \text{ in } \Pcal_2. \end{cases}\]
    Then the desired group isomorphism $\phi: G(\Sigma, \Pfrak, \Pcal_1) \to G(\Sigma, \Pfrak, \Pcal_2)$ is given on the generators by $\varphi(X_{[\sigma_i]}) = Y_{[\sigma_i]}^{\delta_{\sigma_i}}$. We note that this is well-defined on equivalence classes by the assumption that $\Pcal_1$ and $\Pcal_2$ are non-degenerate.
\end{proof}

\subsection{Eilenberg-MacLane spaces of rank 2}

While the picture group is not necessarily isomorphic to the fundamental group in general, it still plays an important role in understanding the classifying spaces of the categories of a partitioned fan. Recall, that $\Cfrak(\Sigma, \Pfrak)$ is a cubical category by \cref{thm:cubicalcat}. 

\begin{prop} \cite[Prop. 3.4, Prop. 3.7]{Igusa2022} \label{prop:igusacriteria} The following additional properties are sufficient for the classifying space $\Bcal \Cfrak(\Sigma, \Pfrak)$ to be locally \textnormal{CAT(0)} and thus a $K(\pi,1)$. 
    \begin{enumerate}
        \item There is a faithful group functor $\Psi: \Cfrak(\Sigma, \Pfrak) \to G$ for some group $G$, viewed as a groupoid with one object.
        \item A set of $k$ rank 1 morphisms $\{ [f_{\sigma \kappa_i}]\}_{i=1}^k$ forms the set of first factors of a rank $k$ morphism if and only if each pair $\{ [f_{\sigma \kappa_i}], [f_{\sigma \kappa_j}]\}$ forms the set of first factors of a rank 2 morphism for $i \neq j$. In other words, first factors are given by pairwise compatibility conditions.
        \item A set of $k$ rank 1 morphisms $\{ [f_{\sigma_i \kappa}]\}_{i=1}^k$ forms the set of last factors of a rank $k$ morphism if and only if each pair $\{ [f_{\sigma_i \kappa}], [f_{\sigma_j \kappa}]\}$ forms the set of last factors of a rank 2 morphism for $i \neq j$. In other words, last factors are given by pairwise compatibility conditions.
    \end{enumerate}
\end{prop}

In particular, we consider the picture group to be the most natural group to study condition (1). We emphasise that the above conditions do not imply that $\Bcal \Cfrak(\Sigma, \Pfrak)$ is a $K(G(\Sigma, \Pfrak,\Pcal),1)$ but rather a $K(\pi,1)$ for its possibly non-isomorphic fundamental group $\pi$. While, there is a recipe for constructing a $K(G,1)$ for any finitely presented group $G$, the result may be an infinite-dimensional CW-complex (see \cite[Sec. 1.B.]{Hatcher2002}). Hence we are interested in understanding whether the finite CW-complex of \cref{thm:CWcomplex} is a $K(\pi,1)$ for its fundamental group. For this purpose, we consider the following functor from the category of a partitioned fan to its picture group with respect to some non-degenerate weak fan poset:
\begin{equation}
\begin{aligned}\label{eq:functor}
    \Psi: \Cfrak(\Sigma, \Pfrak) &\to G(\Sigma, \Pfrak, \Pcal)  \\
[f_{\sigma \kappa}] &\mapsto X_{[\sigma^-, \kappa^-]}.
\end{aligned}
\end{equation}

It follows from basic hyperplane arrangement theory (i.e. convex geometry) and the definition of admissible partitions that the functor is well-defined. Indeed, take two representatives $\sigma_1, \sigma_2 \in [\sigma]$ and $\kappa_1, \kappa_2 \in [\kappa]$ such that $[f_{\sigma_1 \kappa_1}]=[f_{\sigma_2 \kappa_2}]$. Now consider the projection of their stars onto the orthgonal complements, $\pi_{\sigma_1} (\starr \sigma_1) = \pi_{\sigma_2} (\starr \sigma_2)$, then $\pi_{\sigma}(\kappa_1^-) = \pi_{\sigma}(\kappa_2^-)$ follows from $\pi_{\sigma} (\kappa_1) = \pi_{\sigma} (\kappa_2)$ which follows from $[f_{\sigma_1 \kappa_1}]=[f_{\sigma_2 \kappa_2}]$ by definition. The terms $X_{[\sigma_i^-, \kappa_i^-]}$ for $i=1,2$ are determined by the coinciding paths $\pi_{\sigma}(\kappa_i^-)$ to $\pi_{\sigma}(\sigma_i^-)$ for $i=1,2$ and hence coincide. It is easily seen to be well-defined on identity morphisms $[f_{\sigma \sigma}]$, which get sent to the trivial element $X_{[\sigma^-, \sigma^-]} = e$. Furthermore, since the weak poset is non-degenerate, the functor is well-defined on composition of morphisms by construction. 

\begin{exmp}
    Consider a partitioned fan $(\Sigma, \Pfrak)$ as in \cref{fig:lastfactorprob1}, then the set of 3 rank 1 morphisms $\{[f_1],[f_2],[f_3]\}$ cannot be the last factors of a rank 3 morphism since no such morphism exists when $\Sigma$ is a fan in $\R^2$. However, the pairs morphisms $\{[f_1], [f_3]\}$, $\{[f_1], [f_2]\}$ and $\{ [f_2], [f_3]\}$ form the last factors of the morphisms $[f_{0 \tau_1}]$, $[f_{0 \tau_3}]$ and $[f_{0 \tau_2}]$ respectively. On the other hand, consider a finite and complete fan with only three cones of dimension 1, for example the fan whose toric variety is the projective plane, see \cite[p. 6-7]{Fulton1993}. This fan $\Sigma$ has three cones, $\cone\{e_1\}, \cone\{e_2\}$ and $\cone\{-e_1-e_2\}$, of dimension 1, which are all pairwise compatible as first factors. But since there is no morphism of rank 3, this shows that the first factors are not given by pairwise compatibility conditions in this case. Fans with exactly three 1-dimensional cones are the only ones in $\R^2$ which do not satisfy the pairwise compatibility property for first factors. 
\end{exmp}

\begin{lem} \label{lem:lastfactorsubarr}
    Let $\Sigma$ be a fan in $\R^2$. Then $\Cfrak(\Sigma, \Pfrak)$ satisfies the pairwise compatibility of first (resp. last) factors if and only there is no set of three pairwise compatible first (resp. last) factors.
\end{lem}
\begin{proof}
    Any category trivially satisfies the pairwise compatibility of first (resp. last) factors for $k=2$ in \cref{prop:igusacriteria}. For $k\geq 3$, there is no morphism of rank 3 in $\Cfrak(\Sigma, \Pfrak)$ and hence pairwise compatibility condition is equivalent to there being no set of three pairwise compatible first (resp. last) factors. For $k \geq 4$, it is geometrically impossible to have four compatible first (resp. last) factors in $\R^2$ so the result follows.
\end{proof}

We now show that the functor to the picture group is faithful for all fans in $\R^2$, which is possible since there exists only one relation in the group presentation.

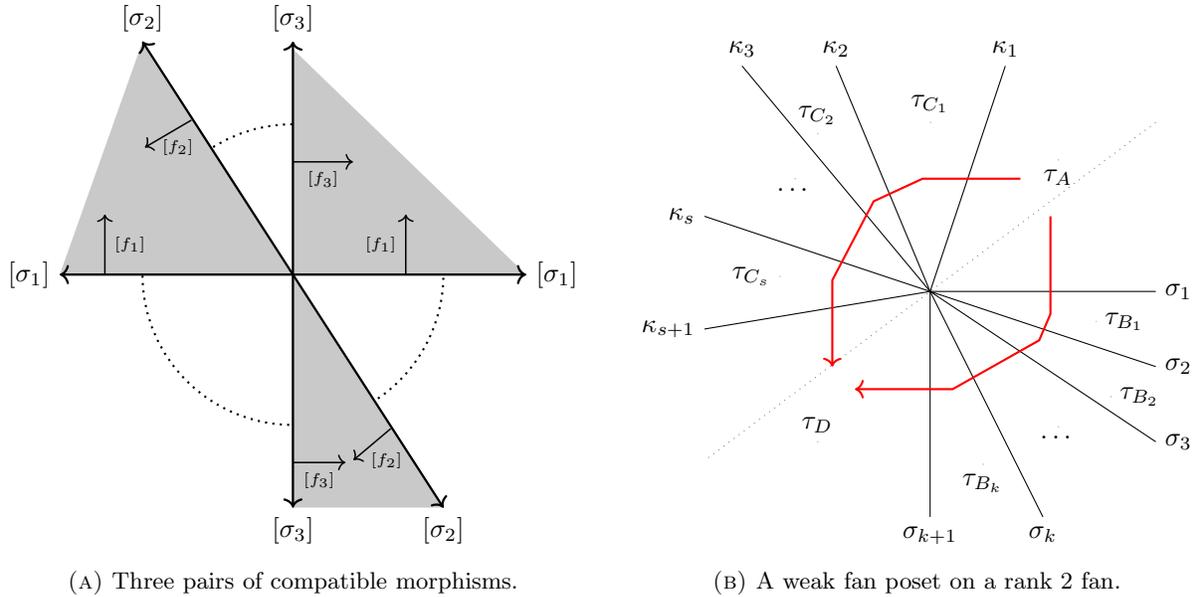
\begin{figure}[ht!]
\centering
\begin{subfigure}{.5\textwidth}
     \centering
        \begin{tikzpicture}
        \draw[dotted, thick] (0,0) circle (2.0);
          \draw[fill=gray!42] (0,3) -- (0,0) -- (3.1,0);
          \draw[fill=gray!42] (-3.1,0) -- (0,0) -- (-2,3.1);
          \draw[fill=gray!42] (0,-3.1) -- (0,0) -- (2,-3.1);
          \draw[line width=0.3mm,->] (0, 0) -- (0, 3.1) node[above] {$[\sigma_3]$};
          \draw[line width=0.3mm,->] (0, 0) -- (0, -3.1) node[below] {$[\sigma_3]$};
          \draw[line width=0.3mm,->] (0,0 ) -- (3.1,0) node[right] {$[\sigma_1]$};
          \draw[line width=0.3mm,->] (0,0 ) -- (-3.1,0) node[left] {$[\sigma_1]$};
          \draw[line width=0.3mm,->] (0,0 ) -- (2,-3.1) node[below] {$[\sigma_2]$};
          \draw[line width=0.3mm,->] (0,0 ) -- (-2,3.1) node[above] {$[\sigma_2]$};
          \draw[line width=0.2mm,->] (1.5,0) -- (1.5,0.8) node[midway,left] {\tiny$[f_1]$};
          \draw[line width=0.2mm,->] (0,1.5) -- (0.8,1.5) node[midway,below] {\tiny$[f_3]$};
          \draw[line width=0.2mm,->] (-2.5,0) -- (-2.5,0.8) node[midway,right] {\tiny$[f_1]$};
          \draw[line width=0.2mm,->] (0,-2.5) -- (0.7,-2.5) node[midway,below] {\tiny$[f_3]$};
          \draw[line width=0.2mm,->] (-1.35,2.055) -- (-1.969185,1.69181) node[right=3pt] {\tiny$[f_2]$};
          \draw[line width=0.2mm,->] (1.3,-2.05) -- (0.804315,-2.46839) node[right=3pt] {\tiny$[f_2]$};
          \filldraw[black] (1.2,1) circle (0pt) node[above]{$\tau_1$};
          \filldraw[black] (-1.3,0.6) circle (0pt) node[above]{$\tau_3$};
          \filldraw[black] (0.5,-2) circle (0pt) node[above]{$\tau_2$};
        \end{tikzpicture}
      \caption{Three pairs of compatible morphisms.}
      \label{fig:lastfactorprob1}
\end{subfigure}%
\begin{subfigure}{.5\textwidth}
  \centering
  \begin{tikzpicture}
      \draw[-] (0, 0) -- (1, 3) node[above] {$\kappa_1$};
      \draw[-] (0, 0) -- (0, -3) node[below] {$\sigma_{k+1}$};
      \draw[-] (0,0 ) -- (3,0 ) node[right] {$\sigma_1$};
      \draw[-] (0,0 ) -- (-3,-0.5) node[left] {$\kappa_{s+1}$};
      \draw[-] (0,0 ) -- (3,-1) node[right] {$\sigma_2$};
      \draw[-] (0,0 ) -- (3,-2) node[right] {$\sigma_3$};
      \draw[-] (0,0 ) -- (1.5,-3) node[below] {$\sigma_k$};
      \filldraw[black] (1.7,-1.8) circle (0pt) node[below]{$\dots$};
      \draw[-] (0,0 ) -- (-1.25,3) node[above] {$\kappa_2$};
      \draw[-] (0,0 ) -- (-2.5,3) node[above] {$\kappa_3$};
      \draw[-] (0,0 ) -- (-3,1.) node[left] {$\kappa_s$};
      \filldraw[black] (-1.8,1.5) circle (0pt) node[below]{$\dots$};
      \filldraw[black] (1.7,1.75) circle (0pt) node[below]{$\tau_A$};
      \filldraw[black] (-1.5,-2) circle (0pt) node[above]{$\tau_{D}$};
      \filldraw[black] (2.2,-0.4) circle (0pt) node[right]{$\tau_{B_1}$};
      \filldraw[black] (2.4,-1.4) circle (0pt) node[right]{$\tau_{B_2}$};
      \filldraw[black] (0.7,-2.3) circle (0pt) node[below]{$\tau_{B_k}$};
      \filldraw[black] (0,2.25) circle (0pt) node[above]{$\tau_{C_1}$};
      \filldraw[black] (-1.5,2.1) circle (0pt) node[above]{$\tau_{C_2}$};
      \filldraw[black] (-2,0.2) circle (0pt) node[left]{$\tau_{C_s}$};
      \draw [red, thick, ->] (1.6,1) -- (1.6,-0.3) -- (1.45,-0.65) -- (0.3,-1.3) -- (-1, -1.3);
      \draw [red, thick, ->] (1.2, 1.5) -- (-0.1,1.5) -- (-0.75, 1.2) -- (-1.3, 0.15) -- (-1.3, -1);
       \draw[-, dotted, gray] (3,2.25) -- (-3,-2.25) node[above] {};
    \end{tikzpicture}
    \caption{A weak fan poset on a rank 2 fan.}
    \label{fig:polygonrel}
\end{subfigure}%
\caption{Last factors and fan poset of fans in $\R^2$.}
\end{figure}

\begin{thm} \label{thm:rank2kpi1}
    Let $\Sigma$ be a fan in $\R^2$ and $\Pcal$ be a non-degenerate weak fan poset, then the functor of \cref{eq:functor} is faithful. Moreover, if $\Cfrak(\Sigma, \Pfrak)$ does not admit a set of three pairwise compatible rank 1 morphisms, then $\Bcal \Cfrak(\Sigma, \Pfrak)$ is a $K(\pi,1)$.
\end{thm}
\begin{proof}
    By \cref{lem:isopicgroups} we may choose one particular non-degenerate fan poset defined as follows: Choose a base region $\tau_A \in \Sigma^2$ and consider the angle bisector of the angle spanned by the two dimension 1 cones defining $\tau_A$. Then let $\tau_D$ be the chamber containing the opposite of the angle bisector. If the opposite of the angle bisector is contained in a cone of codimension 1, then choose either of the adjacent maximal cones as $\tau_D$. 
    %A choice of base region such that the opposite of the angle bisector is not contained in a dimension 1 cone is possible unless the fan is a regular $(2k+1)$-gon i.e. every angle is $\frac{2 \pi}{2k+1}$ for $k\geq 1$. In this case choose either of the adjacent maximal cones as $\tau_D$. 
    The set-up is depicted in \cref{fig:polygonrel}, where the angle bisector is the dotted line and the fan poset is indicated in red and given by the following Hasse diagram:
    
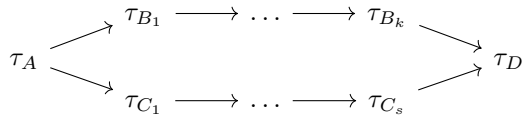
\[ \begin{tikzcd}[row sep=0.1cm] & \tau_{B_1} \arrow[r] & \dots \arrow[r] & \tau_{B_k} \arrow[rd] \\
    \tau_A \arrow[ru] \arrow[rd] & && & \tau_D \\
    & \tau_{C_1} \arrow[r] & \dots \arrow[r] & \tau_{C_s} \arrow[ru] \end{tikzcd} \]
    
    To show that the functor is faithful we want to show that $\Psi_{[\sigma][\tau]}: \Hom([\sigma], [\tau]) \to G(\Sigma, \Pfrak, \Pcal)$ is injective. It is sufficient to show that two distinct morphisms $[f_{\sigma \tau_1}]$ and $[f_{\sigma \tau_2}]$ cannot map to the same group element under $\Psi$, by \cref{cor:morphstartingpoint} and \cref{lem:uniquefaq}. We use the description of the picture group in \cref{defn:picgroup} and since we only have one relation, the generators are distinct. Let us consider the different possible dimensions of $\sigma$ and $\tau$:
    \begin{itemize}
        \item Dimension 0 to 1: For $0 \in \Sigma^0$ and $\lambda \in \Sigma^1$, any morphism $[f_{0 \lambda}] \in \Cfrak(\Sigma, \Pfrak)$ gets mapped to $X_{[\tau_D,\lambda^-]}$, which are easily seen to be distinct, except potentially $X_{[\tau_D,\tau_{C_s}]}$ and $X_{[\tau_D,\tau_{B_k}]}$. However, if the dimension 1 cones $\sigma_{k+1}$ and $\kappa_{s+1}$ give the same generator $X_{[\sigma_{k+1}]}$ then their linear span must be equal, hence the cone $\tau_D$ is a half-plane and not strongly-convex. This is also the reason why it is no problem that $[f_{0 \sigma_{k+1}}]$ and $[f_{0 \kappa_{s+1}}]$ have the same image under $\Psi$, since these two dimension 1 cones cannot be identified.
        \item Dimension 0 to 2: For $0 \in \Sigma^0$ and $\tau \in \Sigma^2$, each such morphism $[f_{0 \tau_i}] \in \Cfrak(\Sigma, \Pfrak)$ maps to $X_{[\tau_D, \tau_i]}$ and we know that these are all distinct elements of $G(\Sigma, \Pfrak, \Pcal)$ from the Dimension 0 to 1 case.
        \item Dimension 1 to 2: For $\sigma \in \Sigma^1$ and $\tau \in \Sigma^2$, there are exactly two possibilities, either $[f_{\sigma, \tau}]\in \Cfrak(\Sigma, \Pfrak)$ maps to $X_{[\sigma]}$ or it maps to the identity element $e$.
    \end{itemize}
    Thus the functor is faithful and by \cref{lem:lastfactorsubarr} and our assumption $\Cfrak(\Sigma, \Pfrak)$ satisfies the pairwise compatibility of last factors. Hence $\Bcal \Cfrak(\Sigma, \Pfrak)$ is a $K(\pi,1)$ by \cref{prop:igusacriteria}.
\end{proof}

\begin{cor} \label{cor:rank2KPG1}
    In the setting of \cref{thm:rank2kpi1}. If $\Pfrak$ identifies all maximal cones, then $\Bcal \Cfrak(\Sigma, \Pfrak)$ is a $K(G(\Sigma, \Pfrak, \Pcal),1)$. 
\end{cor}
\begin{proof}
    In this case, the graph $\Bcal \Cfrak(\Sigma, \Pfrak)$ contains 1 vertex and a loop for every equivalence class $[\sigma] \in \Pfrak$ of a dimension 1 cone. Hence the generators of $\pi_1(\Bcal \Cfrak(\Sigma, \Pfrak))$ coincide with those of $G(\Cfrak, \Pfrak, \Pcal)$. Moreover, the attaching map of the unique 2-cell induces a homotopy which is equivalent to the unique relation of the picture group, similar to \cref{prop:polygonalandmaximal}.
\end{proof}

\begin{exmp}\label{exmp:fanposets}
Our main examples of fan posets are the following:
\begin{enumerate}
    \item The poset of regions of a central simplicial hyperplane arrangement as introduced by \cite{Edelman} is a fan poset by \cite{Reading2005}, a lattice by \cite{BjornerEdelmanZiegler} and polygonal by \cite{Reading2016}. Furthermore, it is easily seen to be non-degenerate. 
    \item The fan poset on the $g$-fan of a finite-dimensional algebra is a polygonal lattice \cite{DIRRT2017} and non-degenerate, see \cref{prop:fanposet}.
    \item Non-degenerate fan posets in $\R^2$ constructed as in the proof of \cref{thm:rank2kpi1}.
\end{enumerate}
\end{exmp}

\subsection{Hyperplane arrangements}\label{subsec:HA}
We briefly recall the theory of hyperplane arrangements, which are collections of subspaces of codimension 1. A \textit{central simplicial hyperplane arrangement} $\Hcal$ in $\R^n$ defines a simplicial fan $\Sigma_\Hcal$ and dissects the space into regions. These closures of these regions are maximal cones of the fan and two regions are called \textit{adjacent} whenever their closures intersect in a cone of codimension 1. Choosing any such region $B$ as the base region and orienting the adjacency graph away from $B$ defines a poset, called the \textit{poset of regions} $\Pcal(\Hcal, B)$ by \cite{Edelman}. For any choice of base region $B$, this defines a fan poset, see \cite[Sec. 4]{Reading2005}. \\

A \textit{flat} of a hyperplane arrangement $\Hcal \subseteq \R^n$ is an intersection of hyperplanes of $\Hcal$, and thus a linear subspace of the ambient space $\R^n$. In particular, the empty intersection gives the ambient space $\R^n$ as a flat. The \textit{support} of a cone $\sigma \in \Sigma_{\Hcal}$ is the smallest flat $s(\sigma)$ which contains $\sigma$. This leads to the \textit{flat-partition} $\Pfrak_{\flatt}$ of the simplicial fan $\Sigma_{\Hcal}$ given by $[\sigma_1]_{\Pfrak_{\flatt}} = [\sigma_2]_{\Pfrak_{\flatt}}$ if and only if $s(\sigma_1) = s(\sigma_2)$. In other words, we identify cones whose support is the same flat.

\begin{prop} \label{prop:flatpartition}
    The partition $\Pfrak_{\flatt}$ is an admissible partition of $\Sigma_{\Hcal}$ and thus the category of the flat-partition $\Cfrak(\Sigma_\Hcal, \Pfrak_{\flatt})$ is a well-defined category.
\end{prop}
\begin{proof}
    Let $\sigma_1, \sigma_2 \in \Sigma_{\Hcal}$ be two cones such that the flat $X \coloneqq s(\sigma_1) = s(\sigma_2)$ is the same intersection of hyperplanes. It follows immediately that
    \[ \spann\{ \sigma_1\}= X = \spann\{ \sigma_2\}. \]
    It follows from \cite[Lem. 1.36]{AguiarMahajan} that both $\starr(\sigma_1)$ and $\starr(\sigma_2)$ are ``equivalent'' to the \textit{arrangement over the flat $X$} whose \textit{essentialisation} is precisely the projection onto the orthogonal complement, see \cite{AguiarMahajan} for more details. The partition is admissible by \cref{lem:possident} since it identifies  the whole set of possible identifications. 
\end{proof}

Shards of hyperplane arrangements were introduced in \cite{Reading2003}. Informally speaking a shard of a hyperplane arrangement $\Hcal \subseteq \R^n$ is a ``piece of a hyperplane''. They are obtained as follows: Let $B$ be a choice of base region and call the $n$ hyperplanes defining it \textit{basic}. Each pair of hyperplanes $H_1, H_2 \in \Hcal$ defines a \textit{subarrangement} $\Hcal(H_1, H_2) \coloneqq \{ H \in \Hcal: H_1 \cap H_2 \subseteq H\}$ with an induced base region. We say that $H_1$ \textit{cuts} $H_2$ if $H_1$ is basic in $\Hcal(H_1,H_2)$ and $H_2$ is not basic in $\Hcal(H_1,H_2)$. For each hyperplane $H \in \Hcal$ remove from $H$ all points contained in hyperplanes $H'$ which cut $H$ in the subarrangement $\Hcal(H, H')$. The closures of the remaining connected components are called \textit{shards}.

\begin{rmk}
    Essentially, shards cut hyperplanes in the same way as stability spaces of bricks define parts of hyperplanes in the wall-and-chamber structure of a finite-dimensional algebra, see \cite{BST2019}. In particular, for preprojective algebras whose wall-and-chamber structure is a hyperplane arrangement \cite{Mizuno2013}, the shards coincide exactly with stability spaces of bricks \cite{Thomas2018} and shards were generalised and this result extended to all finite-dimensional algebras in \cite{Mizuno2022}. For more details see \cref{sec:specialcase}.
\end{rmk} 

Similar to \cite{Reading2011}, we consider the set $\Xi$ of arbitrary intersections of shards, which has a natural poset structure by inclusion with maximal element the empty intersection $\R^n$. Using this set, we define a partition $\Pfrak_{\shard}$ of the cones $\Sigma_{\Hcal}$ of a hyperplane arrangement given by $[\sigma_1]_{\Pfrak_{\shard}}= [\sigma_2]_{\Pfrak_{\shard}}$ if and only if the smallest elements $\xi_i \in \Xi$ which contain $\sigma_i$, for $i=1,2$ respectively, coincide.

\begin{prop} \label{lem:shard-partition}
    Let $\Sigma_{\Hcal}$ be the fan of a finite central simplicial hyperplane arrangement in $\R^n$. The partition $\Pfrak_{\shard}$ is an admissible partition of $\Sigma_{\Hcal}$ and thus the category of the shard-partition $\Cfrak(\Sigma_\Hcal, \Pfrak_{\flatt})$ is a well-defined category.
\end{prop}
\begin{proof}
    Since by definition each intersection of shards is contained in the intersection of corresponding hyperplanes, two such cones $\sigma_1, \sigma_2$ have the same support and hence the proof of \cref{prop:flatpartition} yields that they are in the same class of potential identifications. \\

    To see that $\Pfrak_{\shard}$ is admissible we let $\sigma_1 \sim \sigma_2$ in $\Pfrak_{\shard}$ be distinct cones, and let $\kappa_i \in \starr(\sigma_i)$ for $i=1,2$ be such that $\pi_{\sigma_1}(\kappa_1) = \pi_{\sigma_2}(\kappa_2)$. It follows from \cref{lem:possident} that $\spann\{ \kappa_1\} = \spann \{ \kappa_2\}$ and since $\Sigma_{\Hcal}$ is finite and complete there exists a hyperplane $H$ containing $\kappa_1, \kappa_2$ such that $H = \spann\{\tau\}$ for some $\tau \in \Sigma^{n-1}$. Assume for a contradiction that $\kappa_1 \not \sim \kappa_2$ in $\Pfrak_{\shard}$ that means that the hyperplane $H$ is cut by a hyperplane $H'$ separating $\kappa_1$ and $\kappa_2$ in the sense that $\kappa_1$ and $\kappa_2$ lie in opposite half-spaces defined by $H'$. Since $\sigma_i \subseteq \kappa_i$, the hyperplane $H'$ also separates $\sigma_1$ and $\sigma_2$ and in particular the cut on $H$ induced by $H'$ implies that $\sigma_1$ and $\sigma_2$ lie in different shards, a contradiction.\\

    Finally we need to verify that $\Pfrak_{\shard}$ is well-defined in other words when $\sigma_1$ and $\sigma_2$ are two distinct cones identified in $\Pfrak_{\shard}$ then given cones $\kappa_i, \kappa_i' \in \starr(\sigma_i)$ for $i=1,2$ satisfying
    \[ \pi_{\sigma_1}(\kappa_1)= \pi_{\sigma_2}(\kappa_2), \quad \pi_{\sigma_1} (\kappa_1') = \pi_{\sigma_2}(\kappa_2'),\]
    we have $\kappa_1 \sim \kappa_1'$ if and only if $\kappa_2 \sim \kappa_2'$. Let $\sigma_1, \sigma_2$ be two distinct cones identified in $\Pfrak_{\shard}$ and $\kappa_i, \kappa_i' \in \starr(\sigma_i)$ be as above for $i=1,2$. Assume for a contradiction that $\kappa_1 \sim \kappa_1'$ but $\kappa_2 \not \sim \kappa_2'$ in $\Pfrak_{\shard}$. By assumption every hyperplane $H$ containing $\sigma_1$ also contains $\sigma_2$. Since $\kappa_2 \not \sim \kappa_2'$, there exists a shard $S$ that contains $\kappa_2$ but not $\kappa_2'$, thus there exists a hyperplane $H'$ which separates $\kappa_2$ and $\kappa_2'$. In particular $H'$ cuts the hyperplane $H_S$ corresponding to the shard $S$. By definition if $H'$ separates $\kappa_2$ and $\kappa_2'$ then it must pass through their intersection, in other words it contains $\sigma_2$. Now both $H_S$ and $H$ contain $\sigma_2$ and thus also $\sigma_1$. Since the orthogonal projections $\pi_{\sigma_1}(\starr(\sigma_1))$ and $\pi_{\sigma_2}(\starr(\sigma_2))$ coincide, the hyperplane $H'$ must also separate $\kappa_1$ and $\kappa_1'$, both of which lie on $H_S$. Since $H'$ cuts $H_S$ as it is basic in $\Hcal(H_S, H')$, the cones $\kappa_1$ and $\kappa_1'$ lie in different shards and hence are not identified, a contradiction.
\end{proof}

\begin{lem}\label{lem:allowedaltdef}
    The poset of regions of a central simplicial hyperplane arrangement with any choice of base region is well-defined on identified stars and hence we may use the description of the picture group in \cref{defn:altdefn}.
\end{lem}
\begin{proof}
    This follows from the fact that any poset of regions $\Pcal$ of a simplicial hyperplane arrangement is induced by any linear functional $b \in (\R^n)^\lor$ whose minimum on the unit sphere lies inside the base region of $\Pcal$, see \cite[Thm. 4.2]{Reading2005}, and \cref{lem:linfunctposet}.
\end{proof}

For the maximal partition $\Pfrak_{\max} = \Pfrak_{\flatt}$, equivalence classes $[\sigma]$ of cones $\sigma \in \Sigma_\Hcal^{n-1}$ of codimension 1 are given by all such cones contained in the same hyperplane $H_{\sigma} \coloneqq \spann\{ \sigma\}$. Therefore we may represent these equivalence classes simply with the normal vector $\boldsymbol{n}_{H_\sigma}$ to the hyperplane. More precisely, for every hyperplane $H \in \Hcal$ let $\boldsymbol{n}_{H} \in \R_{\geq 0}^n$ be a unit normal vector to $H$. The following is one of the main results of this paper and the rest of this section will concern its proof. In \cref{sec:specialcase} we discuss its algebraic implications.
\begin{thm} \label{thm:mainthm}
    Let $\Hcal \subseteq \R^n$ be a (central, simplicial) hyperplane arrangement, $B$ a region and $\Pcal_B$ the corresponding poset of regions. Assume that the normal vectors $\boldn_H$ to all hyperplanes $H \in \Hcal$ can be taken to lie in the positive orthant $\R_{\geq 0}^n$. Then the functor from \cref{eq:functor} is faithful.
\end{thm}

Denote by $\Z[[\R^n]]$ the formal power series with generators $\{x^{\boldsymbol{v}} : \boldsymbol{v} \in \R^n\}$ over $\Z$ whose multiplication is given by $x^{\boldsymbol{v}} * x^{\boldsymbol{v}'} = x^{\boldsymbol{v}+ \boldsymbol{v}'}$. This is a commutative associative algebra. The group of units $\Z[[\R^n]]^*$ is therefore an abelian group consisting of all formal sums with constant term equal to 1 or -1. 

\begin{lem}
    In the setting of \cref{thm:mainthm} there exists a group homomorphism 
    \begin{align*}
        \phi: G(\Hcal, \Pfrak_{\max}, \Pcal) &\to \Z[[\R^n]]^* \\
        X_{\boldn_H} &\mapsto 1 + x^{\boldn_H}
    \end{align*} 
\end{lem}
\begin{proof}  
    As pointed out in \cref{exmp:fanposets} the poset of regions $\Pcal$ is a polygonal lattice, so by \cref{prop:polygonalandmaximal} it is sufficient to consider the relations coming from facial intervals of cones of codimension 2, which we call \textit{polygon relations}. To show that the polygon relations are preserved is simple because of the commutativity of $\Z[[\R^n]]^*$ and the hyperplane structure. In particular, because we have a hyperplane arrangement and take the maximal partition, the labels of the two disjoint chains in any polygon of $\Pcal$ correspond to the two sequences $(H_1,\dots, H_r)$ and $(H_r, \dots, H_1)$. Thus we simply obtain
    \begin{align*}
        \phi(X_{\boldn_{H_1}} \dots X_{\boldn_{H_r}}) &= (1 + x^{\boldn_{H_1}}) * \dots * (1 + x^{\boldn_{H_r}}) \\
        & = (1 + x^{\boldn_{H_r}}) * \dots * (1 + x^{\boldn_{H_1}}) \\
        & = \phi(X_{\boldn_{H_r}} \dots X_{\boldn_{H_1}})
    \end{align*}
    since $\Z[[\R^n]]^*$ is abelian.
\end{proof}

\begin{rmk}
    In \cite[Sec. 4.3]{HansonIgusa2021} and \cite[Sec. 5.2]{HansonIgusaPW2SMC} the authors take a similar approach and find a group homomorhism into the groups of units of different versions of Hall algebras. However, in our case the hyperplane arrangement gives the polygon relations a nicely symmetric structure, so that the abelian group $\Z[[\R^n]]^*$ is seems a natural candidate. 
\end{rmk}

\begin{cor}
    For $H_1 \neq H_2 \in \Hcal$ we have $\boldn_{H_1} \neq \boldn_{H_2}$ and hence $e \neq X_{\boldsymbol{n}} \neq X_{\boldsymbol{m}} \in G(\Hcal, \Pfrak_{\max}, \Pcal_B)$.
\end{cor}
\begin{proof}
    This immediately follows from the group homomorphism as $\phi(X_{\boldsymbol{n}}) = 1 + x^{\boldsymbol{n}_{H_1}} \neq 1 + x^{\boldsymbol{n}_{H_2}} = \phi(X_{\boldsymbol{m}})$, i.e. they get mapped to distinct elements in $\Z[[\R^n]]^*$.
\end{proof}

Since the poset of regions is non-degenerate by \cref{lem:allowedaltdef}, we use the alternative definition of the picture group given in \cref{defn:altdefn} and consider its additional generators in the following. 
\begin{lem}\label{lem:distinctgens} 
    In the setting of \cref{thm:mainthm}, choose all normal vector $\boldn_H$ to lie in the positive orthant. Let $R_1 \neq R_2 \in \R^n \setminus \Hcal$ be distinct regions, then $g_{\overline{R}_1} \neq g_{\overline{R}_2}$ in $G(\Hcal, \Pfrak_{\max}, \Pcal_B)$.
\end{lem}
\begin{proof}
    Consider the maximal cone $R_1 \land R_2$ in $\Pcal_B$ and let $(\boldn_{H_1}, \dots, \boldn_{H_s})$ label a path $R_1 \to R_1 \land R_2$ and let $(\boldn_{H_1'}, \dots, \boldn_{H_r'})$ label a path $R_2 \to R_1 \land R_2$ in $\Pcal$. Assume for a contradiction that $g_{R_1} = g_{R_2}$. If $\max(r,s) > 0$, let $\boldsymbol{m} \in \{\boldn_{H_1}, \dots, \boldn_{H_s}, \boldn_{H_1'}, \dots, \boldn_{H_r'})$ be a vector such that the sum of its entries is minimal among them, then the assumption $g_{\overline{R}_1} = g_{\overline{R}_2}$ implies
    \[ X_{\boldn_{H_1}} \dots X_{\boldn_{H_s}} g_{\overline{R_1 \land R_2}} = g_{\overline{R}_1} = g_{\overline{R}_2} = X_{\boldn_{H_1'}} \dots X_{\boldn_{H_r'}} g_{\overline{R_1 \land R_2}}.\]
    Multiplying both sides by $g_{\overline{R_1 \land R_2}}^{-1}$ on the right and applying $\phi$ yields
    \[ (1 + x^{\boldn_{H_1}}) * \dots * (1 + x^{\boldn_{H_s}}) = (1 + x^{\boldn_{H_1'}}) * \dots * (1 + x^{\boldn_{H_r'}}).\]
    Therefore $x^{\boldsymbol{m}}$ arises with positive coefficient on one of the sides. Since the sum of the entries of $\boldsymbol{m}$ is minimal (and importantly all normal vectors have non-negative entries) $x^{\boldsymbol{m}}$ cannot be written as the product of other terms (since ${\boldsymbol{m}}$ cannot be written as a sum of $\boldn_{H_i}$). Hence it must appear on both sides of the equality. It follows that $\bx$ separates both $R_1$ and $R_2$ from $R_1 \land R_2$, hence
    \[ R_1 \land R_2 < R_{\bx} < R_1, R_2\]
    where $R_{\bx}$ is the minimal (in $\Pcal$) region whose separating set contains $S(R_1 \land R_2) \cup \{ \bx\}$. This contradicts the fact that $R_1 \land R_2$ is the maximal lower bound, hence $\max\{r,s\} = 0$ and $R_1 = R_2$. 
\end{proof}

\begin{proof}[Proof of \cref{thm:mainthm}]
    Consider the functor
    \begin{equation}
    \begin{aligned}
        \Psi: \Cfrak(\Sigma, \Pfrak) &\to G(\Sigma, \Pfrak, \Pcal)  \\
        [f_{\sigma \kappa}] &\mapsto g_{\kappa^-} \cdot g_{\sigma^-}^{-1}
    \end{aligned}
    \end{equation}
    which is simply an alternative way of writing the one in \cref{eq:functor} using the presentation of the picture group given in \cref{defn:altdefn}, which is possible by \cref{lem:allowedaltdef}. We need to show that the induced function $\Psi_{\sigma \kappa}$ from $\Hom_{\Cfrak(\Sigma_{\Hcal}, \Pfrak_{\flatt})}([\sigma],[\kappa])$ to $ \Hom_{G(\Sigma_{\Hcal}, \Pfrak_{\flatt}, \Pcal)}(\bullet, \bullet)$ is injective. By \cref{cor:morphstartingpoint} and \cref{lem:uniquefaq}, it suffices to prove this for one representative $\sigma \in [\sigma]$. Hence, take distinct morphisms $[f_{\sigma \kappa_1}] \neq [f_{\sigma \kappa_2}] \in \Hom_{\Cfrak(\Sigma_{\Hcal}, \Pfrak_{\flatt})}([\sigma],[\kappa])$ for which we have $\kappa_1 \neq \kappa_2$ but $\kappa_1 \sim \kappa_2$. Then we have
    \[ \Psi([f_{\sigma \kappa_1}]) = g_{\kappa_1^-} \cdot g_{\sigma^-}^{-1} \quad \text{and} \quad \Psi([f_{\sigma \kappa_2}]) = g_{\kappa_2^-} \cdot g_{\sigma^-}^{-1}.\]
    Then \cref{lem:distinctgens} implies that these group elements coincide if and only if $\kappa_1^- = \kappa_2^-$. This means that $\kappa_1^- = \kappa_2^- \in \starr(\kappa_1) \cap \starr(\kappa_2) \neq \emptyset$ which is a contradiction as $\spann\{ \kappa_1\} = \spann\{ \kappa_2\}$ but $\kappa_1 \neq \kappa_2$ implies that their stars may not intersect. Therefore the functor is faithful. 
\end{proof}

\begin{cor}
    If $\Cfrak(\Sigma_{\Hcal}, \Pfrak_{\flatt})$ satisfies the pairwise compatibility of first and last factors, then the classifying is a $K(\pi,1)$ for $\pi$ the picture group.
\end{cor}

\section{Lattice of admissible partitions} \label{sec:lattice}

It is well-known that the collection of partitions of a set, ordered by refinement, forms a complete lattice, see \cite[Sec. IV.4]{Gratzer}. We show that the restriction to admissible partitions of a fan preserves the lattice structure, thus establishing a lattice of categories (one for each admissible partition) of a fan. In the special case where the underlying fan is the $g$-fan of a finite-dimensional algebra, see \cref{sec:specialcase}, the $\tau$-cluster morphism category is an element of this lattice. Similarly if the underlying fan is a hyperplane arrangement, the category of the flat-partition and the category of the shard-partition lie in this lattice. We begin by introducing some basic notions. It is a well-known that an equivalence relation induces a partition on a set $X$ and vice versa. Let us define the following poset relation on partitions. 
\begin{defn}
    Let $P_1$ and $P_2$ be partitions of $X$. We say that $P_1$ is a \textit{finer} partition than $P_2$ if
    \[ x \sim_{P_1} y \Longrightarrow x \sim_{P_2} y, \quad \text{or equivalently,} \quad \{x,y\} \subseteq a \in P_1 \Longrightarrow \{x,y\} \subseteq b \in P_2\]
    for  $x,y  \in X$ and some $a \in P_1$ and $b \in P_2$. In this case, we write $P_1 \leq P_2$ and say $P_2$ is \textit{coarser} than $P_1$.
\end{defn}

Denote by $\Part(X)$ the set of all partitions of a set $X$. In view of \cref{defn:thecategory}, denote by $\APart(\Sigma)$ the set of all admissible partitions of a fan $\Sigma$. The definitions of join and the meet do not depend on the restriction to admissible partitions and are the operations defined for a general \textit{partition lattice}. Given a subset $\Scal =\{ P_1, \dots, P_k\} \subseteq \Part(X)$, define the meet $\bigwedge \Scal$ to be the partition satisfying 
\[x \sim_{\bigwedge \Scal} y \quad \text{if and only if} \quad x \sim_{P_i} y \text{ for all $i =1, \dots, k$}. \] 

Define the join $\bigvee \Scal$ to be the partitioned satisfying $x \sim_{\bigvee \Scal} y$ if and only if there exists a natural number $n$, indices $i_0, \dots, i_{n-1} \in \{1, \dots, k\}$ and $x_0, \dots, x_{n+1} \in X$ such that $x= x_0$, $y = x_{n+1}$ and $x_j \sim_{P_{i_j}} x_{j+1}$ for $0 \leq j \leq n$. 

\begin{prop} \label{prop:lattice}
    The partially ordered set $\APart(\Sigma)$ forms a complete lattice. 
\end{prop}
\begin{proof}
    Since $\Part(\Sigma)$ is a complete lattice, we need to show that $\APart(\Sigma)$ is closed under the lattice operations. Recall from \cref{sec:definition} that we need to show the join and meet partitions are \textit{possible} and \textit{admissible}. Let us start with the meet. Take a subset $\Scal = \{ \Pfrak_1, \dots, \Pfrak_k\} \subseteq \APart(\Sigma)$ of admissible partitions and take a block $\Ecal_{\sigma}^\ell \in \bigwedge \Scal$. If $|\Ecal_{\sigma}^\ell |=1$, then there is nothing to show, so assume $|\Ecal_{\sigma}^\ell|> 1$ and take $\sigma_1, \sigma_2 \in \Ecal_{\sigma}^\ell$. By definition $\sigma_1 \sim_{\Pfrak_i} \sigma_2$ for all $i =1,\dots, k$, so the partition $\bigwedge \Scal$ consists of possible identification. Similarly, whenever $\pi_{\sigma_1}(\tau_1) = \pi_{\sigma_2}(\tau_2)$, then $\tau_1 \sim_{\Pfrak_i} \tau_2$ $i=1,\dots, k$ and thus $\tau_1 \sim_{\bigwedge \Scal} \tau_2$, hence $\bigwedge \Scal$ is an admissible partition. \\

    For the join, take $\Ecal_{\sigma}^\ell \in \bigvee \Scal$ and $\sigma_1, \sigma_2 \in \Ecal_{\sigma}^\ell$, then if both $\sigma_1$ and $\sigma_2$ are contained in one block $\Ecal_{\lambda}^h \in \Pfrak_i$ for some $i \in \{1, \dots, k\}$ the result is immediate. Therefore assume that $\sigma_1$ and $\sigma_2$ are not contained in the same block in any $\Pfrak_i$. However, by the construction of the join, there exists a sequence 
    \[\sigma_1 \sim_{\Pfrak_{i_1}} \sigma_{i_1} \sim_{\Pfrak_{i_2}} \dots \sim_{\Pfrak_{i_{r-1}}} \sigma_{i_{r-1}} \sim_{\Pfrak_{i_{r}}} \sigma_2\]
    for some $r \geq 1$, such that each term is contained in $\Ecal_{\sigma_1}= \Ecal_{\sigma_2}$, hence $\bigvee \Scal$ consists of possible identifications. This sequence of possible identifications implies that we have 
    \[ \pi_{\sigma_1}(\starr(\sigma_1)) = \pi_{\sigma_{i_j}}(\starr(\sigma_{i_j})) = \pi_{\sigma_2}(\starr(\sigma_2)), \]
    for $j=1, \dots, r-1$. To show that $\bigvee \Scal$ is admissible, assume we have $\pi_{\sigma_1}(\tau_1) = \pi_{\sigma_2}(\tau_2)$ for some $\tau_1 \in \starr(\sigma_1)$ and $\tau_2 \in \starr(\sigma_2)$, then from the above we must have a sequence of $\tau_{i_j} \in \starr(\sigma_{i_j})$ for $j=1, \dots, r-1$ such that
    \[ \pi_{\sigma_1}(\tau_1) = \pi_{\sigma_{i_j}}(\tau_{i_j}) = \pi_{\sigma_2}(\tau_2). \]
    Since each $\Pfrak_{i_j}$ is admissible for $j = 1, \dots, r$ we obtain a sequence $\tau_1 \sim_{\Pfrak_{i_1}} \tau_{i_1} \sim_{\Pfrak_{i_2}} \dots \sim_{\Pfrak_{i_{r-1}}} \tau_{i_{r-1}} \sim_{\Pfrak_{i_r}} \tau_2$ as required. Hence $\bigvee \Scal$ is an admissible partition.
\end{proof}

It is easily seen that this is a bounded lattice, that is, there exists a maximal and a minimal element. The minimal element is the partition $\Pfrak_{\text{poset}}$ mentioned earlier, which is just the finest partition with trivial equivalence classes. The maximal element is the coarsest partition, whose equivalence classes are exactly the $\Ecal_{\sigma}$ of \cref{defn:thecategory}. When the fan is a hyperplane arrangement the coarsest partition is exactly the flat-partition, see \cref{prop:flatpartition}, whereas the shard-partition, see \cref{lem:shard-partition}, sits somewhere in the middle. We conclude this section by describing the relationship between comparable partitions.

\begin{thm} \label{thm:latticeproperties}
    Let $(\Sigma, \Pcal)$ be a fan poset and $\Pfrak_1, \Pfrak_2$ be two admissible partitions of $\Sigma$ such that $\Pfrak_1$ is finer than $\Pfrak_2$. Then the following hold:
    \begin{enumerate}
        \item There exists a faithful surjective-on-objects functor $F: \Cfrak(\Sigma, \Pfrak_1) \to \Cfrak(\Sigma, \Pfrak_2)$.
        \item The classifying spaces satisfy $\Bcal \Cfrak(\Sigma, \Pfrak_2) \equiv \Bcal \Cfrak(\Sigma, \Pfrak_1)/\sim$, where $\sim$ identifies the cells $e([\sigma_1]_{\Pfrak_1})$ and $e([\sigma_2]_{\Pfrak_1})$ whenever $[\sigma_1]_{\Pfrak_2}= [\sigma_2]_{\Pfrak_2}$.
        \item If $(\Sigma, \Pcal)$ is non-degenerate, then the picture groups satisfy $G(\Sigma, \Pfrak_2, \Pcal) \cong G(\Sigma, \Pfrak_1, \Pcal)/I$, where $I$ is generated by the relations $X_{[\sigma_1]} = X_{[\sigma_2]}$ whenever $[\sigma_1]_{\Pfrak_1} \neq [\sigma_2]_{\Pfrak_1}$ but $[\sigma_1]_{\Pfrak_2}= [\sigma_2]_{\Pfrak_2}$.
    \end{enumerate}
\end{thm}
\begin{proof}
    \begin{enumerate}
        \item The functor $F: \Cfrak(\Sigma, \Pfrak_1) \to \Cfrak(\Sigma, \Pfrak_2)$ is given on objects by the map sending $[\sigma]_{\Pfrak_1} \mapsto [\sigma]_{\Pfrak_2}$ for all cones $\sigma \in \Sigma$ and on morphisms by sending $[f_{\sigma \tau}]_{\Pfrak_1} \mapsto [f_{\sigma \tau}]_{\Pfrak_2}$ for all $\sigma, \tau \in \Sigma$. Objects of $\Cfrak(\Sigma,\Pfrak_i)$ are exactly the blocks of $\Pfrak_i$ for both $i=1,2$. By definition, $\Pfrak_2$ being coarser than $\Pfrak_1$ means $[\sigma]_{\Pfrak_1} \subseteq [\sigma]_{\Pfrak_2}$. Hence $\Pfrak_2$ has at most as many blocks as $\Pfrak_1$. Thus there is a surjection on objects of the categories. \\

        To see that the functor is faithful, one observes that the identification of morphisms in the construction of the category depends not on a choice but is induced when certain cones are identified. Take for example two different morphisms $[f_{\sigma_1 \tau_1}]_{\Pfrak_1} \neq [f_{\sigma_2 \tau_2}]_{\Pfrak_1} \in \Hom_{\Cfrak(\Sigma, \Pfrak_1)}([\sigma]_{\Pfrak_1}, [\tau]_{\Pfrak_1})$ satisfying $\sigma_1 \neq \sigma_2$ or $\tau_1 \neq \tau_2$ but $\sigma_1 \sim_{\Pfrak_1} \sigma_2$ and $\tau_1 \sim_{\Pfrak_1} \tau_2$. Assume for a contradiction that $[f_{\sigma_1 \tau_1}]_{\Pfrak_2} = [f_{\sigma_2 \tau_2}]_{\Pfrak_2}$ then the cones must satisfy $\pi_{\sigma_1}(\tau_1) = \pi_{\sigma_2}(\tau_2)$ by definition. However, then they should have also been identified in $\Pfrak_1$, a contradiction. Hence the induced map
        \[ F_{[\sigma],[\tau]}: \Hom_{\Cfrak(\Sigma, \Pfrak_1)}([\sigma]_{\Pfrak_1}, [\tau]_{\Pfrak_1}) \to \Hom_{\Cfrak(\Sigma, \Pfrak_2)}([\sigma]_{\Pfrak_2}, [\tau]_{\Pfrak_2})\]
        is injective and hence the functor faithful. 
        \item This follows from the description of the CW-complex in \cref{thm:CWcomplex} whose cells are the union of factorisation cubes, and the fact that the factorisation cubes of two identified morphisms get identified, by \cref{lem:uniquefaq}.
        \item Let $\kappa_1, \kappa_2, \kappa_3 \in \Sigma$. The quotient group homomorphism $H: G(\Sigma, \Pfrak_1, \Pcal) \to G(\Sigma, \Pfrak_1, \Pcal)/I$ ``identifies'' $X_{[\kappa_1]}, X_{[\kappa_2]} \in G(\Sigma, \Pfrak_1, \Pcal)$ whenever $\kappa_1 \sim_{\Pfrak_2} \kappa_2$ but $\kappa_1 \not \sim_{\Pfrak_1} \kappa_2$. More precisely, we have that $X_{[\kappa_1]} + I = X_{[\kappa_2]} + I \in G(\Sigma, \Pfrak_1, \Pcal)/I$. Furthermore if $\kappa_3 \sim_{\Pfrak_1} \kappa_1$, then also $X_{[\kappa_3]} + I = X_{[\kappa_2]} + I \in G(\Sigma,\Pfrak_1, \Pcal)/I$. It is easy to see that we have $X_{[\kappa_i]} + I = X_{[\kappa_j]} + I \in G(\Sigma, \Pfrak_1, \Pcal)/I$ if and only if $X_{[\kappa_i]} = X_{[\kappa_j]} \in G(\Sigma, \Pfrak_2, \Pcal)$ for any $\kappa_i, \kappa_j \in \Sigma$. The first kind of relations in \cref{defn:picgroup} hold since they only depend on the shared fan poset $(\Sigma, \Pcal)$ and the second kind relations trivially hold in $G(\Sigma, \Pfrak_1, \Pcal)$, thus its quotient and $G(\Sigma, \Pfrak_2, \Pcal)$ by \cref{lem:nondegimplystars} since the fan poset is non-degenerate.
    \end{enumerate}
\end{proof}

\begin{cor}\label{cor:faithfulfunct}
    Let $(\Sigma, \Pcal)$ be a fan poset and $\Pfrak_1 \leq \Pfrak_2$ be two admissible partitions of $\Sigma$. If there exists a faithful group functor $H: \Cfrak(\Sigma, \Pfrak_2) \to G(\Sigma, \Pfrak_2, \Pcal)$, then there exists a faithful group functor $\Cfrak(\Sigma, \Pfrak_1) \to G(\Sigma, \Pfrak_2, \Pcal)$.
\end{cor}
\begin{proof}
    The functor is given by the composition $H \circ F$, where $F: \Cfrak(\Sigma, \Pfrak_1) \to \Cfrak(\Sigma, \Pfrak_2)$ is the faithful functor of \cref{thm:latticeproperties} (1). In particular, let $[\sigma]_{\Pfrak_1}, [\tau]_{\Pfrak_1} \in \Cfrak(\Sigma, \Pfrak_1)$, then
    \[ \Hom_{\Cfrak(\Sigma, \Pfrak_1)}([\sigma]_{\Pfrak_1}, [\tau]_{\Pfrak_1}) \xhookrightarrow{F} \Hom_{\Cfrak(\Sigma, \Pfrak_2)}([\sigma]_{\Pfrak_2}, [\tau]_{\Pfrak_2}) \xhookrightarrow{H} \Hom_{G(\Sigma, \Pfrak_2)}(\bullet,\bullet)\]
    is a composition of injective morphisms between $\Hom$-sets and thus injective. Hence $G \circ F$ is faithful. 
\end{proof}

\begin{cor} \label{cor:HAallfaithful}
    In the setting of \cref{thm:mainthm}, the category $\Cfrak(\Sigma_\Hcal, \Pfrak)$ admits a faithful functor to the group $G(\Sigma_\Hcal, \Pfrak_{\flatt}, \Pcal)$ for all admissible partitions $\Pfrak$.
\end{cor}
\begin{proof}
    Since $\Pfrak_{\flatt}$ is the maximal partition any other partition $\Pfrak$ is finer, i.e. $\Pfrak \leq \Pfrak_{\flatt}$. The result then follows from \cref{cor:faithfulfunct} and \cref{thm:mainthm}.
\end{proof}

\section{Applications to the $\tau$-cluster morphism category}
\label{sec:specialcase}

In the representation theory of finite-dimensional algebras, the Auslander-Reiten translation $\tau$ defines an important relationship within the category of (finitely generated) modules of an algebra. Adachi, Iyama and Reiten \cite{AIR2014} introduced $\tau$-tilting theory as a completion of classical tilting theory from the viewpoint of mutations. We refer to \cite{bluebookV1} for a general introduction to the representation theory of finite-dimensional algebras and to \cite{Hipolitosurvey} for a survey of $\tau$-tilting theory. While the first sections were purely geometric, we now reintroduce the algebraic tools of $\tau$-tilting theory to study the $\tau$-cluster morphism category of a finite-dimensional algebra as a special case of the category of a partitioned fan. \\

We begin by giving the necessary background on $\tau$-tilting theory and provide a new algebraic proof of the connection between a finite-dimensional algebra and its $g$-fan. Then we show that there is a natural fan poset on the $g$-fan induced by the torsion classes and that our definition of the picture group generalises that of \cite{HansonIgusa2021} and \cite{IgusaTodorovWeyman2016}. We conclude by giving an application of the previous sections and prove that $\tau$-cluster morphism category $\Cfrak(A)$ admits a faithful functor to some group whenever its $g$-fan $\Sigma(A)$ is a hyperplane arrangement.

\subsection{$\tau$-tilting theory}
Let $\mods A$ be the category of finitely generated right $A$-modules of a finite-dimensional basic algebra over an arbitrary field $K$ and $\proj A$ be the full subcategory of projective modules. Write $\bdim M$ for the dimension vector of a module $M \in \mods A$, defined such that $(\bdim M)_i \coloneqq \dim_{\End_A(S(i))}(Me_i)$, where $\{e_i\}_{i=1}^{n}$ is a set of primitive orthogonal idempotents. Let $|M|$ be the number of non-isomorphic indecomposable direct summands of $M$. Assume that $|A|=n$ so that we may write $A = \bigoplus_{i=1}^n P(i)$, where $P(i) \in \proj A$ are the projective covers of the simple modules. The focal objects of $\tau$-tilting theory are the $\tau$-rigid modules. 

\begin{defn}
    Let $M \in \mods A$ and $P \in \proj A$, then 
    \begin{enumerate}
        \item $M$ is \textit{$\tau$-rigid} if $\Hom_A(M, \tau M) = 0$ and \textit{$\tau$-tilting} if additionally $|M|=n$.
        \item $(M,P)$ is a \textit{(support) $\tau$-rigid pair} if $M$ is $\tau$-rigid and $\Hom_A(P,T)=0$,  and a \textit{(support) $\tau$-tilting pair} if additionally $|M|+|P|=n$.
        \item $M$ is \textit{$\tau^{-1}$-rigid} if $\Hom_A(\tau^{-1} M, M)=0$.
    \end{enumerate}
\end{defn}
Denote by $\strigid A$ (resp. $\sttilt A$) the $\tau$-rigid (resp. $\tau$-tilting) pairs of $A$. For a module $M \in \mods A$, denote by $\add M$ (resp. $\Fac M$, $\Sub M$) the full subcategory of $\mods A$ consisting of all direct summands (resp. factor modules, submodules) of finite direct sums of copies of $M$. Furthermore, define
\[ M^\perp \coloneqq \{ N \in \mods A: \Hom_A(M,N) =0 \}, \quad {}^\perp M \coloneqq \{ N \in \mods A: \Hom_A(N,M) = 0\}. \]
\begin{prop} \label{lem:homtoext} \label{prop:HomToExt2} \cite[Prop. 5.6 and 5.8]{AuslanderSmalo1980}
Let $M,N \in \mods A$, then 
\begin{enumerate}
    \item $\Hom_A(M, \tau N) = 0$ if and only if $\Ext_A^1(N, \Fac M)=0$, and
    \item $\Hom_A(\tau^{-1}M, N) = 0$ if and only if $\Ext_A^1(\Sub N, M)=0$.
\end{enumerate}
\end{prop}

To every module, we associate the following integer vector.
\begin{defn}
    Let $M \in \mods A$. Let $P_{-1} \to P_0 \to T \to 0$ be a minimal projective presentation, where $P_0 = \bigoplus_{i=1}^n P(i)^{a_i}$ and $P_{-1}= \bigoplus_{i=1}^n P(i)^{b_i}$. Define the \textit{$g$-vector} of $M$ as
    \[ g^M \coloneqq (a_1 - b_1, a_2 -b_2, \dots, a_n - b_n).\]
    Furthermore, define the $g$-vector of a $\tau$-rigid pair $(M,P)$ as $g^{(M,P)} = g^M - g^P$. In this way we associate a polyhedral cone $\overline{\Ccal}_{(M,P)} \coloneqq \cone\{ g^{M_1}, \dots, g^{M_k}, - g^{P_{k+1}}, \dots, - g^{P_t}\}$ to each $\tau$-rigid pair $(M,P)$, where $\{M_i\}_{i=1}^k$ and $\{P_j\}_{j=k+1}^t$ are the indecomposable direct summands of $M$ and $P$ respectively.
\end{defn}

The $g$-vectors of $\tau$-rigid pairs of an algebra are important examples of \textit{stability conditions} in the sense of King \cite{King1994}, who translated the Geometric Invariant Theory of Mumford \cite{Mumford65} to quiver representations. These stability conditions define a \textit{stability scattering diagram} \cite{Bridgeland2017} whose support is the \textit{wall-and-chamber structure} of the algebra \cite{BST2019}. Let $\langle \cdot, \cdot\rangle$ denote the standard inner product on $\R^n$. 

\begin{defn}\cite{King1994} Let $M \in \mods A$, take $v \in \R^n$. We say $M$ is \textit{$v$-semistable} if $\langle v, \bdim M \rangle = 0$ and if for any nonzero proper subobject $L$ of $M$ we have $\langle v, \bdim L \rangle \leq 0$. We say $M$ is \textit{$v$-stable} if these inequalities are strict.
\end{defn}

\begin{defn}
Fix a nonzero module $M \in \mods A$, then the \textit{stability space} $\Dcal(M)$ of $M$ is 
\[ \Dcal(M) = \{ v \in \R^n: M \text{ is $v$-semistable}\} \subseteq \R^n. \]
We say $\Dcal(M)$ is a \textit{wall} if $\codim \Dcal(M) = 1$. A \textit{chamber} is an open connected component of 
\[\R^n \setminus \overline{\bigcup_{\substack{M \in \mods A \\ 0 \neq M}} \Dcal(M)}.\]
\end{defn}

The combination of all stability spaces $\Dcal(M)$ for nonzero modules $M$ and the corresponding chambers they define is called the wall-and-chamber structure. A module $B \in \mods A$ is called a \textit{brick} if $\End_A(B)$ is a division ring, denote the collection of bricks of an algebra by $\brick A$. By \cite[Prop. 2.7]{Asai2019WS} and \cite[Prop. 12]{KaipelTreffinger} it suffices to consider stability spaces of bricks to obtain the wall-and-chamber structure. We refer to \cite{KaipelTreffinger} for a survey on the connection between King's stability conditions and $\tau$-tilting theory. \\

Closely related to $\tau$-rigid pairs are \textit{torsion pairs} \cite{Dickson66}. A torsion pair is a pair of full subcategories $(\Tcal, \Fcal)$ of $\mods A$ such that $\Hom_A(T,F)=0$ for all $T \in \Tcal$ and $F \in \Fcal$ and such that for each $M \in \mods A$, there exists a unique short exact sequence with submodule in $\Tcal$ and quotient in $\Fcal$. The subcategory $\Tcal$ is called the \textit{torsion class} and the corresponding $\Fcal$ the \textit{torsion-free class}. The study of stability conditions tells us that we can associate two torsion pairs to every element of $\R^n$.

\begin{prop} \cite[Prop. 3.1]{BKT2011}\label{prop:BKTtorsion}
Let $v \in \R^n$, denote the small and the large torsion class associated to $v$ by
\begin{enumerate}
    \item $\Tcal_v = \{ 0 \} \cup \{ Y \in \mods A : \langle v, \bdim Z \rangle > 0 \text{ for quotients } Y \to Z \to 0\}$, and
    \item $\overline{\Tcal}_v = \{ Y \in \mods A : \langle v, \bdim Z \rangle \geq 0 \text{ for quotients } Y \to Z \to 0\}$.
\end{enumerate}
The corresponding torsion-free classes are given, respectively, by
\begin{enumerate}
    \item $\overline{\Fcal}_v = \{ Y \in \mods A : \langle v, \bdim X \rangle \leq 0 \text{ for submodules } 0 \to X \to Y \}$, and
    \item $\Fcal_v = \{0 \} \cup \{ Y \in \mods A : \langle v, \bdim X \rangle < 0 \text{ for submodules } 0 \to X \to Y \}$.
\end{enumerate}
\end{prop}

Asai defines an equivalence relation such that $v_1, v_2 \in \R^n$ are \textit{TF-equivalent} whenever $\overline{\Tcal}_{v_1} = \overline{\Tcal}_{v_2}$ and $\overline{\Fcal}_{v_1} = \overline{\Fcal}_{v_2}$. In this case we write $\overline{\Tcal}_E$ and $\overline{\Fcal}_E$ to mean the large torsion and torsion-free class of any vector in the TF-equivalence class $E$. 

\begin{thm}\cite[Prop. 3.15]{BST2019} \cite[Thm. 3.17]{Asai2019WS}\label{prop:maximalchambers}
    Let $(M,P)$ be a $\tau$-rigid pair then the geometric interior of the polyhedral cone $\overline{\Ccal}_{(T,P)}$, denoted by $\Ccal_{(M,P)}$, is a TF-equivalence class. 
\end{thm}

It is well-known that each $\tau$-rigid module $M$ defines two torsion classes $\Fac M$ and ${}^\perp \tau M$ \cite{AIR2014}. The work of \cite[Thm. 3.27]{BST2019}, \cite[Prop. 3.3]{Yurikusa2018} and \cite[Thm. 3.11]{Asai2019WS} implies that these coincide with the torsion classes of \cref{prop:BKTtorsion} in the interior of the corresponding cone of $g$-vectors. The following result is essential in obtaining the fan structure of $\Sigma(A)$.

\begin{thm} \label{lem:facminimal} \cite[Thm. 3.11]{Asai2019WS}
Let $(M,P)$ be a $\tau$-rigid pair and $\nu$ be the Nakayama functor. Then
\begin{enumerate}
    \item $(\Tcal_v, \overline{\Fcal}_v) = (\Fac M, M^\perp)$ for all $v \in \Ccal_{(M,P)}$, and
    \item $(\overline{\Tcal}_v, \Fcal_v) = ({}^\perp (\tau M) \cap P^{\perp}, \Sub(\tau M \oplus \nu P))$ for all $v \in \Ccal_{(M,P)}$.
\end{enumerate}
\end{thm}
We remark that the torsion pairs coincide when $(M,P)$ is a $\tau$-tilting pair \cite[Prop. 2.16]{AIR2014} and conclude this preliminary section by giving another key result for our proof of the fan structure. Approximations were introduced in \cite{AuslanderSmaloApprox1980}.

\begin{lem} \cite[Lem. 2.6]{AIR2014} \label{AIRLemma2.6}
Let $0 \to Y \to M' \xrightarrow{f} X$ be an exact sequence in $\mods A$, where $M$ is $\tau$-rigid and $f: M' \to X$ is a right ($\add M$)-approximation. Then $Y \in {}^\perp (\tau M)$.
\end{lem}

The following is a well-known construction of such an approximation.

\begin{lem} \label{lem:addTapprox}
Let $M, N \in \mods A$ and let $\{ f_1, \dots, f_p \}$ be a collection of generators of $\Hom_A(M,N)$. Then $f: M^p \to N$ given by $f \begin{pmatrix} t_1 & \dots & t_p \end{pmatrix}^T = f_1(t_1) + \dots + f_p (t_p)$ is a right $(\add M)$-approximation. 
\end{lem}
\begin{proof} Clearly $M^p \in \add M$. We need to show that $\Hom_A(M',M^p) \to \Hom_A(M',N)$ is a surjection for all $M' \in \add M$. Since $\{f_1, \dots, f_p \}$ is a generating set for $\Hom_A(M, M')$ any map $g: M' \to N$ may be written as $g = f_1 a_1 + \dots + f_p a_p$ for some $a_1, \dots, a_p \in A$. Then we may construct a map $h: M' \to M^p$ given as $ h(t') = \begin{pmatrix} t' a_1 & \dots & t' a_p \end{pmatrix}^T$ which satisfies $g = f \circ h$. Hence $f$ is a right $(\add M)$-approximation. 
\end{proof}

\subsection{The fan structure of $g$-vectors}\label{sec:fanstructure}
Define the \textit{$g$-fan} of an algebra $A$ to be 
\[ \Sigma(A) \coloneqq \{ \overline{\Ccal}_{(M,P)} \subseteq \R^n: (M,P) \in \strigid A\}. \]
Faces of cones are cones of direct summands. We need to prove that two cones intersect at a face. The following technical lemma will be the key to obtain the fan structure.

\begin{lem}\label{lem:mainlem}
    Let $(M,P), (M',P') \in \strigid A$ be basic and such that $\Ccal_{(M,P)} \cap \Ccal_{(M',P')} \neq \emptyset$. Then $M \cong M'$ and $P \cong P'$.
\end{lem}
\begin{proof}
    First we show that $M \cong M'$. Take $v \in \Ccal_{(M,P)} \cap \Ccal_{(M',P')}$, then \cref{lem:facminimal} implies the identifications
    \[ \Fac M = \Tcal_v = \Fac M', \quad \text{and} \quad {}^\perp(\tau M) \cap P^\perp = \overline{\Tcal}_v = {}^\perp (\tau M') \cap (P')^\perp.\]
    Therefore $M, M' \in \Fac M = \Fac M'$. Construct a right $(\add M)$-approximation $f: M^p \to M'$ like in \cref{lem:addTapprox}, then $\ker f \in {}^\perp (\tau M)$ by \cref{AIRLemma2.6}. Moreover, $\ker f \in P^\perp$, since a nonzero morphism from $P$ may be composed with the natural inclusion to give a nonzero morphism from $P$ to $M$, contradicting the assumption that $(M,P)$ is $\tau$-rigid. It follows that $\ker f \in \overline{\Tcal}_v$ and hence $\Hom_A(\ker f, \tau M') = 0$. \cref{lem:homtoext} implies $\Ext_A^1(M', \ker f) = 0$ and hence the right $(\add M)$-approximation splits. Thus $M^p \cong M' \oplus N$ for some $N \in \mods A$. Since $M'$ is a basic direct summand of $M^p$ it has at most $|M|$ indecomposable direct summands. Thus $|M|\geq |M'|$ and reversing the the argument, it follows that $|M|= |M'|$. Then $M \cong M'$, because a basic direct summand of $M^p$ with exactly $|M|$ indecomposable direct summands has to be isomorphic to $M$. The argument to show $P \cong P'$ is similar. The previous argument and \cref{lem:facminimal} imply 
    \[ M^\perp = \overline{\Fcal}_v = (M')^\perp \quad \text{and} \quad \Sub(\tau M \oplus \nu P) = \Fcal_v = \Sub(\tau M \oplus \nu P').\]
    We observe that the module $\tau^{-1}(\tau M \oplus \nu P) \cong \tau^{–1} \tau M$ as $\tau$ is additive and $\nu P$ injective and moreover $\Sub(\tau M \oplus \nu P) \subseteq M^\perp$, therefore $\tau M \oplus \nu P$ is $\tau^{–1}$-rigid. Construct a left $\add(\tau M \oplus \nu P)$-approximation $g$ dually to \cref{lem:addTapprox}. Then the dual of \cref{AIRLemma2.6} implies that $\coker g \in (\tau^{-1} ( \tau M \oplus \nu P'))^\perp$. Hence by \cref{lem:homtoext} the approximation $g$ splits. It follows that $\tau M \oplus \nu P'$ is isomorphic to a basic direct summand of $(\tau M \oplus \nu P)^r$. Moreover, $\nu P'$ cannot be isomorphic to a a direct summand of $(\tau T)^r$ and therefore $\nu P' \cong (\nu P)^r$ and thus $P \cong P'$.
\end{proof}

We make the following immediate observation about cones of maximal basic direct summands.

\begin{lem} \label{lem:largestbasicsummand}
Let $(M,P)$ be a $\tau$-rigid pair, and $(N,Q)$ be a basic $\tau$-rigid pair such that $\add N = \add M$ and $\add Q = \add P$. Then $\overline{\Ccal}_{(M,P)} = \overline{\Ccal}_{(N,Q)}$ and $\Ccal_{(M,P)} = \Ccal_{(N,Q)}$.
\end{lem}

We are now able to give a new proof of \cite[Thm. 1.9]{DIJ2019}, establishing the fan structure of cones of $g$-vectors.

\begin{thm}\label{thm:gfan}
    Let $(M,P)$ and $(M',P')$ be $\tau$-rigid pairs such that $(N,Q)$ is their maximal common direct summand, then $\overline{\Ccal}_{(M,P)} \cap \overline{\Ccal}_{(M',P')} = \overline{\Ccal}_{(N,Q)}$.
\end{thm}
\begin{proof}
    By \cref{lem:largestbasicsummand} it suffices to restrict to basic $\tau$-rigid pairs. Trivially, $\overline{\Ccal}_{(N,Q)} \subseteq \overline{\Ccal}_{(M,P)} \cap \overline{\Ccal}_{(M',P')}$. Conversely, if $(M,P)$ is different from $(M',P')$ then \cref{lem:mainlem} implies that the interiors of the cones do not intersect. Hence consider an intersection at the boundary, which given by
    \[ \left\{ v \in \R^n: v = \sum_{i=1}^k \alpha_i g^{M_i} - \sum_{j=k+1}^t \alpha_j g^{P_j} = \sum_{i=1}^\ell \beta_i g^{M_i'} - \sum_{j=\ell +1}^s \beta_j g^{P_j'} \right\},\]
    where $\alpha_i, \beta_i \geq 0$ but not all strictly positive. Let $\boldsymbol{\alpha}_v \coloneqq \{i : \alpha_i = 0\}$ and $\boldsymbol{\beta}_v \coloneqq \{ i: \beta_i = 0\}$ be the indices for which the coefficients are zero. Consider the direct summands of $(M,P)$ and $(M',P')$ given only by those direct summands in the intersection
    \[ \left( \bigoplus_{i \not \in \boldsymbol{\alpha}_v} M_i, \bigoplus_{j \not \in \boldsymbol{\alpha}_v} P_j \right), \quad \text{and} \quad \left( \bigoplus_{i \not \in \boldsymbol{\beta}_v} M_i', \bigoplus_{j \not \in \boldsymbol{\beta}_v} P_j' \right). \]
    By construction, the interior cones of these two $\tau$-rigid pairs intersects non-trivially at $v$. Then \cref{lem:mainlem} implies that they are isomorphic and a thus a direct summand of $(N,Q)$. It follows that $v \in \overline{\Ccal}_{(M,P)} \cap \overline{\Ccal}_{(M',P')} \subseteq \overline{\Ccal}_{(N,Q)}$. 
\end{proof}

The $g$-vectors spanning a $g$-vector cone $\overline{\Ccal}_{(M,P}$ are linearly independent by \cite[Thm. 5.1]{AIR2014}. Therefore, the faces of $\overline{\Ccal}_{(M,P)}$ are given by $\overline{\Ccal}_{(N,Q)}$ where $N \in \add M$ and $Q \in \add P$. It is clear that if $(M,P)$ is a $\tau$-rigid pair then $(N,Q)$ is $\tau$-rigid pair, hence $\overline{\Ccal}_{(N,Q)}$ is also a cone of $\Sigma(A)$. Together with \cref{thm:gfan} this proves that $\Sigma(A)$ is a polyhedral fan. For a summary of known properties of the $g$-fan we refer to \cite{AHIKM2022}. We want to highlight the following result. 

\begin{thm} \label{defn:taufinite}\cite[Thm. 4.7]{Asai2019WS}\cite[Thm. 5.4]{DIJ2019}
    The algebra $A$ has finitely many $\tau$-tilting modules if and only if the $g$-fan $\Sigma(A)$ is finite and complete. In this case, we call $A$ \textit{$\tau$-tilting finite}.
\end{thm}

In this case, there is a bijection between maximal cones $\overline{\Ccal}_{(T,P)} \in \Sigma(A)$ and torsion classes $\Fac T \in \tors A$ by \cite[Thm. 2.7]{AIR2014}. In this way the inclusion of torsion classes induces a partial order on maximal cones. So, for a maximal cone $\sigma \in \Sigma(A)$, we write $\Tcal_\sigma$ to mean $\Tcal_v$ where $v$ lies in the interior of $\sigma$. This defines a partial order $(\Sigma(A),\leq) = (\Sigma(A),\leq_{\tors A})$ on maximal cones satisfying $\sigma_1 \leq \sigma_2$ if and only if $\Tcal_{\sigma_1} \subseteq \Tcal_{\sigma_2}$. 

\begin{prop} \label{prop:fanposet}
    Let $A$ be a $\tau$-tilting finite algebra, then the poset of torsion classes $\tors A$ induces a fan poset $(\Sigma, \leq_{\tors A})$ on $\Sigma(A)$.
\end{prop}
\begin{proof}
    By definition, every cone $\sigma \in \Sigma(A)$ is the $g$-vector cone of a $\tau$-rigid pair $(M,P)$. Then the maximal cones containing it are exactly those corresponding to the interval $[\Fac M, {}^\perp \tau M \cap P^\perp] \subseteq \tors A$. This is because maximal $g$-vector cones containing $\overline{\Ccal}_{(M,P)}$ are exactly those $\overline{\Ccal}_{(N,Q)} \in \Sigma(A)$ such that $(N,Q)$ is a $\tau$-tilting pair with $M \in \add N$ and $P \in \add Q$. By \cite[Thm. 2.10]{AIR2014}, see also \cite[Thm. 4.4]{DIRRT2017} and the bijection of $\tau$-tilting pairs with torsion classes \cite[Thm. 2.7]{AIR2014} when $A$ is $\tau$-tilting finite we get the desired interval. \\
    
    We are left with showing that each interval is a cone. For a torsion class $\Tcal \in \tors A$ and a torsion-free class $\Fcal \in \torf A$, respectively, define the following subspaces of $\R^n$:
    \[ \Hcal_{\Tcal}^+ \coloneqq \bigcap_{T \in \Tcal}\{ v \in \R^n: \langle v, \bdim T \rangle \geq 0\}, \quad \Hcal_{\Fcal}^- \coloneqq \bigcap_{F \in \Fcal} \{ v \in \R^n: \langle v, \bdim F \rangle \leq 0 \}.\]
    
    Both spaces are intersection of half-spaces and thus a convex cone. Take an interval $I = [\sigma_1, \sigma_2] \in (\Sigma(A), \tors A)$ corresponding with torsion classes $[\Tcal_{\sigma_1}, \Tcal_{\sigma_2}] \subseteq \tors A$, define $ \Tcal_i \coloneqq \Tcal_{\sigma_i}$. To prove that $(\Sigma(A), \leq_{\tors A})$ is a fan poset we need to show 
    \[ \bigcup_{\sigma_i \in I} \sigma_i = \Hcal_{\Tcal_1}^+ \cap \Hcal_{\Tcal_2^\perp}^-.\]

    Take $v \in \bigcup_{\sigma_i \in I} \sigma_i$. From \cite[Lem. 3.12]{BST2019} it follows that $\langle v, \bdim T \rangle \geq 0$ for all $T \in \Tcal_{\sigma_1}$ since $\Tcal_1 \subseteq \Tcal_i$ for all $\sigma_i \in I$. Therefore $v \in \Hcal_{\Tcal_1}^+$. Dually, it follows for the torsion-free class $(\Tcal_2)^\perp$ that $\langle v, \bdim F \rangle \leq 0$ for all $F \in (\Tcal_2)^\perp$. Hence $\bigcup_{\sigma_i \in I} \sigma_i \subseteq \Hcal_{\Tcal_1}^+ \cap \Hcal_{(\Tcal_2)^\perp}^-$. \\

    Conversely, assume for a contradiction that there exists $v \in \Hcal_{\Tcal_1}^+ \cap \Hcal_{(\Tcal_2)^\perp}^-$ such that there is no $\sigma_i \in I$ satisfying $v \in \sigma_i$. By completeness of the $g$-fan, there exists some other maximal cone $\sigma'$ with $\sigma' \not \in I$ containing $v$ in its interior. This interior has to be fully contained in $\Hcal_{\Tcal_1}^+ \cap \Hcal_{(\Tcal_2)^\perp}^-$. By construction of $\Hcal_{\Tcal_1}^+$, $\Tcal_w$ contains $\Tcal_1$. Therefore $\Tcal_1 \subseteq \Tcal_{\sigma'}$. Dually, it follows that $(\Tcal_{\sigma'})^\perp \subseteq (\Tcal_2)^\perp$ and hence $\sigma' \in I$. 
\end{proof}

Therefore any properties of the poset $\tors A$ apply to the fan poset $(\Sigma(A), \leq_{\tors A})$, in particular it is a polygonal completely semi-distributive lattice \cite{DIRRT2017}. We conclude the study of the fan structure by showing that $g$-vectors determine $\tau$-rigid pairs, giving a new proof of \cite[Thm. 6.5]{DIJ2019}.

\begin{thm}
    Let $(M,P)$ and $(M',P')$ be $\tau$-rigid pairs such that $g \coloneqq g^{(M,P)}= g^{(M',P')}$.  Then $M \cong M'$ and $P \cong P'$. 
\end{thm}
\begin{proof}
    We repeatedly reduce the problem to the largest basic direct summands and show that they coincide. Write $M \cong B_1^M \oplus M_1$, where $B_1^M$ is a basic direct summand of $M$ such that $\add B_1^M = \add M$. Similarly, write $M' \cong B_1^{M'} \oplus M_1'$, such that $B_1^{M'}$ is basic and $\add B_1^{M'} = \add M'$. In the same way, write $P \cong Q_1^P \oplus P_1$ and $P' \cong Q_1^{P'} \oplus P_1'$ with $Q_1^P$ and $Q_1^{P'}$ largest basic direct summands. \cref{lem:largestbasicsummand} implies $\Ccal_{(B_1^M, Q_1^P)}= \Ccal_{(M,P)}$ and $\Ccal_{(B_1^{M'}, Q_1^{P'})} = \Ccal_{(M',P')}$. Since $g$ is contained in the intersection of these cones, \cref{lem:mainlem} implies that $B_1^M \cong B_1^{M'}$ and $Q_1^P \cong Q_1^{P'}$. We repeat the process with $g_1 \coloneqq g - g^{(B_1^M, Q_1^P)}$ and the $\tau$-rigid pairs $(M_1, P_1)$ and $(M_1', P_1')$. Since the modules are finitely generated, eventually $M \cong M'$ and $P \cong P'$. 
\end{proof}

\subsection{$\tau$-cluster morphism category}\label{subsec:tauclustermorph} The definition of the category of a partitioned fan, \cref{defn:thecategory}, is inspired by the geometric construction of the $\tau$-cluster morphism category from the $g$-fan \cite{STTW2023}. In the original constructions (cf. \cite{BuanHanson2017}, \cite{BuanMarsh2018}, \cite{IgusaTodorov2017}) the objects of the category are \textit{wide subcategories} of the form $\Wcal_E \coloneqq \overline{\Tcal}_E \cap \overline{\Fcal}_E$ for some TF-equivalence class $E= \Ccal_{(M,P)}$ coming from a $\tau$-rigid pair $(M,P)$. The relationship between TF-equivalence classes and the $g$-fan is described in \cref{prop:maximalchambers}. Let $\TF_A$ be the poset category of TF-equivalence classes coming from $\tau$-rigid pairs with poset relation $E_1 \leq E_2$ whenever $E_1 \subseteq \overline{E_2}$. 

\begin{defn}\cite[Defn. 3.3]{STTW2023} \label{defn:classical}
    Let $A$ be a finite-dimensional algebra. The \textit{$\tau$-cluster morphism category} $\Cfrak(A)$ is the category
    \begin{itemize}
        \item whose objects are objects of $\TF_A$ under the identification $E_1 \sim E_2$ whenever $\Wcal_{E_1} = \Wcal_{E_2}$;
        \item whose morphisms $\Hom_{\Cfrak(A)}([E], [F])$ are given by the set of morphisms
        \[ \bigcup_{E' \in [E], F' \in [F]} \Hom_{\TF_A}(E',F')\]
        under the identification $f_{EF} \sim f_{E'F'}$ whenever $\pi_E(F) = \pi_{E'}(F')$.
    \end{itemize}
\end{defn}

Translating from TF-equivalence classes to cones $\Ccal_{(M,P)}$, we see that the objects of the $\tau$-cluster morphism category are interior cones of cones in $\Sigma(A)$ with identification given by $\Ccal_{(M_1,P_1)} \sim \Ccal_{(M_2,P_2)}$ whenever $M_1^\perp \cap {}^\perp \tau M_1 \cap P_1^\perp = M_2^\perp \cap {}^\perp \tau M_2 \cap P_2^\perp$. By taking the closure, this induces a partition of the fan $\Sigma(A)$, which we denote by $\Pfrak_{\textnormal{WAC}}$. By \cite[Cor. 3.7, Lem. 3.8]{STTW2023}, $\Pfrak_{\textnormal{WAC}}$ is an admissible partition, so that the $\tau$-cluster morphism category $\Cfrak(A)$ is easily seen to be equivalent to the category of the partition fan $\Cfrak(\Sigma(A), \Pfrak_{\textnormal{WAC}})$. It follows from the definition of $\tau$-rigid pairs and their correspondence with the cones of $\Sigma(A)$ that the first factors of $\Cfrak(\Sigma(A), \Pfrak_{\textnormal{WAC}})$ are given by pairwise compatibility conditions, see \cite{HansonIgusa2021, IgusaTodorov2017, IgusaTodorov22}. \\

Following \cite{AsaiSB2020,BarnardCarrolZhu19,DIRRT2017}, we label the edges of $\Hasse(\tors A)$ in the following way: Let $\Tcal_1 \lessdot \Tcal_2$ be a cover relation in $\tors A$, then $\Tcal_1^\perp \cap \Tcal_2 = \Filt\{B\}$ is the category of modules filtered by a brick $B$ by \cite[Thm. 3.3]{DIRRT2017}. We call the labelling of arrows $\Tcal_1 \xrightarrow{B} \Tcal_2$ in $\Hasse(\tors A)$ in this way the \textit{brick-labelling}. Originally, the picture group of an algebra was defined in \cite{IgusaTodorovWeyman2016} for representation-finite hereditary algebras, and has since been extended to all $\tau$-tilting finite algebras in \cite{HansonIgusa2021} as follows.
\begin{defn} \cite[Prop. 4.4a]{HansonIgusa2021}\label{defn:tradpicgroup}
    Let $A$ be $\tau$-tilting finite. The \textit{picture group} $G(A)$ is generated by $\{ X_S : S \in \brick A\}$ with a relation
    \[ X_{S_1} \dots X_{S_k}= X_{S_1'} \dots X_{S_l'}\]
    if there exist torsion classes $\Tcal, \Tcal' \in \tors A$ such that $(S_1, \dots, S_k)$ and $(S_1', \dots, S_l')$ are sequences of bricks labelling directed paths $\Tcal \to \Tcal'$ in $\Hasse(\tors A)$, denote this element by $X_{[\Tcal, \Tcal']}$.
\end{defn}

\begin{prop} \label{thm:picgroupgeneralisation}
    Let $A$ be $\tau$-tilting finite. Let $\Cfrak(A) = \Cfrak(\Sigma, \Pfrak_{\textnormal{WAC}})$, then $G(\Sigma(A), \Pfrak_{\textnormal{WAC}}, \leq_{\tors A})$ is isomorphic to $G(A)$, where $(\Sigma, \leq_{\tors A})$ is the fan poset induced by $\tors A$.
\end{prop}
\begin{proof}
    First we show that equivalence classes of $\Pfrak_{\textnormal{WAC}}$ are determined by bricks. Let $\overline{\Ccal}_{(M,P)} \in \Sigma(A)$ be of codimension 1, then by \cite[Prop. 3.17]{BST2019} $\overline{\Ccal}_{(M,P)} \subseteq \Dcal(N)$, where $N$ is $v$-stable for all $v \in \Ccal_{(M,P)}$ and a brick by \cite[Thm. 1]{Rudakov1997}. Take now a distinct cone $\overline{\Ccal}_{(M',P')} \sim \overline{\Ccal}_{(M,P)} \in \Cfrak(\Sigma, \Pfrak_{\textnormal{WAC}})$. Then \cite[Prop. 3.13, Thm. 3.14]{BST2019} implies that $\overline{\Ccal}_{(M',P')} \subseteq \Dcal(N)$ and that this brick is unique in $M^\perp \cap {}^\perp \tau M \cap P^\perp = (M')^\perp \cap {}^\perp \tau M' \cap (P')^\perp$. Hence the generators are in bijection. It follows from \cite[Prop. 4.5, Rmk. 4.6]{BST2019} and the bijection between $\tau$-tilting pairs and torsion classes \cite[Thm. 2.7]{AIR2014} in the $\tau$-tilting finite case that the set of relations of $G(A)$ coincide under the bijection of generators with the set of type 1 relations for $G(\Sigma(A), \Pfrak_{\textnormal{WAC}}, \leq_{\tors A})$. That $G(A)$ satisfies the type 2 relations of $G(\Sigma(A), \Pfrak_{\textnormal{WAC}}, \leq_{\tors A})$ follows from \cite[Thm. 4.12, Prop. 4.13]{DIRRT2017} which establishes a label-preserving isomorphism between the dual graphs of the stars of cones whose corresponding wide subcategories coincide and that of their shared projection.
\end{proof}

Therefore \cref{defn:picgroup} is a generalisation of  \cref{defn:tradpicgroup}. We also remark that by \cite[Thm. 4.10]{HansonIgusa2021}, the picture group $G(A)$ is isomorphic to the fundamental group of $\Bcal \Cfrak(A)$. We conclude by applying the theory developed in the previous sections to obtain a large class of algebras admitting a faithful functor.

\begin{thm} \label{thm:algmainthm}
    Let $A$ be such that $\Sigma(A)$ is a finite hyperplane arrangement. Then $\Cfrak(A)$ admits a faithful functor to $G(\Sigma(A), \Pfrak_{\flatt}, \leq_{\tors_A})$, and hence to $G(A)$.
\end{thm}
\begin{proof}
    For any finite-dimensional algebra, cones of $\Sigma(A)$ are contained in stability spaces $\Dcal(B)$ for some brick $B \in \brick A$. In particular, $\Dcal(B)$ is orthogonal to the dimension vector $\bdim B \in \Z_{\geq 0}^n$, hence $\Sigma(A)$ satisfies the assumptions of \cref{thm:mainthm} and since the $\tau$-cluster morphism category $\Cfrak(A)$ corresponds to some admissible partition, \cref{cor:HAallfaithful} gives the desired result. Since $G(A)$ is the fundamental group of $\Bcal \Cfrak(A)$ by \cite[Thm. 4.10]{HansonIgusa2021}, the faithful functor to $G(\Sigma(A), \Pfrak_{\flatt}, \leq_{\tors A})$ factors through a faithful functor to $G(A)$, see the proof of \cite[Prop. 3.7]{Igusa2022}.
\end{proof}

\begin{cor}
    Let $A$ be such that $\Sigma(A)$ is a finite hyperplane arrangement. Then the classifying space $\Bcal \Cfrak(A)$ is a $K(\pi,1)$ space if $\Cfrak(A)$ satisfies the pairwise compatibility condition of last factors. In particular, if $\Sigma(A) \subseteq \R^3$, then $\Bcal \Cfrak(A)$ is a $K(\pi,1)$ space.
\end{cor}
\begin{proof}
    The first part is the remaining sufficient condition of \cref{prop:igusacriteria}, and the second result follows from \cite{BarnardHanson2022}, where it is shown that the pairwise compatibility always holds in rank 3.
\end{proof}

The existence of a faithful group functor had previously been proven for $K$-stone algebras i.e. algebras $A$ where every brick $B \in \brick A$ satisfies $\Ext_A^1(B,B)= 0$ and $\End_A(B)=k$, see \cite{HansonIgusaPW2SMC,HansonIgusa2021} and hereditary algebras of finite or tame type \cite{IgusaTodorov2017, IgusaTodorov22}. The class of $K$-stone algebras contains Nakayama(-like) algebras \cite{HansonIgusaPW2SMC,HansonIgusa2021}, preprojective algebras of Dynkin type ADE \cite{IRRT} and gentle algebras with no loops or 2-cycles \cite{HansonIgusaPW2SMC}. As a consequence of \cref{thm:algmainthm} we obtain faithful functors for generalised preprojective algebras coming from Cartan matrices of finite (Dynkin) type as introduced in \cite{GLS2017}. For these algebras \cite[Thm. 3.19]{Murakami2022} states that their $g$-vector fan defines a finite hyperplane arrangement which comes from the root system of the corresponding Weyl group. Conjecturally, the finite hyperplane arrangements coming from crepant resolutions as in \cite{Wemyss2018} define contraction algebras whose $g$-vector fan would then be the finite hyperplane arrangement by \cite{August2020a}.\\

For Nakayama(-like) algebras, the pairwise compatibility property of last factors is also satisfied in dimension 4 and higher \cite{HansonIgusaPW2SMC,HansonIgusa2021}. However, this is not the case for preprojective algebras of Dynkin type ADE \cite{BarnardHanson2022}, whose $g$-vector fan is a finite hyperplane arrangement \cite{Mizuno2013} and certain gentle algebras \cite{HansonIgusaPW2SMC}. However, it is unknown whether $\Bcal \Cfrak(A)$ is a $K(\pi,1)$ space for preprojective algebras and whether the generalised preprojective algebras satisfy the pairwise compatibility property of last factors. We conclude by giving a new example of an algebra $A$ for which $\Bcal \Cfrak(A)$ is a $K(\pi,1)$ space.

\begin{exmp} 
    Let $A$ be the generalised preprojective algebra of type $C_3$, see \cite[Sec. 13.8]{GLS2017}. That is, the algebra coming from the Cartan matrix $C$ with symmetrizer $D$ given as follows:
    \[ C = \begin{pmatrix} 2 & -1 & 0 \\ -1 & 2 & -2 \\ 0 & -1 & 2\end{pmatrix}, \quad D = \diag(1,1,2). \]
    We may write $A \cong KQ/I$ for the quiver
    \[ Q : \begin{tikzcd}[column sep=4em] 1 \arrow[r, shift left, "\alpha_{21}"] & 2 \arrow[l, shift left, "\alpha_{12}"] \arrow[r, shift left, "\alpha_{32}"] & 3 \arrow[l, shift left, "\alpha_{23}"] \arrow[loop right,"\epsilon_3",swap, looseness=10, in=40, out=-40]\end{tikzcd}\]
    and $I = \langle \alpha_{21} \alpha_{12}, \alpha_{32} \alpha_{23}- \alpha_{12} \alpha_{21}, \epsilon_3 \alpha_{23} \alpha_{32} + \alpha_{23} \alpha_{32} \epsilon_3 \rangle$. Similar to \cite[Thm. 3.9]{Mizuno2013} for preprojective algebras of type ADE, \cite[Thm. 3.19]{Murakami2022} states that up to a base-change the $g$-vector cones of generalised preprojective algebras of arbitrary Dynkin type of \cite{GLS2017} coincide with the (Weyl) chambers in the hyperplane arrangement obtained by taking orthogonal hyperplanes to the roots of the corresponding root system. Moreover, by \cite[Thm. 1.3]{GLS2017} these roots are a positive scalar multiple of the dimension vectors of certain ($\tau$-locally free) modules. For our example of type $C_3$, we obtain from \cite[Sec. 13.8]{GLS2017} that the dimension vectors of these modules are the following 9 integer vectors:
     \[ \begin{pmatrix} 1 \\ 0 \\ 0 \end{pmatrix}, \begin{pmatrix} 0 \\ 1 \\ 0 \end{pmatrix}, \begin{pmatrix} 0 \\ 0 \\ 1 \end{pmatrix}, \begin{pmatrix} 1 \\ 1 \\ 0 \end{pmatrix}, \begin{pmatrix} 0 \\ 1 \\ 1 \end{pmatrix}, \begin{pmatrix} 1 \\ 1 \\ 1 \end{pmatrix}, \begin{pmatrix} 0 \\ 1 \\ 2 \end{pmatrix},
     \begin{pmatrix} 1 \\ 1 \\ 2 \end{pmatrix}, \begin{pmatrix} 1 \\2 \\ 2\end{pmatrix}.\]
    However, one can easily see that there the following polygon arises in $\Hasse(\tors A)$, whose edges are labelled by the composition series of the respective modules:
    \[
    \begin{tikzcd}[ampersand replacement=\&]
    \&\bullet \arrow[ld, "{\tiny\begin{array}{c} 2 \end{array}}",swap] \arrow[rd, "{\tiny\begin{array}{c} 3 \end{array}}"]\\
    \bullet \arrow[d, "{\tiny\begin{array}{c} 3 \\2 \end{array}}",swap] \& \& \bullet \arrow[d, "{\tiny\begin{array}{c} 2 \\3 \\ 3 \end{array}}"] \\
    \bullet \arrow[d,"{\tiny\begin{array}{c} 3 \\3 \\2 \end{array}}",swap] \& \& \bullet \arrow[d,"{\tiny\begin{array}{c} 2 \\ 3 \end{array}}"] \\
    \bullet \arrow[rd, "{\tiny\begin{array}{c} 3\end{array}}",swap] \& \& \bullet \arrow[ld, "{\tiny\begin{array}{c} 2\end{array}}"] \\
    \&\emptyset
    \end{tikzcd}
    \]
    Therefore, from the description given in \cite[Prop. 4.33]{DIRRT2017}, it follows that $A$ is not a $K$-stone algebra, since the polygons which may arise in the Hasse quiver of posets of torsion classes of $K$-stone algebras do not include the above. Since $\Sigma(A)$ is a finite hyperplane arrangement, \cref{thm:algmainthm} implies that there exists a faithful functor to some group, and from \cite{BarnardHanson2022} we know that $\Cfrak(A)$ satisfies the pairwise compatibility of last factors since $A$ has 3 simples. Therefore $\Bcal \Cfrak(A)$ is a $K(G(A),1)$ space by \cref{prop:igusacriteria}.

\end{exmp}

\subsection*{Acknowledgements} The author is thankful to Sibylle Schroll and Hipolito Treffinger for their guidance and many meaningful conversations and suggestions. The author thanks Eric J. Hanson for pointing out a mistake in an earlier version of this paper. The author would also like to thank Erlend D. Børve, Edmund Heng, Lang Mou, Calvin Pfeifer, Aran Tattar and José Vivero for inspiring discussions. The author also thanks Bethany Marsh whose question was the starting point of this paper. Finally, the author extends his gratitude to an anonymous referee whose thorough review greatly helped to improve the paper.

\subsection*{Funding} 
MK is supported by the Deutsche Forschungsgemeinschaft (DFG, German Research Foundation) -- Project-ID 281071066 -- TRR 191.

\providecommand{\bysame}{\leavevmode\hbox to3em{\hrulefill}\thinspace}
\providecommand{\MR}{\relax\ifhmode\unskip\space\fi MR }
\providecommand{\MRhref}[2]{%
  \href{http://www.ams.org/mathscinet-getitem?mr=#1}{#2}
}
\providecommand{\href}[2]{#2}

\end{document}